\documentclass
[12pt]
{article}
\usepackage{amsmath}
\usepackage{amssymb}
\usepackage{amsthm}

\newcommand{\ga}{\gamma}
\newcommand{\e}{\varepsilon}
\newcommand{\la}{\lambda}

\newcommand{\sig}{\sigma}
\newtheorem{theorem}{Theorem}[section]
\newtheorem{lemma}[theorem]{Lemma}
\newtheorem{remark}[theorem]{Remark}
\newtheorem{proposition}[theorem]{Proposition}

\newtheorem{corollary}[theorem]{Corollary}

\numberwithin{equation}{section}
\usepackage{color}
\begin{document}
\title{
Concentration and oscillation analysis of positive solutions to semilinear elliptic equations with exponential growth in a disc
}
\date{}
\author{Daisuke Naimen\thanks{Muroran Institute of Technology, 27-1, Mizumoto-cho, Muroran-shi, Hokkaido, 0508585, Japan, naimen@muroran-it.ac.jp
}}
\maketitle
\begin{abstract}
We establish a series of concentration and oscillation estimates for elliptic equations with exponential nonlinearity $e^{u^p}$ in a disc. Especially, we show various new results on the supercritical case $p>2$ which are left open in the previous works. We begin with the concentration analysis of blow-up solutions by extending the scaling and pointwise techniques developed in the previous studies. A striking result is that we detect an infinite sequence of bubbles in the supercritical case $p>2$. The precise characterization of the limit profile, energy, and location of each bubble is given. Moreover, we arrive at a natural interpretation, the infinite sequence of bubbles causes the infinite oscillation of the solutions. Based on this idea and our concentration estimates,  we next carry out the oscillation analysis. The results allow us to estimate  intersection points and numbers between blow-up solutions and singular functions. Applying this, we finally demonstrate the infinite oscillations of the bifurcation diagrams of supercritical equations. In addition, we also discuss what happens on the sequences of bubbles in the limit cases $p\to 2^+$ and $p\to \infty$ respectively. As above, the present work discovers a direct path connecting the concentration and oscillation analyses. It leads to a consistent and straightforward understanding of concentration, oscillation, and bifurcation phenomena on blow-up solutions of supercritical problems.  
\end{abstract}
\newpage
\tableofcontents
\newpage
\section{Introduction}\label{sec:intr}
We study the following elliptic problem with exponential nonlinearity,
\begin{equation}\label{p0}
\begin{cases}
-\Delta u=\la h(u)e^{u^p},\ \  u>0\ \text{ in }\Omega,\\
u=0\ \ \ \ \ \ \ \ \ \ \ \ \ \ \ \ \ \ \ \ \ \ \ \ \  \ \text{ on }\partial \Omega,
\end{cases}
\end{equation}
where $\Omega\subset \mathbb{R}^2$ is a smooth bounded domain, $p>0$  a given constant, $\la>0$  a parameter, and $h:[0,\infty)\to \mathbb{R}$ a nonnegative smooth function with lower growth.  In our main argument, we assume $\Omega$ is a unit disc and investigate the concentration and oscillation  behavior of blow-solutions of \eqref{p0}.

For many years, \eqref{p0} focuses various attentions. In the case $p=1$ and $h=1$, it is known as the Gelfand problem \cite{G}. An interesting discussion appears in the study of  the bifurcation diagram of solutions $(\la,u)$ of \eqref{p0}. Assuming $\Omega$ is a ball in $\mathbb{R}^n$ with  $n\ge1$, one sees that the shape  of the  diagram drastically changes depending on $n$. This allows several uniqueness, multiplicity, and nonexistence results on \eqref{p0}. In fact, if $1\le n\le2$, the diagram, which emanates from $(0,0)$, tuns only once and goes to the point at infinity $(0,\infty)$. This implies the existence of two solutions for all small $\la$ and no solutions for all large $\la$.  If $3\le n\le 9$, one finds that  it has infinitely many turning points  which permit  the existence of infinitely many solutions $u$ for some value $\la=\la^*>0$ and many solutions near $\la=\la^*$. Moreover, if $n\ge 10$, the diagram has no turning point and yields the existence of a unique solution for all small $\la$ and no solution for all large $\la$. For more precise and complete statements, see \cite{JL}.  We also note that certain extensions to more general nonlinearities are carrying out  for  $n\ge3$ more recently. We refer the readers to  \cite{Mi2},  \cite{KW},  \cite{MN} and references therein. 
 
On the other hand, if  $p=2$,  the nonlinearity has the term $e^{u^2}$. This is the critical growth of the Sobolev embedding in dimension two  as is known by the Trudinger-Moser inequality   \cite{T} \cite{M}. Hence, we may regard \eqref{p0} as the critical problem in dimension two. In other words, it gives the two-dimensional counter part of the Brezis-Nirenberg problem in higher dimension \cite{BN}. Of particular interest is that \eqref{p0} has noncomact properties which  can be observed as the concentration of solutions or Palais-Smale sequences.  In this point of view, various interesting and challenging problems arise in the study of the existence and asymptotic behavior of solutions. We refer the readers to  the previous works \cite{A}, \cite{AD}, \cite{AS}, \cite{D}, \cite{DMR}, \cite{DMMT}, \cite{DT}, \cite{MM1}, and references therein. Lastly, we remark that there are several papers which  cover also the case $p\not=1,2$.  See,  for example, \cite{AtPe}, \cite{FMR}, \cite{DM}, and \cite{OS}.

In any case of $p>0$, in order to understand the structure of solutions of \eqref{p0}, it is important to investigate the asymptotic behavior of blow-up solutions. See the blow-up analyses in \cite{BM} and \cite{NS} for $p=1$, \cite{AD}, \cite{D}, \cite{DT} and references therein for $p=2$, and \cite{OS} for $1< p\le2$. In these works, one of the key ideas is to utilize the scaling property of \eqref{p0}. Indeed, after the suitable scaling around the maximum point, one sees that the blow-up solutions converge to a solution $z$ of the Liouville equation, 
\begin{equation}\label{limp}
-\Delta z=e^z\ \text{in }\mathbb{R}^2,\ \ \ \int_{\mathbb{R}^2}e^zdx<\infty.
\end{equation}
Moreover, one also finds that the limit value of  the ``energy" is characterized by $\int_{\mathbb{R}^2}e^zdx$.  Then, noting the classification result by \cite{CL}, one may expect that the asymptotic behavior of blow-up solutions can be precisely described by the explicit information of \eqref{limp}.  Actually, based on this idea, in \cite{D} and \cite{DT} for $p=2$, the authors prove that any energy-bounded sequence of blow-up solutions of \eqref{p0} (with some conditions on $h$) is decomposed by a finite number of ``bubbles" which are characterized by the classical solution of \eqref{limp}. This leads us to observe the ``energy quantization" of solutions. This quantification result plays an important role  in the variational analysis of the Trudinger-Moser energy functional. See the recent development in \cite{DMMT} along this direction. 

As noted above, we observe various interesting discussions on the case $p\le2$. On the other hand, there seems to have been few works for the supercritical case $p>2$ after the basic  existence result and asymptotic formulas were obtained by Atkinson-Peletier \cite{AtPe} in a disc. (One finds  a suggestive work \cite{McMc}  by McLeod-McLeod which will be noted later.)  Since the scaling structure of the equation does not change even for $p>2$, it is natural to ask what happens on supercritical blow-up solutions in view of  the scaling analysis. Starting from this question, our main aim in this paper is to attack the blow-up analysis of the supercritical problems of \eqref{p} and  discover and clarify any supercritical phenomena through the scaling approach. 

To accomplish our aim, we begin with applying the pointwise technique and radial analysis in \cite{D} and \cite{MM1}. In particular, we first detect the standard bubble around the maximum point. Then we next search the outside of the first concentration region for the next bubble to appear for general $p>0$. As a result, we figure out the striking difference between the cases $p\le2$ and $p>2$. A fundamental observation is that, in the subcritical case, no additional bubble appears while in the supercritical case $p>2$, an infinite sequence of bubbles does appear. Then, following this fact, in the supercritical case, we establish  the precise characterization of the limit profile, energy, and location of each bubble via the limit equation \eqref{limp}. We emphasize that they are completely described with the two sequences $(a_k)$ and $(\delta_k)$ of numbers  defined by a system of recurrence formulas \eqref{eq:del1} and \eqref{eq:del2} below.  These results will be summarized in Theorems \ref{thm1}, \ref{thm30}, and \ref{thm3}.   A symbolic consequence of our concentration estimates is that  infinitely many bubbles break the uniform boundedness of the ``energy" of solutions which is usually assumed or proved in the study of the critical case. See \eqref{sup2}. Moreover, we arrive at a key observation which connects two important supercritical phenomena on \eqref{p0}. 

That is, we observe that the infinite sequence of bubbles causes  the infinite oscillation of  solutions.   See  Theorem \ref{thm:osc} which is a direct consequence of our concentration estimates. Furthermore,  our oscillation estimates naturally lead us to study the points and numbers of intersections between blow-up solutions and singular functions. See the corollaries of the theorem. In particular, this enables us to estimate the intersection numbers between regular  and singular solutions of \eqref{p0}. Then, with the idea in \cite{Mi},  it allows us to study the oscillation of the bifurcation diagram of \eqref{p0}. Based on this idea, we will finally  demonstrate that for some nonlinearities, the intersection numbers between regular and singular solutions diverge to infinity and as a consequence, the corresponding bifurcation diagrams oscillate infinitely many times.  See Theorem \ref{thm:app2} and Remark \ref{rmk:B}.

As mentioned above, we successfully establish the precise estimates for the concentration, oscillation, and intersection phenomena, and at last, prove the infinite oscillations of the bifurcation diagrams of  \eqref{p0}. One of the remarkable novelties here is that we detect the infinite sequence of bubbles with complete characterization. In the previous studies for $p\le2$ referred above, one finds or constructs concentrating solutions with at most finitely many bubbles. Hence what we show here is nothing but a typical  supercritical concentration phenomenon.

 Another novelty is that we clarify  the  direct connection between bubbling and oscillation behaviors of blow-up solutions with precise quantification. It leads to our consistent  and straightforward analysis of concentration, oscillation, and  bifurcation phenomena on  blow-up solutions in dimension two.  This observation will deserve consideration also in the study of related problems.

The other novelty we remark here is that we provide a useful approach to prove the infinite oscillations of the bifurcation diagrams and further, we actually succeed in showing it for some supercritical equations. Here, we point out that the  study of supercritical problems with exponential nonlinearity $e^{u^p}$ or more general one is recently developed in, for instance, \cite{KW} and \cite{MN}. They assume that the dimension of the domain is greater than or equal to three and  leave the two-dimensional case open. As noted above, our approach is based on the direct analysis of the infinite sequence of bubbles which is very different from those in the papers. Thanks to this, we can accomplish  the precise and extensive analysis, including the proof of the infinite oscillations of the bifurcation diagrams, in dimension two.

 We here refer to an earlier work \cite{McMc} by McLeod-McLeod (which the author knew after the present work was completed). An oscillation property of solutions for the supercritical case is indicated as the ``bouncing process" in Sections 1 and 2 there. Actually, their observation coincides with what we prove in the present work. Their discussion is very interesting since they find out such a oscillation phenomenon and some effective formulas from a heuristic and qualitative argument.  On the other hand, we do it through a direct analysis of our concentration estimates which is completely different way. Moreover, our approach enables us to achieve a rigorous justification and precise quantification. Thanks to this, we can proceed to further analysis of intersection and bifurcation problems. We also note that it is interesting to carefully compare what we prove with what they observe. Interestingly, we encounter some correspondences  although the approaches or interpretations are different from each other. A surprising fact is that formulas which are  equivalent to \eqref{eq:del1} and \eqref{eq:del2}  have already been observed there. Their derivation is still based on a heuristic observation while we deduce them via a precise and different way  as noted above.  We will give some more discussions about the relation between our present work and \cite{McMc} in  Remark \ref{rmk:mm} below.

Finally, we remark on two resent works \cite{FIRT} and \cite{Ku} on \eqref{p0} with general nonlinearity which are carried out  independently of us. In \cite{FIRT}, Fujisjima-Ioku-Ruf-Terraneo construct singular solutions with precise asymptotic formulas. Their result is fruitful also for our analysis of bifurcation diagrams. See Remarks \ref{rmk:B} and \ref{rmk:D} below. On the other hand, more recently, Kumagai \cite{Ku} establishes suitable nonexistence results and also the uniform  boundedness of finite Morse index solutions. Consequently, he proves, with a different approach,  that the bifurcation diagrams admit infinitely many turning points and more.  For the details, check \cite{Ku}.  As above, we recently observe various new developments in the study of \eqref{p0} in dimension two. The present work establishes one of the leading results.

Now let us give our main results on concentration phenomena. 
\subsection{Concentration estimates}\label{subsec:ce}
We set $D=\{x\in \mathbb{R}^2\ |\ |x|<1\}$. Throughout this paper, we always assume the next condition without further comments in the results below. 
\begin{enumerate}
\item[(H0)]$h:[0,\infty)\to \mathbb{R}$ is nonnegative $C^1$ function and  there exists a value $t_0>0$ such that $h(t)>0$ for all $t\ge t_0$. 
\end{enumerate} 
First note that  every solution $u$ of \eqref{p0} with $\Omega=D$ is radially symmetric and strictly decreasing for each radial direction by \cite{GNN}. In particular, $u(0)=\max_{x\in D}u(x)$. Then, regarding $u=u(|x|)$ for all $x\in D$, \eqref{p0} turns into a one dimensional equation for $u=u(r)$ $(r\in[0,1])$, 
\begin{equation}\label{p}
\begin{cases}
-u''-\frac1r u''=\la f(u),\ \  u>0\ \text{ in }(0,1),\\
u'(0)=0=u(1),
\end{cases}
\end{equation}
where we put $f(u)=h(u)e^{u^p}$. Now, we introduce our basic condition for the concentration analysis which is a straight adaptation of (H4) in \cite{D} to our situation $p>0$.  
 \begin{enumerate}
\item[(H1)]  It holds that 
\[\displaystyle \lim_{t\to \infty}\frac{h'(t)}{t^{p-1}h(t)}=0.\]
\end{enumerate}
This implies that $\lim_{t\to \infty}(t^{-p}\log{h(t)})=0$. See \eqref{ha} in Lemma \ref{lem:h} below. In particular,  $h(t)$ has a lower growth than $e^{t^p}$ at infinity and
 $f(t)\to \infty$ as $t\to \infty$.  Moreover, it is  equivalent to  
\[
\displaystyle \lim_{t\to \infty}\frac{f'(t)}{pt^{p-1}f(t)}=1.
\]
(H1) admits the standard example $h$ which satisfies  $h(t)=t^me^{\alpha t^{q}}$ for all large $t>0$ where $m,\alpha \in \mathbb{R}$ and $0<q<p$.  Some more properties derived from (H1) are summarized in Section \ref{sec:pre}. 

From now on, we consider any sequence of solutions of \eqref{p}. That is, we study the behavior of any sequence  $\{(\la_n,\mu_n,u_n)\}$ satisfying 
\begin{equation}\label{q}
\begin{cases}
-u_n''-\frac1r u_n'=\la_n f(u_n),\ \  u_n>0\ \text{ in }(0,1),\\
u_n(0)=\mu_n,\ u_n'(0)=0=u_n(1),
\end{cases}
\end{equation}
for all $n\in \mathbb{N}$. We are interested in the case $\mu_n\to \infty$ as $n\to \infty$. The existence of such a sequence is ensured, for example,  by Theorem 3 in \cite{AtPe} with a few more conditions on $h$ which still admits the standard example noted above, or Lemma 2.1 in \cite{AKG}.

 We finally  note that our concentration analysis is based on the study of the following quantity,
\[
p\mu_n^{p-1}\int_0^r\la_n f(u_n)rdr
\]
which can be connected  to the  integral value $\int_0^\infty e^zrdr$ via the scaling argument where $z$ is some solution of the  limit equation \eqref{limp}.   Moreover,  we also refer to two variants of the previous quantity, 
\[
p\mu_n^{p-2}\int_0^r\la_n u_nf(u_n)rdr\text{ \ \ and \ \ } p\int_0^r\la_n u_n^{p-1}f(u_n)rdr.
\]
One may be interested in  the former  one in view of the variational analysis in the case $0< p \le2$. The latter one, which seems to be less meaningful if $p\not=2$, will be a symbolic quantity in this paper for the case $p>2$ which purely quantify the appearance of bubbles. Regarding those quantities as variants or extensions of the usual variational energy $\int_0^1 \la u_nf(u_n)rdr $, we often call them ``energy" for convenience in this paper.  

Now, we first give our basic result on the first concentration around the maximum point for any $p>0$.  
\begin{theorem}\label{thm1} Assume $p>0$, (H1), and $\{(\la_n,\mu_n,u_n)\}$ is a sequence of solutions of \eqref{q} such that $\mu_n\to \infty$ as $n\to \infty$. Let $(\ga_{0,n})$ be a sequence of positive numbers such that 
\[
 p\la_n\mu_n^{p-1}f(\mu_n)\ga_{0,n}^2=1
\] 
and put 
\[
z_{0,n}(r)=p\mu_n^{p-1}(u_n(\ga_{0,n} r)-\mu_n)
\]
for all $r\in[0,1/\ga_{0,n}]$ and $n\in\mathbb{N}$. Then, up  to a subsequence,  there exists a sequence $(\rho_{0,n})\subset (0,1)$ such that $\ga_{0,n}\to0$, $\rho_{0,n}\to0$, $\rho_{0,n}/\ga_{0,n}\to\infty$, $u_n(\rho_{0,n})/\mu_n\to1$,   and $\|z_{0,n}-z_0\|_{C^1([0,\rho_{0,n}/\ga_{0.n}])}\to0$  where $z_0$ is  a function on $[0,\infty)$ defined by
\[
z_0(r)=\log{\frac{64}{(8+r^2)^2}}
\] 
satisfying 
\begin{equation}\label{leq}
\begin{cases}
-z_0''-\frac1r z_0'=e^{z_0}\ \text{ in }(0,\infty),\\
z_0(0)=0=z_0'(0),
\end{cases}
\end{equation}  
and further, 
\begin{equation}\label{eq:sq}
p\mu_n^{p-1}\int_0^{\rho_{0,n}}\la_n f(u_n)rdr\to \ \displaystyle \int_0^\infty e^{z_0}rdr=4
\end{equation}
as $n\to \infty$. Moreover, there exists a sequence $(r_{0,n})\subset (0,\rho_{0,n})$ such that $r_{0,n}/\ga_{0,n}\to 2\sqrt{2}$, $u_n(r_{0,n})/\mu_n\to1$, and 
\begin{equation}\label{eq:r0}
p\la_nr_{0,n}^2 u_n(r_{0,n})^{p-1}f(u_n(r_{0,n}))\to 2
\end{equation}
as $n\to \infty$.
\end{theorem}
Here, we prove that the limit profile of the first bubble is given by the regular function $z_0$ as observed in previous works on $0<p\le2$ referred above. Moreover, \eqref{eq:sq} shows that for some appropriate sequence $(\rho_{0,n})$, the limit energy on the interval $[0,\rho_{0,n}]$ is also characterized by $z_0$. A natural question is that  if we can replace $\rho_{0,n}$ with $1$ there. We will see below that  the answer depends on our choice of $p$ and $h$. We also note that  \eqref{eq:r0} will be one of the key estimates in our oscillation analysis in Section \ref{sec:osc} below. 

Let us proceed to further analysis in each case of $p<2$, $p=2$, and $p>2$. 
\subsubsection{The case $p\le2$}
For the subcritical case $p<2$, we prove the following. 
\begin{theorem}\label{thm2} Assume $p\in(0,2)$, (H1), and  $\{(\la_n,\mu_n,u_n)\}$ is a sequence of solutions of \eqref{q} with $\mu_n\to \infty$ as $n\to \infty$. Then we get, up  to a subsequence, that 
\begin{equation}\label{sub1}
\lim_{n\to \infty}p\mu_n^{p-1} \int_0^1 \la_n f(u_n)rdr=4,
\end{equation}  
and 
\begin{equation}\label{sub10}
\lim_{n\to \infty}p\mu_n^{p-2} \int_0^1 \la_n u_nf(u_n)rdr=4.
\end{equation}  
Moreover, we have
\begin{equation}\label{sub2}
\lim_{n\to \infty}\frac{\log{\frac1{\la_n}}}{\mu_n^p}=\frac{2-p}{2}.
\end{equation}   
Especially,  we obtain  $\lim_{n\to \infty}\la_n=0$. Finally, we get
\begin{equation}\label{sub3}
\lim_{n\to \infty}p\mu_n^{p-1}u_n(r)= 4\log{\frac1{r}}\text{ in }C_{\text{loc}}^2((0,1]).
\end{equation}
\end{theorem}
\eqref{sub1} ensures that  the limit energy in whole interval is simply characterized by the limit function $z_0$.  Then, it allows us to complete the precise asymptotic formulas \eqref{sub2} and \eqref{sub3}. From \eqref{sub2}, we see that $\la_n\to0$ as $n\to \infty$. On the other hand, \eqref{sub3} implies that $u_n(r)\to 0$ if $1<p<2$ and $u_n(r)\to 4\log{(1/r)}$ if $p=1$ in $C^2_{\text{loc}}((0,1])$ and $u_n(r)\to \infty$  for all $r\in[0,1)$  if $p\in(0,1)$ as $n\to \infty$. Hence \eqref{sub3} unifies some consequences in \cite{OS} and \cite{NS}. In addition, we remark that \eqref{sub10} implies the divergence of  the variational energy 
\[
\lim_{n\to \infty}\int_0^1 \la_n u_nf(u_n)rdr= \infty
\]
since $0<p<2$. 

Next, the situation becomes  more delicate in the critical case $p=2$. We give  a less complete but reasonable result.
\begin{theorem}\label{thm:crit}
Assume $p=2$, (H1),  and  $\{(\la_n,\mu_n,u_n)\}$ is a sequence of solutions of \eqref{q} with $\mu_n\to \infty$ as $n\to \infty$. Then, up to a subsequence, there exists a sequence $(r_n)\subset (0,1)$ such that 
 $u_n(r_n)/\mu_n\to0$ as $n\to \infty$, and 
\begin{equation}\label{cri1}
\lim_{n\to \infty}2\mu_n\int_0^{r_n}\la_n f(u_n)rdr= 4.
\end{equation}
Moreover,  we get
\begin{equation}\label{cri2}
\lim_{n\to \infty}\frac{\log{\frac1{\la_n}}}{\mu_n^2}=0.
\end{equation}
\end{theorem}
Notice that  \eqref{cri1} extends the integral interval $[0,\rho_{0,n}]$ in \eqref{eq:sq} to a wider one. Moreover, \eqref{cri2} confirms that \eqref{sub2} holds also for $p=2$. There remains some questions which ask if \eqref{sub1}, \eqref{sub10}, and \eqref{sub3} still holds for $p=2$ and if $\lim_{n\to \infty}\la_n=0$. The answer depends on the growth condition on the perturbation term $h$.  Actually, both answers are positive if $h(t)=ut^{\alpha t^q}$ with $\alpha\ge0$ and $0<q<2$ by Theorem 2 in \cite{MM1}, (5), (6) of Theorem 2 in \cite{AD}, and Theorems 1.1 and 1.5 in \cite{N}.  On the other hand, both  are negative if  $h(t)=te^{-\alpha t}$ and $\alpha>0$ by Theorem 0.3 in \cite{MT}. 
This is because of the appearance of the residual mass in addition to the standard bubble. Hence, in the critical case, \eqref{p0} admits more various behaviors than those in the subcritical case. Further analysis seems to be still an interesting problem.  

In this paper, we   proceed to our main discussion on the supercritical case $p>2$.
\subsubsection{The case $p>2$}
Let us consider the supercritical case. In this case, we encounter a striking phenomenon, the  infinite sequences of bubbles, as  noted above. To see this, we assume $p> 2$. Put $a_0=2$, $\delta_0=1$. For any $k\in \mathbb{N}$, we define numbers $\delta_k\in (0,\delta_{k-1})$ and $a_k\in(0,2)$ by the relations,
\begin{equation}\label{eq:del1}
\frac{2p}{2+a_{k-1}}\left(1-\frac{\delta_{k}}{\delta_{k-1}}\right)-1+\left(\frac{\delta_{k}}{\delta_{k-1}}\right)^p=0
\end{equation}
and
\begin{equation}\label{eq:del2}
a_k=2-\left(\frac{\delta_k}{\delta_{k-1}}\right)^{p-1}(2+a_{k-1}).
\end{equation}
We can prove that if $p>2$, $\delta_k$ and $a_k$ are well-defined, strictly decreasing with respect to $k\in \mathbb{N}$, and converge to zero as $k\to \infty$. See Lemmas \ref{lem:bf}, \ref{lem:bf2}, and \ref{lem:div}.  These sequences of numbers, which are determined only by $p>2$, completely characterize the asymptotic profile, energy, and location of each bubble.  We get the following.  
\begin{theorem}\label{thm30}  Suppose $p>2$, (H1), and  $\{(\la_n,\mu_n,u_n)\}$ is a sequence of solutions of \eqref{q} with $\mu_n\to \infty$ as $n\to \infty$.  Moreover, choose a sequence $(\rho_{0,n})$  in Theorem \ref{thm1}. Then,  for all $k\in \mathbb{N}$, there exist sequences $(r_{k,n}),(\rho_{k,n}),(\bar{\rho}_{k,n})\subset (0,1)$ of values  such that $u_n(r_{k,n})/\mu_n\to \delta_k$,  
\begin{equation}\label{eq:rk}
p\la_n r_{k,n}^2 u_n(r_{k,n})^{p-1} f(u_n(r_{k,n}))\to \frac{a_k^2}{2},
\end{equation}
$u_n(\rho_{k,n})/\mu_n\to \delta_k$, $u_n(\bar{\rho}_{k,n})/\mu_n\to \delta_k$,  $\rho_{k-1,n}/\bar{\rho}_{k,n}\to 0$, $\bar{\rho}_{k,n}/r_{k,n}\to0$, $r_{k,n}/\rho_{k,n}\to 0$ as $n\to \infty$, and if we set sequences $(\ga_{k,n})$  of positive values and $(z_{k,n})$ of functions  so that 
\[
\ga_{k,n}^2 p\la_nu_n(r_{k,n})^{p-1}f(u_n(r_{k,n}))=1
\]
and
\[
z_{k,n}(r)=p u_n(r_{k,n})^{p-1}(u_n(\ga_{k,n}r)-u_n(r_{k,n}))
\]  
for all $r\in [0,1/\ga_{k,n}]$ and  $n\in \mathbb{N}$, then we have that, up to a subsequence,  $\ga_{k,n}\to0$ and
\[
\|z_{k,n}-z_k\|_{C^2([\bar{\rho}_{k,n}/\ga_{k,n},\rho_{k,n}/\ga_{k,n}])}\to0,
\]  
where  $z_k$ is a function on $(0,\infty)$ given by 
\[
z_k(r)=\log{\frac{2a_k^2 b_k }{r^{2-a_k}(1+b_k r^{a_k})^2}}
\]
with $b_k=\left(\sqrt{2}/a_k\right)^{a_k}$, which satisfies 
\[
\begin{cases}
-z_k''-\frac{1}{r}z_k'=e^{z_k}\ \ \text{ in }(0,\infty),\\
z_k(a_k/\sqrt{2})=0,\ -(a_k/\sqrt{2})z_k'(a_k/\sqrt{2})=2,
\end{cases}
\]
 and
\[
\int_0^\infty e^{z_k}rdr=2a_k.
\]
Moreover, we have
\begin{equation}\label{eq:en1}
p\mu_n^{p-1}\int_{\rho_{k-1,n}}^{\bar{\rho}_{k,n}}\la_n f(u_n)rdr \to0,\ \ \ \ p\mu_n^{p-1}\int_{\bar{\rho}_{k,n}}^{\rho_{k,n}}\la_n f(u_n)rdr \to\frac{2a_k}{\delta_k^{p-1}},
\end{equation}
and 
\begin{equation}\label{eq:en2}
p\int_{\rho_{k-1,n}}^{\bar{\rho}_{k,n}}\la_n u_n^{p-1}f(u_n)rdr \to0,\ \ \ \ p\int_{\bar{\rho}_{k,n}}^{\rho_{k,n}}\la_n u_n^{p-1} f(u_n)rdr \to 2a_k,
\end{equation}
as $n\to \infty$. 
\end{theorem}
The previous theorem shows that for all $k\in \mathbb{N}$, if we scale the sequence $(u_n)$ of solutions  around the sequence  $(r_{k,n})$ of points satisfying \eqref{eq:rk}, then it converges to a singular solution $z_k$ of the Liouville equation with the energy $\int_0^{\infty}e^{z_k}rdr=2a_k$. In particular, we detect an infinite sequence of bubbles characterized by the sequence of singular limit profiles $(z_k)$.  This phenomenon causes remarkable differences on several asymptotic formulas.  Among others, noting the fact  $\sum_{k=0}^{\infty}a_k=\infty$, which will be proved in Lemma \ref{lem:div} below, we get that the uniform boundedness of the energy is no longer true. We prove the following. 
\begin{theorem}\label{thm3} Suppose as in Theorem \ref{thm30}. Then, for any sequence $(r_n)\subset (0,1)$ such that $u_n(r_n)/\mu_n\to 0$, we get 
\begin{equation}\label{sup2}
\lim_{n\to \infty}p\int_0^{r_n}\la_nu_n^{p-1}f(u_n)rdr=\infty
\end{equation}
up to a subsequence. Especially, we see
\begin{equation}\label{sup222}
\lim_{n\to \infty}p\mu_n^{p-1}\int_0^{r_n}\la_nf(u_n)rdr=\infty.
\end{equation}
Moreover, we obtain   
\begin{equation}\label{sup3}
\lim_{n\to \infty}\frac{\log{\frac1{\la_n}}}{\mu_n^p}=0
\end{equation}
and
\begin{equation}\label{eq:rk2}
\lim_{n\to \infty}\frac{\log{\frac1{r_{k,n}}}}{\mu_n^p}= \frac{\delta_k^p}2
\end{equation}
for all $k\in \mathbb{N}\cup\{0\}$, and 
\begin{equation}\label{sup4}
\lim_{n\to \infty}p\mu_n^{p-1}u_n(r)= \infty
\end{equation}
for all $r\in[0,1)$.
\end{theorem}
\begin{remark}
Note that, by Theorem 3 in \cite{AtPe}, the basic  asymptotic formula for $\la_n$ is obtained. Adapting it to our problem with (H1), we get 
\[
\liminf_{n\to \infty}\frac{\log{\frac1{\la_n}}}{\mu_n^p}\ge -\left(\frac{p}{2}-1\right).
\]
Hence, we confirm that \eqref{sup3} improves the estimate.  Moreover, \eqref{eq:rk2} is the consequence of  \eqref{eq:r0}, \eqref{eq:rk}, and \eqref{sup3}. This gives the precise information on the location of the center of each concentrating bubble.   
\end{remark}
\begin{remark} Since $\delta_k\to0$ and $a_k\to0$ as $k\to \infty$ as noted above,  we deduce the asymptotic profile of the sequence of limit functions. That is, for all $r>0$, we get 
\[
z_k(r)-\log{\frac{a_k^2}{2}}\to 2\log{\frac1r}
\]
as $k\to \infty$.
\end{remark}
From \eqref{sup3} and \eqref{sup4}, we confirm that not only \eqref{sub1} but also  \eqref{sub2} or $\eqref{sub3}$ are no longer valid  for $p>2$ because of the appearance of the infinitely many bubbles.  Moreover, similarly to the critical case, the limits of $(\la_n)$ and $(u_n)$ are not clear from \eqref{sup3} and \eqref{sup4} in  general.  Hence, at the moment, it is reasonable to summarize the global behavior as the next alternative.  Here, we note  an additional assumption on $h$,
\begin{enumerate}
\item[(H2)] it holds that $\inf_{t>0}(f(t)/t)>0$.
\end{enumerate} 
Then  
we get the following.
\begin{theorem}\label{thm:gl} Let $p\ge2$ and assume  (H1). Suppose $\{(\la_n,\mu_n,u_n)\}$ is a sequences of solutions of \eqref{q} with $\mu_n\to \infty$ and  $\la_n\to \la_*$ for some $\la_*\in[0,\infty]$ as $n\to \infty$.  Then, either one of the next (i), (ii), and (iii) occurs up to a subsequence.
\begin{enumerate}
\item[(i)] $\la_*=0$ and $u_n\to 0$ in $C^2_{\text{loc}}((0,1])$. 
\item[(ii)] $\la_*\in (0,\infty)$ and there exists a smooth nonincreasing function $u_*$ on $(0,1]$ such that $u_n\to u_*$ in $C^2_{\text{loc}}((0,1])$. Furthermore, $(\la_*,u_*)$ satisfies  
\begin{equation}\label{q*}
\begin{cases} -u_*''-\frac1r u_*'=\la_*f(u_*),\ u_*\ge 0\text{ in }(0,1),\\
\ u_*(1)=0,\ \  \int_0^1 \la_* f(u_*)rdr<\infty.
\end{cases}
\end{equation}
Moreover,  if $f(0)>0$, we get $u_*>0$ in $(0,1)$.
\item[(iii)] $\la_*=\infty$ and $(u_n)$ is uniformly bounded in $C_{\text{loc}}((0,1])$.
\end{enumerate}
In addition, if we assume (H2), then (iii) does not  happen.
\end{theorem}
\begin{remark}
This theorem is simply the consequence of the Pohozaev identity and proved independently of the concentration estimates in the previous theorems.  
\end{remark}
As noted  above, for $p=2$, the case (i) is observed  in \cite{MM1} and \cite{N} and (ii) is found in \cite{MT} with $0<u_*(0)<\infty$. On the other hand,  for $p>2$,  our oscillation estimates in Section \ref{sec:osc} will give a useful approach to find $f$ which admits the case (ii).  See Theorem \ref{thm:app2} and also Remark \ref{rmk:B} there.  In view of them, in the supercritical case, we observe the behavior in (ii) for typical nonlinearities. See also an earlier work \cite{McMc} and a very recent one \cite{Ku} which generalizes the former one.

 The next sections are devoted to the proof of these results. We note that our second aim, the oscillation analysis, will be summarized in Section \ref{sec:osc}. 
\subsection{Idea and  organization}
As explained above, we apply the scaling techniques developed in the previous works on \eqref{p0}. Among them, we are especially inspired by the arguments in \cite{AD}, \cite{D}, \cite{GN},  \cite{MM1}, \cite{MM2}, \cite{N}, and \cite{OS}. We simply extend and utilize the scaling argument found in \cite{AD} for $p=2$ and also that in \cite{OS} for $0<p\le2$. The basic techniques to deduce and integrate the limit equation comes from \cite{GN} and also \cite{GGP}. Moreover, for our main argument on $p>2$, we apply the pointwise techniques and radial analysis in  \cite{D} and \cite{MM1} together with the Green type identities in Lemma \ref{id}, which is an extension of Lemma 2.2 in \cite{N}, below. Particularly, in order to detect more and more bubbles,  extending the idea in \cite{D} to our case $p>2$, we look for a suitable sequence $(r_n)\subset (0,1)$ such that $p\la_nr_n^2u_n(r_n)^{p-1}f(u_n(r_n))\to c$ as $n\to \infty$ for some $c\not=0$. To do this, we also extend the idea of Lemma 3 in Appendix A in \cite{D}.  See Lemma \ref{lemD} below.  A different point  from  \cite{D} is that we consider any radial positive solutions without the assumption for the boundedness of the energy. Hence our starting point is closer to  \cite{MM1} and also \cite{MM2}. Following them, after deducing each  concentration profile, we establish the suitable pointwise estimate for the scaled functions. Then, using this, we estimate the energy on the out side of the concentration region and extend the interval with no additional bubble.  Along this line, some of the proofs in Lemmas \ref{lem21}, \ref{lem:thm21}, \ref{lem:e10}, and  \ref{lem:ff} are inspired by them. We note that since we do not need their small order expansion of the energy, our calculation is rather rough and easier. The crucial different point  is that since we want to show  that the next bubble does appear,  we need to confirm that the nontrivial energy remains on the outside of the previous concentration region. This is accomplished by the argument based on the identities in Lemma \ref{id}. See Lemma \ref{lem:e1}. Then, we succeed in detecting the next bubble with the aid of Lemma \ref{lemD}. See Lemmas \ref{lem:e3} and \ref{lem:f1}. We also emphasize that we complete  the precise characterization of each bubble via the system of the recurrence formulas \eqref{eq:del1} and \eqref{eq:del2} which are determined by the suitable balance between adjacent two bubbles. These are obtained by carefully combining the identity \eqref{id2} and our concentration estimates.  See Lemmas \ref{lem:dl} and \ref{lem:ak}. Moreover, we note that \eqref{id2} will become also a key tool to connect our concentration estimates and oscillation ones in Section \ref{sec:osc}. See Theorem \ref{thm:osc} and its proof.  
  
The organization of this paper is the following. In Section \ref{sec:pre}, we collect some basic facts and key tools which will be  used in main sections. Next in Section \ref{sec:con1}, we carry out  the standard concentration analysis for the case $p>0$ and prove Theorems \ref{thm1}, \ref{thm2}, and some more estimates used in the following sections. Next, Section \ref{sec:con2} is devoted to deduce the infinite sequence of bubbles  for $p>2$. Then, we  next establish our second main results on the oscillation phenomena in Section \ref{sec:osc}.  Finally,  in Sections \ref{sec:lim}, we discuss what happens on the infinite sequence of bubbles in the limit cases $p\to2^+$ and $p\to \infty$ respectively. 

In this paper, we sometimes use the same characters $c,C>0$, $(r_n)$, $(s_n)$, $(t_n)$, $(R_n)$, $(\e_n)$, and so on to denote several constants and  sequences if there are no confusions. In addition, in our main argument, we often extract a subsequence from a given sequence without any change of its suffix for simplicity.  
\section{Preliminaries and key tools}\label{sec:pre}
In this section, we give some  basic facts and key tools for our proof. 
\subsection{Conditions on  nonlinearities}
We first show some basic facts deduced from our assumptions. The next lemma is a consequence of  (H1) which corresponds to Properties (P1) and (P3) in \cite{D} when $p=2$. 
\begin{lemma}\label{lem:h} Suppose (H1). Then  there exists a constant $t_1 >0$ such that  $f(t)$ is increasing for all $t\ge t_1$. Moreover, it holds that  
\begin{equation}\label{ha}
\lim_{t\to \infty}\frac{\log {h(t)}}{t^p}=0.
\end{equation}
In addition, for any number $A>0$ and sequence $(\xi_n)$ of values such that $\xi_n\to\infty$, we get
\begin{equation}\label{hb}
\lim_{n\to \infty}\max_{\xi\in\left[\xi_n-A/\xi_n^{p-1},\  \xi_n+A/\xi_n^{p-1}\right]}\left|\frac{h(\xi)}{h(\xi_n)}-1\right|=0.
\end{equation}
\end{lemma} 
\begin{proof} For any $t>0$, we have 
\[
f'(t)=pt^{p-1}h(t)\left(1+\frac{h'(t)}{pt^{p-1}h(t)}\right)e^{t^p}.
\]
 Then (H1) implies the first assertion. \eqref{ha}  is proved by the de l'H\^opital rule with (H1). Finally, the proof of \eqref{hb} is given by a trivial modification of that of Property (P3) in \cite{D}. For readers' convenience, we show the proof. First, we choose a sequence $(\alpha_n)\subset \left[\xi_n-A/\xi_n^{p-1},\  \xi_n+A/\xi_n^{p-1}\right]$ so that 
\[
\begin{split}
\max_{\left[\xi_n-A/\xi_n^{p-1},\  \xi_n+A/\xi_n^{p-1}\right]}\left|\frac{h(\xi)}{h(\xi_n)}-1\right|&=\left|\frac{h(\alpha_n)}{h(\xi_n)}-1\right|\end{split}
\]
for all $n\in \mathbb{N}$.  Then there exists a sequence $(\beta_n)\subset \left[\xi_n-A/\xi_n^{p-1},\  \xi_n+A/\xi_n^{p-1}\right]$ of values  such that 
\[
\begin{split}
\left|\frac{h(\alpha_n)}{h(\xi_n)}-1\right|=\left|\frac{h'(\beta_n)(\alpha_n-\xi_n)}{h(\xi_n)}\right|=\left|\frac{h'(\beta_n)}{\beta_n^{p-1}h(\beta_n)}\beta_n^{p-1}(\alpha_n-\xi_n)\frac{h(\beta_n)}{h(\xi_n)}\right|
\end{split}
\]
for all $n\in \mathbb{N}$. Then noting (H1) and the fact that $\beta_n^{p-1}(\alpha_n-\xi_n)$ is bounded uniformly for all $n\in \mathbb{N}$, we find a sequence $(\e_n)$ of positive values  such that $\e_n\to0$ as $n\to \infty$ and  
\[
\begin{split}
\left|\frac{h(\alpha_n)}{h(\xi_n)}-1\right|&=\e_n \left|\frac{h(\beta_n)}{h(\xi_n)}\right|\le \e_n\left( \left|\frac{h(\alpha_n)}{h(\xi_n)}-1\right|+1\right)
\end{split}
\]
for all $n\in \mathbb{N}$. This gives the desired conclusion. We complete the proof. 
\end{proof}
We also use the  next property which covers the case $f(t)$ is not increasing for all $t>0$.
\begin{lemma}\label{lem:h2} Assume (H1). Then there exists a constant $t_2>0$ such that  
\[
\sup_{0<s<t}\frac{f(s)}{f(t)}\le1
\] 
for all $t\ge t_2$.
\end{lemma}
\begin{proof} Let $t_1>0$ be the constant in the previous lemma and set $t> t_1$. Then, we have that $\sup_{t_1\le s<t}f(s)/f(t)\le1$. Moreover, since \eqref{ha} implies $f(t)\to \infty$ as $t\to \infty$, we get $\sup_{0<s< t_1}f(s)/f(t)\to0$ as $t\to \infty$. This shows the desired conclusion. We finish  the proof.  
\end{proof}
We use the next property for energy estimates.
\begin{lemma}\label{lem:h3} Suppose (H1). Let $(\xi_n)\subset (0,\infty)$, $M>0$, and $\delta\in(0,1)$ be any sequence and constants such that $\xi_n\to \infty$ as $n\to \infty$. Then we get
\[
\lim_{n\to \infty}\max_{-p(1-\delta)\xi_n^p\le s\le -M} \left|s^{-1}\log{\frac{h\left(\xi_n+\frac{s}{p\xi_n^{p-1}}\right)}{h(\xi_n)}}\right|=0.
\]
\end{lemma}
\begin{proof} Let $(s_n)\subset [-p(1-\delta)\xi_n^p, -M]$ be a sequence such that
\[
\left|s_n^{-1}\log{\frac{h\left(\xi_n+\frac{s_n}{p\xi_n^{p-1}}\right)}{h(\xi_n)}}\right|=\max_{-p\xi_n^p(1-\delta)\xi_n\le s\le -M} \left|s^{-1}\log{\frac{h\left(\xi_n+\frac{s}{p\xi_n^{p-1}}\right)}{h(\xi_n)}}\right|
\]
for all $n\in \mathbb{N}$. Take any  $\e>0$. From (H1), there exists a value $t_\e>0$ such that
\[
|(\log{(h(t))})'|\le pc_pt^{p-1}\e
\]
for all $t\ge t_\e$ where we put $c_p=1$ if $p\ge1$ and $c_p=\delta^{1-p}$ otherwise. It follows that  for any $t>s>t_\e$, 
\[
\left|\log{\frac{h(s)}{h(t)}}\right|\le  c_p(t^p-s^p)\e.
\]
Then there exists a sequence $(\theta_n)\subset (0,1)$ such that 
\[
\begin{split}
\left|\log{\frac{h\left(\xi_n+\frac{s_n}{p\xi_n^{p-1}}\right)}{h(\xi_n)}}\right|&\le c_p\e\left\{\xi_n^p-\left(\xi_n+\frac{s_n}{p\xi_n^{p-1}}\right)^p\right\}=c_p\e\left(1+\frac{\theta_n s_n}{p\xi_n^p}\right)^{p-1}|s_n|\\
&\le \e |s_n|
\end{split}
\]
for all large $n\in \mathbb{N}$.  We finish the proof. 
\end{proof}
The next one is a result of (H2). 
\begin{lemma}\label{kap}
Assume (H2). Then 
\[
\sup\{\la>0\ |\ \text{\eqref{p0} has a classical solution $u$} \}<\infty.
\]
\end{lemma}
\begin{proof} The proof is given by the Kaplan method. We here follow the proof in \cite{Dp} (See Proposition 3.3.1 there). Set a value $c>0$ so that $\inf_{t>0}(f(t)/t)=c$.  Let $u>0$ be a smooth  solution of \eqref{p0} and $\Lambda_1>0$ and $\Phi_1>0$ in $\Omega$ the  first eigenvalue and eigenfunction of $-\Delta$ on $\Omega$ respectively. Then, integrating by parts, we get
\[
\Lambda_1\int_\Omega u\Phi_1dx=\int_\Omega \la f(u) \Phi_1dx\ge c\la \int_\Omega  u \Phi_1dx.
\] 
Since $u,\Phi_1>0$ in $\Omega$, we have $\la \le \Lambda_1/c$. This completes the proof. 
\end{proof}

\subsection{Key tools}\label{sub:kt}
We next collect the key tools for our main arguments.  We first give the key identities which we often use in the proof below. In this subsection, we always suppose $\{(\la_n,\mu_n,u_n)\}$ is a sequence of solutions of \eqref{q}. 
\begin{lemma}\label{id} Suppose $0< s\le1$. Then we have  
\begin{equation}\label{id0}
-su_n'(s)=\int_{0}^{s} \la_nf(u_n(r))rdr
\end{equation}
and
\begin{equation}\label{id1}
u_n(0)-u_n(s)=\int_{0}^{s} \la_nf(u_n(r))r\log{\frac{s}{r}}dr.
\end{equation}
Let $0< s<t\le1$. Then, we get
\begin{equation}\label{id2}
u_n(s)-u_n(t)=\left(\log{\frac{t}{s}}\right)\int_{0}^{s} \la_nf(u_n(r))rdr+\int_{s}^{t} \la_nf(u_n(r))r\log{\frac{t}{r}}dr.
\end{equation}
\end{lemma}
\begin{proof} First, let $0<s\le1$. Multiplying the equation in \eqref{q} by $r$ and integrating from $0$ to $s$, we get \eqref{id0}. On the other hand, multiplying the equation by $r\log{r}$ and    integrating from $0$ to $s$,  we have
\[
\begin{split}
\int_0^s\la_n f(u_n)r\log{r}dr
&=-u_n'(s)s\log{s}+u_n(s)-u_n(0).\\
\end{split}
\] 
Then using \eqref{id0}, we readily get \eqref{id1}. Next let $0<s<t\le1$. Similarly, multiplying the equation by $r\log{r}$ and integrating from $s$ to $t$, we obtain 
\[
\begin{split}
\int_s^t\la_n f(u_n)r\log{r}dr
&=-u_n'(t)t\log{t}+u_n'(s)s\log{s}+u_n(t)-u_n(s)\\
\end{split}
\] 
Then using \eqref{id0} for the first two terms on the right-hand side, we get \eqref{id2} after simple calculations. We finish the proof.
\end{proof}
Next, we extend  some tools from the pointwise technique in \cite{D}. For any $r\in[0,1]$ and $n\in \mathbb{N}$, we set 
\[
\phi_n(r):=p\la_nr^2 u_n(r)^{p-1} f(u_n(r))
\]
and 
\[
\psi_n(r):=-pru_n(r)^{p-1}u_n'(r)=pu_n(r)^{p-1}\int_0^r\la_n f(u_n(s))ds
\]
by \eqref{id0}. To study the interaction between $\phi_n$ and $\psi_n$ is a key to detect a sequence of bubbles. The following two lemmas are inspired by Lemma 3 of Appendix in \cite{D}. We first get the following. 
\begin{lemma}\label{lem:pp}
Assume (H1). Let $(r_n)\subset  (0,1)$ be any sequence such that $u_n(r_n)\to \infty$ as $n\to \infty$. Then we have
\[
\begin{split}
\phi_n'(r_n)&=\frac{\phi_n(r_n)}{r_n}\left[2-\psi_n(r_n)\left\{\frac{f'(u_n(r_n))}{pu_n(r_n)^{p-1}f(u_n(r_n))}+\frac{p-1}{pu_n(r_n)^{p}}\right\}\right]
\\
&=\frac{\phi_n(r_n)}{r_n}\left[2-\psi_n(r_n)(1+o(1))\right]
\end{split}
\]
where $o(1)\to0$ as $n\to \infty$. In particular, if $\limsup_{n\to \infty }\psi_n(r_n)<2$, then $\phi_n'(r_n)>0$ for all large $n\in \mathbb{N}$. \end{lemma}
\begin{proof} Differentiating $\phi_n(r)$ and using the definition of $\psi_n$, we readily deduce the first equality.  Then (H1) shows the second one.  We finish the proof.
\end{proof}
For any $0\le s<t\le1$ and $n\in\mathbb{N}$, we put 
\begin{equation}\label{defE}
E_n(s,t)=p\mu_n^{p-1}\int_s^t \la_nf(u_n)rdr.
\end{equation}
We next obtain  the following.  
\begin{lemma}\label{lemD}
Suppose (H1). Let $(r_n),(s_n)\subset (0,1)$ be any sequences such that $r_n<s_n$ for all $n\in \mathbb{N}$ and $\liminf_{n\to \infty}(u_n(s_n)/\mu_n)>0$. Moreover, we assume that there exists a value $c\in (0,2)$ such that $\lim_{n\to \infty}\psi(r_n)=c$ and  $\lim_{n\to \infty}\sup_{r\in[r_n,s_n]}\phi_n(r)=0$. Then we have that 
\[\lim_{n\to \infty}E_n(r_n,s_n)=0.\]
\end{lemma}
\begin{proof} 
 If the assertion fails, after extracting a subsequence if necessary, we obtain a constant $c_0>0$ such that $E_n(r_n,s_n)\ge c_0$ for all $n\in \mathbb{N}$. Then there exists a sequence $(t_n)\subset (r_n,s_n)$ such that $\lim_{n\to \infty}E_n(r_n,t_n)\in(0,2-c)$. It follows from Lemma \ref{lem:pp} that  there exists a sequence $(t_n')\subset (r_n,t_n)$ such that 
\begin{equation}\label{eq:dd}
\begin{split}
o(1)&=\phi_n(t_n)-\phi_n(r_n)\\
&=\int_{r_n}^{t_n}\left[2-\psi_n(r)\left\{\frac{f'(u_n(r))}{pu_n(r)^{p-1}f(u_n(r))}+\frac{p-1}{pu_n(r)^{p}}\right\}\right]\frac{\phi_n(r)}{r}dr\\
&=\{2-\psi_n(t_n')(1+o(1))\}\int_{r_n}^{t_n}\frac{\phi_n(r)}{r}dr
\end{split}
\end{equation}
for all $n\in\mathbb{N}$. On the other hand, since
\[
\int_{0}^{t_n'}\la_n f(u_n)rdr=\int_{r_n}^{t_n'}\la_n f(u_n)rdr+\int_{0}^{r_n}\la_n f(u_n)rdr,
\] 
we get
\[
\begin{split}
\psi_n(t_n')&=\left(\frac{u_n(t_n')}{\mu_n}\right)^{p-1}E_n(r_n,t_n')+\left(\frac{u_n(t_n')}{u_n(r_n)}\right)^{p-1}\psi_n(r_n)\\
&\le E_n(r_n,t_n)+c+o(1) 
\end{split}
\]
for all $n\in \mathbb{N}$. Moreover, from our assumption, there exists a constant $c_1>0$ such that $u_n(r)\ge  c_1\mu_n$ for all $r\in[r_n,t_n]$ and large $n\in \mathbb{N}$. Using these estimates for \eqref{eq:dd}, we get
\[
o(1)\ge \left(2-c -E_n(r_n,t_n)+o(1)\right)c_1^{p-1} E_n(r_n,t_n)
\]
for all large $n\in \mathbb{N}$. Recalling our choice of $(t_n)$, we get a contradiction. This finishes the proof.
\end{proof}
\subsection{Global behaviors}
Next, we check  the global behavior of sequences of solutions for the case $p>1$. As in the proof of Theorem 2.1 in \cite{OS},  
 we use the following Pohozaev identity. Set $F(t)=\int_0^tf(s)ds$ for $t\ge0$.
\begin{lemma}\label{lem:po} Let $(\la,u)$ be a solution of \eqref{p}. Then for any $r\in (0,1]$, we have
\[
(ru'(r))^2=4\int_0^r\la F(u(s))sds-2\la F(u(r))r^2.
\]
\end{lemma}
\begin{proof} Multiplying the equation in \eqref{p} by $r^2u'$, we get
\[
-\frac12 \{(ru'(r))^2\}'=\{\la F(u(r))\}'r^2.
\]
Then  integrating by parts over $[0,r]$ gives the desired identity. This finishes the proof. 
\end{proof}
We prove the next alternative. 
\begin{lemma}\label{lem:gl} Assume $p>1$, (H1), and $\{(\la_n,\mu_n,u_n)\}$ is a sequence of solutions of \eqref{q} such that $\mu_n\to \infty$ and $\la_n\to \la_*$ as $n\to \infty$ for some $\la_*\in[0,\infty]$. Then, either one of the next (i), (ii), and (iii) occurs up to a subsequence.
\begin{enumerate}
\item[(i)] $\la_*=0$ and $u_n\to 0$ in $C^2_{\text{loc}}((0,1])$. 
\item[(ii)] $\la_*\in(0,\infty)$ and there exists a smooth nonincreasing function $u_*$ on $(0,1]$ such that $u_n\to u_*$ in $C^2_{\text{loc}}((0,1])$ as $n\to \infty$. Moreover, $(\la_*,u_*)$ satisfies \eqref{q*}. In addition,  if $f(0)>0$, we get $u_*>0$ in $(0,1)$.
\item[(iii)] $\la_*=\infty$ and $(u_n)$ is uniformly bounded in $C_{\text{loc}}((0,1])$.
\end{enumerate}
In particular, for every sequence $(r_n)\subset (0,1)$ such that $u_n(r_n)\to \infty$, we have $r_n\to0$ as $n\to \infty$.
\end{lemma}
\begin{proof} First we follow the argument in the proof of Theorem 2.1 in \cite{OS}. From Lemma \ref{lem:po} with $r=1$ and \eqref{id0}, we get
\begin{equation}\label{po1}
\left(\int_0^1 \la_nf(u_n)rdr\right)^2=4\int_0^1 \la_nF(u_n(r))rdr.
\end{equation}
By the de l'H\^opital rule, (H1), and  our assumption $p>1$, we get
\[
\lim_{t\to \infty}\frac{F(t)}{f(t)}=\lim_{t\to \infty}\left(\frac{pt^{p-1}f(t)}{f'(t)}\frac1{pt^{p-1}}\right)=0.
\]
In particular, for any $\e>0$, there exists  a constant $t_\e>0$ such that $F(t)\le \e f(t)$ for any $t\ge t_\e$. Then using \eqref{po1}, we obtain 
\begin{equation}\label{po2}
\left(\int_0^1 \la_nf(u_n)rdr\right)^2\le 4\e \int_0^1 \la_nf(u_n(r))rdr+2\la_n\max_{0\le t\le t_\e}F(t).
\end{equation}
Hence, if  $\la_*=0$, we get  
\[
\limsup_{n\to \infty}\int_0^1 \la_nf(u_n)rdr=0.
\]
It follows from \eqref{id0} that  for any $r\in[0,1]$,  
\[
-ru_n'(r)\le \int_0^1 \la_nf(u_n)rdr\to0
\]
as $n\to \infty$. This implies that  $u_n\to 0$ in $C^1_{\text{loc}}((0,1])$ as $n\to \infty$. Then using the equation in \eqref{q}, we derive that $u_n\to 0$ in $C^2_{\text{loc}}((0,1])$. This proves (i). On the other hand, if $0<\la_*<\infty$,  \eqref{po2} implies that there exists a constant $C>0$ such that for any $r\in [0,1]$ and $n\in \mathbb{N}$,  
\[
-ru_n'(r)\le \int_0^1 \la_nf(u_n)rdr\le C.
\]
In particular, $(u_n)$ is uniformly bounded in $C^1_{\text{loc}}((0,1])$. Then using the Ascoli-Arzel\`a theorem and the equation, we obtain a smooth function $u_*$ in $(0,1]$ such that $u_n\to u_*$ in $C^2_{\text{loc}}(0,1]$ up to a subsequence. Then it is clear that $u_*$ satisfies the equation in \eqref{q*}, $u_*\ge0$ in $(0,1)$, and $u_*(1)=0$. Moreover, the monotonicity of $u_*$ obviously follows by that of $u_n$. Furthermore,  from \eqref{id2} and the Fatou lemma, we get for any $r\in(0,1)$, 
\begin{equation}\label{eq:*11}
u_*(r)\ge \int_0^r\la_* f(u_*)sds\log{\frac1r}.
\end{equation}
This ensures 
\[\int_0^1\la_* f(u_*)sds<\infty.\] 
In addition, assume $f(0)>0$. Then if there exists a point $r_0\in(0,1)$ such that $u_*(r_0)=0$, we get $u_*=0$ on $[r_0,1]$ from the monotonicity.  But then, \eqref{eq:*11} shows that for all $r\in(r_0,1]$,
\[
u_*(r)\ge \int_{r_0}^r\la_* f(u_*)sds\log{\frac1r}>0.
\]
This is a contradiction. Hence we have $u_*>0$ in $(0,1)$. This completes the case (ii). Lastly, if $\la_*=\infty$, we see from \eqref{po2} that 
\[
\int_0^1 f(u_n)rdr\to0
\]
as $n\to \infty$. Here recall that $f(t)$ is monotone increasing for all large $t>0$  and $f(t)\to \infty$ as $t\to \infty$ by Lemma \ref{lem:h}. Then noting  the monotonicity of $u_n(r)$ with respect to $r\in(0,1)$,  we get that $(u_n)$ is uniformly bounded in $C_{\text{loc}}((0,1])$. Moreover, using the previous conclusions, the final assertion is clearly proved. This finishes the proof.
\end{proof}
\begin{remark} In the last step of the proof of (iii) above, if we additionally assume $h(t)>0$ for all $t>0$, we get that $u_n\to 0$ locally uniformly in $(0,1]$.    
\end{remark}
\begin{remark} In the case of (ii), $u_*(r)$ may diverge to infinity as $r\to 0^+$. If it occurs, from \eqref{eq:*11} and the monotonicity of $f(t)$ and $u_*(r)$, we get that 
\[
u_*(r)\ge \frac12\la_*f(u_*(r))r^2\log{\frac1r}
\]
for all small $r\in(0,1)$. This and \eqref{ha} imply the upper bound,
\[
u_*(r)\le \left((2+o(1))\log{\frac1r}\right)^{\frac1p}
\]
for all small $r\in(0,1)$ where $o(1)\to0$ as $r\to0^+$. A similar curve will appear as the upper line of the oscillation of the blow-up solutions in the case $p>2$. See \eqref{de02} and \eqref{de11} in Theorem \ref{thm:osc}.  Analogous upper bounds for singular solutions are also found in  \cite{GGP2}. See, for instance, Corollary 7.7 there.  
\end{remark}
Here, we prove Theorem \ref{thm:gl}.
\begin{proof}[Proof of Theorem \ref{thm:gl}] The proof follows from Lemmas \ref{lem:gl} and \ref{kap}. We finish the proof. 
\end{proof}
\subsection{Definition of $a_k$ and $\delta_k$}
We lastly confirm that $\delta_k$ and $a_k$ are well-defined by \eqref{eq:del1} and \eqref{eq:del2} for any $k\in \mathbb{N}\cup\{0\}$.
\begin{lemma}\label{lem:bf} Let $p>2$, $a_0=2$, and $\delta_0=1$. Then for any $k\in \mathbb{N}$, $\delta_k\in (0,\delta_{k-1})$ and $a_k\in(0,2)$ are well-defined by \eqref{eq:del1} and \eqref{eq:del2}. 
\end{lemma}
\begin{proof} We argue by induction. It is clear that there exists a unique value $0<\delta_1<1$  such that   
\[
\frac p2(1-\delta_1)-1+\delta_1^p=0.
\]
Moreover, we readily get $\delta_1<(1/2)^{1/(p-1)}$. Then $a_1\in (0,2)$ is well-defined by \eqref{eq:del2}. This proves the case $k=1$. Next we suppose $\delta_k\in(0,\delta_{k-1})$ and $a_k\in(0,2)$ are well-defined by \eqref{eq:del1} and \eqref{eq:del2} for some $k\ge1$.  Then again it is easy to see that there exists a unique value $0<x_{k+1}<1$  such that
\[
\frac{2p}{2+a_k}(1-x_{k+1})-1+x_{k+1}^p=0.
\] 
Hence we can uniquely define the desired number $\delta_{k+1}\in (0,\delta_k)$ by $\delta_{k+1}=x_{k+1}\delta_k$. Since $x_{k+1}<\{2/(2+a_k)\}^{1/(p-1)}$, $a_{k+1}\in (0,2)$ is well-defined by \eqref{eq:del2}. We finish the proof.
\end{proof}
We use the next lemma in Section \ref{sec:con2}.
\begin{lemma}\label{lem:bg} Suppose $p>2$. For any $k\in \mathbb{N}\cup\{0\}$, let  $\delta_k$ and $a_k$ be constants defined as above and put $E_k=2a_k/\delta_k^{p-1}$. Then we have that
\[
\delta_k^{p-1}\sum_{i=0}^kE_i=2+a_k.
\]
\end{lemma}
\begin{proof} We argue by induction. As  $\delta_0^{p-1}E_0=4=2+a_0$, we prove the case $k=0$. We suppose the desired formula holds for some $k\ge0$. Then we get
\[
\delta_{k+1}^{p-1}\sum_{i=0}^{k+1}E_i=\delta_{k+1}^{p-1}E_{k+1}+\left(\frac{\delta_{k+1}}{\delta_k}\right)^{p-1}\delta_k^{p-1}\sum_{i=0}^{k}E_i=2+a_{k+1}
\]
by the definition of $E_{k+1}$, our assumption,  and \eqref{eq:del2}. This completes the proof. 
\end{proof}
\section{Standard bubble}\label{sec:con1} 
Let us start the proof of our main theorems. Throughout this section, we assume, without further comments, $p>0$, (H1), $\{(\la_n,\mu_n,u_n)\}$ is a sequence of solutions of \eqref{q} such that $\mu_n\to \infty$ as $n\to \infty$,  and $(\ga_{0,n})$, $(z_{0,n})$, and $z_0$ are the sequences and function defined in Theorem \ref{thm1}. 
\subsection{Proofs of Theorems \ref{thm1} and \ref{thm2}}
Let us give the proof of Theorems \ref{thm1} and \ref{thm2}. To this end, we begin with the next lemma.
\begin{lemma}\label{prop:thm1}
We have that, up  to a subsequence, $\ga_{0,n}\to 0$, $h(u_n(\ga_{0,n}\cdot))/h(\mu_n)\to1$ in $C_{\text{loc}}([0,\infty))$, and $z_{0,n}\to z_0$ in $C^1_{\text{loc}}([0,\infty))\cap C^2_{\text{loc}}((0,\infty))$ as $n\to \infty$. 
\end{lemma}
\begin{proof}  First we show $\ga_{0,n}\to0$ as $n\to \infty$. If this does not hold, then we may suppose $(\ga_{0,n})$ is bounded from below by a positive value up to a subsequence. Then from \eqref{q} and Lemma \ref{lem:h2}, we have that
\[
-(ru_n'(r))'=\frac{f(u_n(r))}{\ga_{0,n}^2 p\mu_n^{p-1}f(\mu_n)}r\le   r
\]
for all $r\in [0,1]$ and large  $n\in \mathbb{N}$. After integration with $u_n'(0)=0$, we deduce
\[
-u_n'(r)\le \frac r2
\]
for all $r\in[0,1]$ and large $n\in \mathbb{N}$. Integrating over $[0,1]$ with $u_n(1)=0$  gives that $u_n(0)$ is uniformly bounded for all $n\in \mathbb{N}$.  This is a contradiction.  Next by \eqref{q}, our choice of $\ga_{0,n}$, and Lemma \ref{lem:h2}, we see that
\begin{equation}\label{zn0}
-z_{0,n}''-\frac1r z_{0,n}'=\frac{f(u_n(\ga_{0,n} \cdot))}{f(\mu_n)}\le 1\text{ on }\left(0,1/\ga_{0,n}\right)
\end{equation}
and $z_{0,n}(0)=0=z_{0,n}'(0)$ for all large $n\in \mathbb{N}$. We claim that $(z_{0,n})$ is uniformly bounded in $C^2_{\text{loc}}([0,\infty))$. To prove this, multiplying the inequality \eqref{zn0} by $r$ and integrating it with $z_{0,n}'(0)=0$, we get
\[
0\le -z_{0,n}'(r)\le \frac{r}{2}
\]
for all $r\in [0,\infty)$ and large $n\in \mathbb{N}$. Then by the integration again with $z_{0,n}(0)=0$, we prove that $(z_{0,n})$ is uniformly bounded in $C^1_{\text{loc}}([0,\infty))$. Consequently, using the previous formula and \eqref{zn0}, we prove the claim.  
In particular, for any $K>0$, there exists a constant $C>0$ such that 
\[
\mu_n -\frac{C}{p\mu_n^{p-1}}\le u_n(\ga_{0,n} r)\le \mu_n\]
for all $r\in [0,K]$ and $n\in \mathbb{N}$. Hence using \eqref{hb}, we derive  
\[
\frac{h(u_n(\ga_{0,n} \cdot))}{h(\mu_n)}\to 1\text{ uniformly in }[0,K]
\]
as $n\to \infty$. Since $K>0$ is arbitrary, this proves the second assertion of the lemma. Finally, noting 
\[
-z_{0,n}''-\frac1r z_{0,n}'=\frac{h(u_n(\ga_{0,n} \cdot))}{h(\mu_n)}e^{\mu_n^p\left\{\left(1+\frac{z_{0,n}}{p\mu_n^p}\right)^p-1\right\}}\text{ on }(0,1/\ga_{0,n})
\]
and the Ascoli-Arzel\`{a} theorem, we get a function $z$ such that $z_{0,n}\to z$ in $C^1_{\text{loc}}([0,\infty))\cap C^2_{\text{loc}}((0,\infty))$ as $n\to \infty$ and $z$ satisfies
\[
\begin{cases}
-z''-\frac1r z'=e^z,\ z\le0 \text{ in }(0,\infty),\\
z(0)=0=z'(0),
\end{cases}
\]
up to a subsequence. Integrating the equation (see the last part of the proof of Lemma 4.3 in \cite{GN}), we deduce $z=z_0$. This finishes the proof.  
\end{proof}
Next, we obtain  the  energy and pointwise estimates. After this, we often use the notation defined by \eqref{defE}.
\begin{lemma}\label{lem21} 
There exists  a sequence $(\rho_{0,n})\subset (0,1)$ of values  such that $\rho_{0,n}\to0$, $\rho_{0,n}/\ga_{0,n}\to \infty$, $\mu_n^{-p/2}\log{(\rho_{0,n}/\ga_{0,n})}\to0$, $u_n(\rho_{0,n})/\mu_n\to1$, $\|z_{0,n}-z_0\|_{C^1{([0,\rho_{0,n}/\ga_{0,n}])}}\to0$, $\|h(u_n(\ga_{0,n}\cdot))/h(\mu_n)-1\|_{C([0,\rho_{0,n}/\ga_{0,n}])}\to 0$, and  
\[
E_n(0,\rho_{0,n})\to 4
\]
as $n\to \infty$ up to a subsequence. Moreover, we have 
\[
z_{0,n}(r)\le -(4+o(1))\log{r}
\]
for all $r\in[\rho_{0,n}/\ga_{0,n},1/\ga_{0,n}]$ and $n\in \mathbb{N}$ where $o(1)\to 0$ as $n\to \infty$ uniformly for all $r$ in the interval. \end{lemma}
\begin{proof} 
First we take a subsequence if necessary so that  all the assertions in the previous lemma hold. Then, we can choose a sequence $(\tilde{R}_n)$ of positive values  so that $\tilde{R}_n \ga_{0,n}\to0$, $\tilde{R}_n\to \infty$,  $\mu_n^{-p/2}\log{\tilde{R}_n}\to 0$, $\|z_{0,n}-z_0\|_{C^1([0,\tilde{R}_n])}\to0$, $\|h(u_n(\ga_{0,n}\cdot))/h(\mu_n)-1\|_{C([0,\tilde{R}_n])}\to 0$, $|\tilde{R}_nz_{0,n}'(\tilde{R}_n)-\tilde{R}_n z_0'(\tilde{R}_n)|\to0$, and 
\[
\begin{split}
E(0,\ga_{0,n}\tilde{R}_n)
&=\int_0^{\tilde{R}_n}\frac{h(u_n(\ga_{0,n} r))}{h(\mu_n)}e^{\mu_n^p\left\{\left(1+\frac{z_{0,n}(r)}{p\mu_n^p}\right)^p-1\right\}}rdr\to 4
\end{split}
\]
as $n\to \infty$. Then since $z_{0,n}(\tilde{R}_n)=z_0(\tilde{R}_n)+o(1)$, we also get that  $u_n(\ga_{0,n}\tilde{R}_n)/\mu_n\to1$ as $n\to \infty$. Furthermore, for any $r\in[\tilde{R}_n,1/\ga_{0,n}]$, we obtain 
\[
rz_{0,n}'(r)\le \tilde{R}_nz_{0,n}'(\tilde{R}_n)=\tilde{R}_nz_{0}'(\tilde{R}_n)+o(1)=-4+o(1).
\] 
Multiplying this by $1/r$ and integrating it,  we see
\[
\begin{split}
z_{0,n}(r)&\le z_{0}(\tilde{R}_n)-(4+o(1))\log{\frac{r}{\tilde{R}_n}}+o(1)\le -(4+o(1))\log{r}
\end{split}
\] 
for all $r\in [\tilde{R}_n,1/\ga_{0,n}]$. Putting $\rho_{0,n}=\ga_{0,n}\tilde{R}_n$ for all $n\in \mathbb{N}$, we prove all the assertions. 
 We finish the proof. 
\end{proof}
The next estimate will be used when we deduce the oscillation property. 
\begin{lemma}\label{lem:thm1} Assume $(\rho_{0,n})$ is a sequence obtained in Lemma \ref{lem21} and take any sequences $(r_n)\subset (0,\rho_{0,n}]$ and $(R_n)\subset (0,\infty)$  such that $r_n=R_n \ga_{0,n}$ for all $n\in \mathbb{N}$. Then we have that
\[
p \la_n r_n^2u_n(r_n)^{p-1}f(u_n(r_n))=\frac{64R_n^2}{(8+R_n^2)^2}(1+o(1))
\]
as $n\to \infty$.  In particular, we obtain that $\phi_n(\rho_{0,n})\to0$ as $n\to \infty$. 
\end{lemma}
\begin{proof}  It follows from Lemma \ref{lem21} that  
\[
\begin{split}
p \la_nr_n^2 u_n(r_n)^{p-1}f(u_n(r_n))&=\left(\frac{u_n(r_n)}{\mu_n}\right)^{p-1}\left(\frac{r_n}{\ga_{0,n}}\right)^2\frac{f(u_n(r_n))}{f(\mu_n)}\\
&=(1+o(1))R_n^2e^{z_0(R_n)+o(1)}.\\
\end{split}
\]
This shows the former  assertion. Moreover, noting  $\rho_{0,n}/\ga_{0,n}\to \infty$ as $n\to \infty$, we confirm the latter one.  This completes the proof.
\end{proof}
Now we prove our first theorem. 
\begin{proof}[Proof of Theorem \ref{thm1}] The former assertions are the consequences of Lemmas  \ref{prop:thm1} and \ref{lem21}. The latter ones are proved by Lemma \ref{lem:thm1} and the facts that 
\[
\max_{R>0}\frac{64R^2}{(8+R^2)^2}=2
\]
and the maximum is attained by $R=2 \sqrt{2}$. We finish the proof. 
\end{proof}
We proceed to the proof of Theorem \ref{thm2}. 
\begin{lemma}\label{lem210} We get
\[
\limsup_{n\to \infty}\frac{\log{\frac1{\ga_{0,n}}}}{\mu_n^p}\le\frac p4
\]
up to a subsequence. 
\end{lemma}
\begin{proof}
From Lemma \ref{lem21}, we deduce
\[
\begin{split}
-p\mu_n^p&=z_{0,n}\left(1/\ga_{0,n}\right)\le -(4+o(1))\log{\frac1{\ga_{0,n}}}
\end{split}
\]
up to a subsequence. This shows the desired formula. We finish the proof. 
\end{proof}
Let us complete the energy estimate in the case $p<2$. 
\begin{lemma}\label{lem:thm21} Suppose $p<2$. Then we have
\[
\lim_{n\to \infty}E_n(0,1)=4
\]
up to a subsequence. 
\end{lemma}
\begin{proof}
Let $(\rho_{0,n})$ be a sequence in Lemma \ref{lem21}. Then   it suffices to show 
\[
\lim_{n\to \infty}E_n(\rho_{0,n},1)=0.
\]
From Lemma \ref{lem210}, the definition of $\ga_{0,n}$, and \eqref{ha}, we derive that $\la_n\le e^{-(1-p/2+o(1))\mu_n^p}$. Noting this, we choose a constant $\e=\{(1-p/2)/2\}^{1/p}$ and a sequence $(r_n)\subset(\rho_{0,n},1)$ so that $u_n(r_n)=\e \mu_n$ for all $n\in \mathbb{N}$. Then we get from Lemma \ref{lem:h} that
\[
E_n(r_n,1)\le p\mu_n^{p-1} e^{-(1-p/2+o(1))\mu_n^p}\left(f(u_n(r_n))+\max_{0\le t\le t_1}f(t)\right)\to0
\]
as $n\to \infty$ where $t_1>0$ is the constant in Lemma \ref{lem:h}. Hence the rest of the proof is to confirm that $E_n(\rho_{0,n},r_n)\to0$ as $n\to \infty$. To conclude the proof, we note that there exists a value $M>0$ such that $z_{0,n}(\rho_{0,n}/\ga_{0,n})=z_0(\rho_{0,n}/\ga_{0,n})+o(1)\le -M$ for all $n\in \mathbb{N}$ by Lemma \ref{lem21}.  It follows that  $-p\mu_n^p(1-\e)\le z_{0,n}(r)<-M$ for all $r\in[\rho_{0,n}/\ga_{0,n},r_n/\ga_{0,n}]$ and $n\in \mathbb{N}$. Then from Lemma \ref{lem:h3}, there exists  a sequence $(\e_n)$ of values  such that $\e_n\to0$ as $n\to \infty$ and, writing  $R_n=\rho_{0,n}/\ga_{0,n}$ and $R_n'=r_n/\ga_{0,n}$, 
\[
\begin{split}
E_n(\rho_{0,n},r_n)&=\int_{R_n}^{R_n'}e^{\mu_n^p\left\{\left(1+\frac{z_{0,n}}{p\mu_n^p}\right)^p-1\right\}+\e_n z_{0,n}}rdr
\end{split}
\]
for all $n\in \mathbb{N}$. Here, if $p \ge1 $, using $0\le 1+z_{0,n}/(p\mu_n^p)\le1$, we get
\[
\mu_n^p\left\{\left(1+\frac{z_{0,n}}{p\mu_n^p}\right)^p-1\right\}\le z_{0,n}/p.
\]
On the other hand, if $0<p<1$, there exists a constant $\theta\in(0,1)$ such that
\[
\mu_n^p\left\{\left(1+\frac{z_{0,n}}{p\mu_n^p}\right)^p-1\right\}=z_{0,n}-\left(1+\frac{\theta z_{0,n}}{p\mu_n^p}\right)^{p-2} \frac{(1-p)z_{0,n}^2}{2p\mu_n^p} \le z_{0,n}. 
\]
Hence, setting $c_p=1/p$ if $1\le p<2$ and $c_p=1$ if $0<p<1$, we obtain from Lemma \ref{lem21} that  
\[
\begin{split}
E_n(\rho_{0,n},r_n)&\le \int_{R_n}^{R_n'}e^{(c_p+\e_n) z_{0,n}}rdr\le \int_{R_n}^{R_n'}e^{-4c_p(1+o(1)) \log{r}}rdr\\
&\to0
\end{split}
\]
since $4c_p>2$ and $R_n\to \infty$ as $n\to \infty$. This finishes the proof. 
\end{proof}
As  a result, we can improve  Lemma \ref{lem210}.
 \begin{lemma}\label{lem:thm22} Assume $p<2$. Then we obtain
\[
\lim_{n\to \infty}\frac{\log{\frac1{\ga_{0,n}}}}{\mu_n^p}=\frac p4
\]
up to a subsequence.
\end{lemma}
\begin{proof}
In view of Lemma \ref{lem210}, it suffices to ensure that
\[
\liminf_{n\to \infty}\frac{\log{\frac1{\ga_{0,n}}}}{\mu_n^p}\ge \frac p4.
\]
Then we use \eqref{id1} with $s=1$ and Lemma \ref{lem:thm21} to get 
\[
\begin{split}
p\mu_n^p&=p\mu_n^{p-1}\int_0^1\la_n f(u_n)r\log{\frac1r}dr\\
&=(4+o(1))\log{\frac1{\ga_{0,n}}}+\int_0^{\frac1{\ga_{0,n}}}\frac{f(u_n(\ga_{0,n}r)}{f(\mu_n)}r\log{\frac1r}dr.
\end{split}
\]
Here choosing large $R>1$, we obtain by Lemma \ref{prop:thm1} that
\begin{equation}\label{eq:7}
\begin{split}
\int_0^{\frac1{\ga_{0,n}}}&\frac{f(u_n(\ga_{0,n}r))}{f(\mu_n)}r\log{\frac1r}dr\\
&=\int_0^R e^{z_0(r)} r \log\frac1r dr-\int_R^{\frac1{\ga_{0,n}}}\frac{f(u_n(\ga_{0,n}r))}{f(\mu_n)}r\log{r}dr+o(1)\\
&<0
\end{split}
\end{equation}
for all large $n\in \mathbb{N}$. Using this estimate for the previous formula, we get
\[
p\mu_n^p\le (4+o(1))\log{\frac1{\ga_{0,n}}}
\]
for all large $n\in \mathbb{N}$. This gives the desired conclusion. 
We complete the proof.
\end{proof}
Let us show Theorem \ref{thm2}.
\begin{proof}[Proof of Theorem \ref{thm2}] First, \eqref{sub1} is proved by  Lemma \ref{lem:thm21}. Then since 
\[
\begin{split}
\int_0^{1/\ga_{0,n}}\left(1+\frac{z_{0,n}(r)}{p\mu_n^p}\right)&\frac{h(u_n(\ga_{0,n} r))}{h(\mu_n)}e^{\mu_n^p\left\{\left(1+\frac{z_{0,n}(r)}{p\mu_n^p}\right)^p-1\right\}}rdr\\
&=p\mu_n^{p-2}\int_0^1 \la_n u_nf(u_n)rdr\le p\mu_n^{p-1}\int_0^1 \la_n f(u_n)rdr,
\end{split}
\]
we get \eqref{sub10} by Lemma \ref{prop:thm1}, the Fatou lemma, and \eqref{sub1}. Next, from Lemma \ref{lem:thm22}, the definition of $\ga_{0,n}$, and \eqref{ha}, we calculate 
\[
\frac{p}{2}=\frac{\log{(p\la_n \mu_n^{p-1}f(\mu_n))}}{\mu_n^p}+o(1)=-\frac{\log{\frac1{\la_n}}}{\mu_n^p}+1+o(1).
\]
This shows  \eqref{sub2}. Finally, fix any $r_0\in(0,1)$. Since  
\[
E_n(0,\rho_{0,n})\le E_n(0,r_0)\text{ and }  E_n(r_0,1)\le E_n(\rho_{0,n},1) 
\]
for all large $n\in \mathbb{N}$, Lemmas \ref{lem21} and \ref{lem:thm21} imply that 
\begin{equation}\label{eq:ll}
E_n(0,r)\to 4\text{\ and }E(r,1)\to0 \text{ uniformly for all $r\in[r_0,1]$.} 
\end{equation}
Now, for any $r\in [r_0,1]$, using \eqref{id2}, we have
\[
p\mu_n^{p-1}u_n(r)=E_n(0,r)\log{\frac1{r}}+p\mu_n^{p-1}\int_r^1 \la_nf(u_n)s\log{\frac1s}ds.
\]
Then noting 
\[
p\mu_n^{p-1}\int_r^1 \la_nf(u_n)s\log{\frac1s}ds\le E_n(r,1)\log{\frac1r}
\]
and \eqref{eq:ll}, we get
\begin{equation}\label{pqp}
p\mu_n^{p-1}u_n(r)\to 4\log{\frac1{r}}\text{  uniformly for all $r\in[r_0,1]$ }
\end{equation}
as $n\to \infty$. Moreover, \eqref{id0} and \eqref{eq:ll} give
\[
-p\mu_n^{p-1}u_n'(r)=\frac{E_n(0,r)}{r}\to\frac 4r \text{  uniformly for all $r\in[r_0,1]$ }
\]
 as $n\to \infty$. In addition, from the equation in \eqref{q} and \eqref{id0} again, we see
\begin{equation}\label{ppq}
p\mu_n^{p-1}u_n''(r)=-\frac{E_n(0,r)}{r^2}-\mu_n^{p-1}\la_n f(u_n(r)).
\end{equation}
Here, if $1\le p<2$, since $(u_n)$ is uniformly bounded  in $[r_0,1]$ by \eqref{pqp}, it follows from \eqref{sub2} that there exists a value $C>0$ such that
\[
\mu_n^{p-1}\la_n f(u_n(r))\le Ce^{-(1-p/2+o(1))\mu_n^p}
\]
for all $r\in[r_0,1]$ and $n\in \mathbb{N}$. On the other hand, if $0<p<1$, we note  $u_n(r_0)=p^{-1}\mu_n^{1-p}\{4\log{(1/r_0)}+o(1))\}$ by \eqref{pqp}. Then it follows from Lemma \ref{lem:h2}, \eqref{ha}, and \eqref{sub2}  again that, if $n\in \mathbb{N}$ is large enough, 
\[
\mu_n^{p-1}\la_n f(u_n(r))\le \mu_n^{p-1}\la_n f(u_n(r_0))  \le e^{-(1-p/2+o(1))\mu_n^p}
\]
for all $r\in[r_0,1]$. Using these two estimates and \eqref{eq:ll} for \eqref{ppq}, we complete the proof of \eqref{sub3}. We finish the proof.
\end{proof}
\subsection{Basic estimates for  $p\ge  2$}
We finish this section by giving some more estimates which will be useful for our study in the case $p \ge 2$. We begin with the next lemma.   
\begin{lemma}\label{lem31} Assume $(r_n)\subset (0,1)$, $(\delta_n)\subset (0,1)$, and $\delta\in[0,1]$ are any sequences of values and constant such that  $u_n(r_n)=\delta_n \mu_n$ for all $n\in \mathbb{N}$, $r_n/\ga_{0,n}\to \infty$, $\delta_n\to \delta$ as $n\to \infty$, and $\liminf_{n\to \infty}u_n(r_n)>t_1$ where $t_1$ the number in Lemma \ref{lem:h}. Then putting $R_n=r_n/\ga_{0,n}$,  we have 
\[
\frac{p \mu_n^p}{E(0,r_n)}\left(1-\delta_n\right) \le \log{R_n} \le \frac{\mu_n^p}2\left(1-\delta^p+o(1)\right)
\]
for all $n\in \mathbb{N}$ up to a subsequence where $o(1)\to0$ as $n\to \infty$. In particular, if $\delta\not=1$, we get
\[
-E(0,r_n)\log{R_n}\le z_n(R_n)\le-2p\left(\frac{1-\delta}{1-\delta^p}+o(1)\right)\log{R_n}
\]
for all $n\in \mathbb{N}$  where $o(1)\to \infty$ as $n\to \infty$.
\end{lemma}
\begin{proof} 
Using \eqref{id1}, we obtain
\[
\begin{split}
u_n(0)-u_n(r_n)
&=\log{R_n}\int_{0}^{r_n} \la_nf(u_n(r))rdr+\int_{0}^{r_n} \la_nf(u_n(r))r\log{\frac{\ga_{0,n}}{r}}dr.
\end{split}
\]
Multiplying this by $p\mu_n^{p-1}$, we get 
\[
\begin{split}
p\mu_n^p(1-\delta_n)&= E_n(0,r_n)\log{R_n}+\int_{0}^{R_n} \frac{f(u_n(\ga_{0,n}r))}{f(\mu_n)}r\log{\frac{1}{r}}dr\\
&\le E_n(0,r_n)\log{R_n}
\end{split}
\]
for all large $n\in \mathbb{N}$ where noting $R_n\to \infty$ as $n\to \infty$ we estimated the second term on the right-hand side of the first equality similarly to \eqref{eq:7}. It follows that 
\begin{equation}\label{eq:1k}
\frac{p\mu_n^p(1-\delta_n)}{E_n(0,r_n)}\le \log{R_n}
\end{equation}
for all large $n\in \mathbb{N}$. On the other hand, we see from \eqref{id1} and the first assertion in Lemma \ref{lem:h} with our assumption  that there exists a constant $c>0$ such that 
\[
\begin{split}
p\mu_n^p(1-\delta+o(1))
&=r_n^2 p\mu_n^{p-1}\la_nf(u_n(r_n))\int_{0}^{1} \frac{f(u_n( r_n r))}{f(u_n(r_n))}r\log{\frac{1}{r}}dr\\
&\ge c r_n^2 p\mu_n^{p-1}\la_nf(u_n(r_n))=c R_n^2 \frac{f(u_n(r_n))}{f(\mu_n)}
\end{split}
\]
for all large $n\in \mathbb{N}$. Here using \eqref{ha}, our assumption, and the fact that $h(t)>0$ for all $t>t_1$, we get $\mu_n^{-p}\log{h(u_n(r_n))}\to 0$
as $n\to \infty$. It follows that 
\[
\begin{split}
p\mu_n^p(1-\delta+o(1))&\ge c R_n^2 e^{-\mu_n^p\left(1-\delta^p+o(1)\right)}
\end{split}
\]
for all large $n\in \mathbb{N}$. 
This implies that
\[
\log{R_n}\le \frac{\mu_n^p(1-\delta^p+o(1))}{2}
\]
for all large $n\in \mathbb{N}$. This and \eqref{eq:1k} prove the first formula of this lemma. Then noting
\[
z_n(R_n)=p\mu_n^{p-1}(u_n(r_n)-\mu_n)=-p\mu_n^p(1-\delta_n),
\]
the second conclusion readily follows from the first one. This finishes the proof.
\end{proof}
Using the previous estimates, we shall prove that more energy appears on the outside of the first concentration interval $(0,\rho_{0,n})$ if $p>2$. Moreover, we can improve the pointwise estimate in Lemma \ref{lem21} in this case.
\begin{lemma}\label{lem32}  
Let $(r_n)$, $(\delta_n)$, and $\delta$ be as in the previous lemma and $\delta\not=1$. Then, up to a subsequence,  we get  
\[
\liminf_{n\to \infty} E(0,r_n)\ge2p\left(\frac{1-\delta}{1-\delta^p}\right)
\]
and 
\[
z_n(r)\le -2p\left(\frac{1-\delta}{1-\delta^p}+o(1)\right)\log{r}
\]
for any $r\in [r_n/\ga_{0,n},1/\ga_{0,n}]$ and all $n\in \mathbb{N}$ where $o(1)\to0$ as $n\to \infty$. 
\end{lemma}
\begin{proof} Take the sequence $(\rho_{0,n})$ in Lemma \ref{lem21}. Since $\delta\not=1$, we have that $\rho_{0,n}<r_n$ for all large $n\in \mathbb{N}$. Then we get from \eqref{id2}  that 
\[
\begin{split}
u_n(\rho_{0,n})-u_n(r_n)
&\le  \left(\log{\frac{r_n}{\rho_{0,n}}}\right)\int_{0}^{r_n} \la_nf(u_n(r))rdr
\end{split}
\]
for all large $n\in \mathbb{N}$. Then putting $R_n=r_n/\ga_{0,n}$ for all $n\in \mathbb{N}$, we estimate by the previous lemma and our choice of $(\rho_{0,n})$ that
\[
\begin{split}
1-\delta+o(1)&\le \left(\log{R_n}+\log{\frac{\ga_{0,n}}{\rho_{0,n}}}\right)\frac{E_n(0,r_n)}{p\mu_n^p}
\le \frac{(1-\delta^p+o(1))}{2p}E_n(0,r_n)
\end{split}
\]
for all large $n\in \mathbb{N}$.
This readily proves the first formula. Next again set $R_n=r_n/\ga_{0,n}$. Then for any $r\in [R_n,1/\ga_{0,n}]$, we have
\[
rz_n'(r)\le R_nz_n'(R_n)=-E_n(0,r_n)
\]
by \eqref{id0}. Multiplying the inequality by $1/r$ and integrating it, we derive from the latter assertion in the previous lemma and the former one of this lemma that if $r\in [R_n,1/\ga_{0,n}]$, then
\[
\begin{split}
z_n(r)&=z_n(R_n)-E_n(0,r_n)\log{\frac{r}{R_n}}
\le -2p\left(\frac{1-\delta}{1-\delta^p}+o(1)\right)\log{r}
\end{split}
\]
for all large $n\in \mathbb{N}$. This proves the second conclusion. We finish the proof.
\end{proof}
Lastly, we show that  \eqref{sub2} fails in the case $p>2$ as follows.
\begin{lemma}\label{lem33} Assume $p>1$. Then we have  that, up to a subsequence, 
\[
\limsup_{n\to \infty}\frac{\log{\frac1{\ga_{0,n}}}}{\mu_n^p}\le\frac12
\]
which implies that 
\[
\liminf_{n\to \infty}\frac{\log{\frac1{\la_n}}}{\mu_n^p}\ge0.
\]
\end{lemma}
\begin{proof} For any small $\e\in(0,1)$, we choose a constant $\delta\in(0,1)$ so that $1-\e=(1-\delta)/(1-\delta^p)$. Then take a sequence $(r_n)\subset (0,1)$ so that $u_n(r_n)=\delta \mu_n$ for all $n\in\mathbb{N}$. It follows from Lemma \ref{lem32} that 
\[
z_n(r)\le -2p(1-\e+o(1))\log{r}
\]
for all $r\in [r_n/\ga_{0,n},1/\ga_{0,n}]$ and all $n\in \mathbb{N}$ up to a subsequence. Then setting $r=1/\ga_{0,n}$, we get  
\[
-p\mu_n^p\le -2p (1-\e+o(1))\log{\frac1{\ga_{0,n}}}.
\]
This implies that 
\[
\limsup_{n\to \infty}\frac{\log{\frac1{\ga_{0,n}}}}{\mu_n^p}\le \frac1{2(1-\e)}.
\]
Since this holds for any small $\e\in(0,1)$, we get the first assertion. Using the previous estimate, the definition of $\ga_{0,n}$, and \eqref{ha}, we deduce
\[
\frac1{1-\e}+o(1)\ge \frac{\log{(p\la_n\mu_n^{p-1}f(\mu_n))}}{\mu_n^p}=-\frac{\log{\frac1{\la_n}}}{\mu_n^p}+1+o(1).
\]
It follows that 
\[
\liminf_{n\to \infty}\frac{\log{\frac1{\la_n}}}{\mu_n^p}\ge -\left(\frac{1}{1-\e}-1\right). 
\]
Since again this is true for all small $\e\in(0,1)$, we obtain the second assertion. We complete the proof.
\end{proof}

\section{Infinite sequence of bubbles}\label{sec:con2}
The main aim of this section is to prove Theorems \ref{thm:crit}, \ref{thm30}, and \ref{thm3}. In particular, we shall detect a sequence of bubbles in the case $p>2$. Throughout this section, we always assume $p\ge2$ and (H1) without further comments. Recall  the sequences $(\phi_n)$ and $(\psi_n)$ of functions  defined in Section \ref{sub:kt}.  
\subsection{Preliminaries}
We begin with some preliminaries. Let $(\ga_{0,n})$ and $(\rho_{0,n})$ be sequences of values in Theorem \ref{thm1}.  We first give the next lemma which will be used for the proof of \eqref{sup3}. 
\begin{lemma}\label{lem34} Up to  a subsequence, there exists a constant  $\nu\in[0,1]$ such that
\[
\lim_{n\to \infty}\frac{\log{\frac1{\ga_{0,n}}}}{\mu_n^p}=\frac \nu2.
\]
Moreover, if $(r_n)\subset (0,1)$ and $\delta\in [0,1]$ are a sequence and a value such that $r_n/\ga_{0,n}\to \infty$, $u_n(r_n)/\mu_n\to \delta$ as $n\to \infty$, and $\liminf_{n\to \infty}u_n(r_n)>t_1$ where $t_1$ is the constant in Lemma \ref{lem:h}, then we get 
\[
\lim_{n\to \infty}\frac{\log{\frac1{r_n}}}{\mu_n^p}\ge\frac{-1+\delta^p+\nu}2
\]
up to a subsequence. The equality holds if $\lim_{n\to \infty}(\mu_n^{-p}\log{\phi_n(r_n)})=0$. 
\end{lemma}
\begin{proof} The former assertion is clearly follows from  the fact $\ga_{0,n}\to 0$ as $n\to \infty$ and Lemma \ref{lem33}. For the latter ones, we put $R_n=r_n/\ga_{0,n}$ for all $n\in \mathbb{N}$.  Then noting 
\begin{equation}\label{eq:rn}
\frac{\log{\frac1{r_n}}}{\mu_n^p}=\frac{-\log{R_n}+\log{\frac{1}{\ga_{0,n}}}}{\mu_n^p},
\end{equation}
we get from Lemma \ref{lem31} that, up to a subsequence, there exists a constant $\nu \in[0,1]$ such that 
\[
\begin{split}
  \frac{-(1-\delta^p)+\nu+o(1)}{2}\le \frac{\log{\frac1{r_n}}}{\mu_n^p}\le \frac{\nu+o(1)}{2}
\end{split}
\]
for all $n\in \mathbb{N}$. Hence, taking a subsequence again if necessary, we deduce 
\[
\lim_{n\to \infty}\frac{\log{\frac1{r_n}}}{\mu_n^p}\ge \frac{-1+\delta^p+\nu}{2}.
\]
This is the second conclusion. Finally, assume $\mu_n^{-p}\log{\phi_n(r_n)}\to0$ as $n\to \infty$. Then noting \eqref{ha} and our choice of $(r_n)$, we obtain
\[
\phi_n(r_n)=R_n^2\left(\frac{u_n(r_n)}{\mu_n}\right)^{p-1}\frac{f(u_n(r_n))}{f(\mu_n)}=R_n^2 e^{-\mu_n^p\left(1-\delta^p+o(1)\right)}.
\]
It follows that 
\[
\log{R_n}=\frac{\mu_n^p}{2}(1-\delta^p+o(1)).
\]
Using this for \eqref{eq:rn}, we obtain
\[
\lim_{n\to \infty}\frac{\log{\frac1{r_n}}}{\mu_n^p}=\frac{-1+\delta^p+\nu}2.
\]
This completes the proof.
\end{proof}
We will also use the next one. 
\begin{lemma}\label{lem:22} Let $1\ge \delta>\delta'>0$  and $(r_n),(s_n)\subset (0,1)$ be any constants and sequences such that $r_n/\ga_{0,n}\to \infty$, $u_n(r_n)/\mu_n\to\delta$, and $u_n(s_n)/\mu_n\to \delta'$ as $n\to \infty$. Then we have
\[
\liminf_{n\to \infty}\frac{\log{\frac{s_n}{r_n}}}{\mu_n^p}\ge \frac{p(\delta-\delta')}{\limsup_{n\to \infty}E_n(0,s_n)}.
\] 
up to a subsequence. Moreover, assume $\lim_{n\to \infty}(\mu_n^{-p}\log{\phi_n(r_n)})=0$. Then  we get
\[
\lim_{n\to \infty}\frac{\log{\frac{s_n}{r_n}}}{\mu_n^p}\le \frac{\delta^p-\delta'^p}{2}
\]
up to a subsequence. The equality holds if we additionally suppose that $\lim_{n\to \infty}(\mu_n^{-p}\log{\phi_n(s_n)})=0$. 
\end{lemma}
\begin{proof} From \eqref{id2}, we get
\[
\begin{split}
u_n(r_n)-u_n(s_n)&\le \left(\log{\frac{s_n}{r_n}}\right)\int_{0}^{s_n} \la_nf(u_n)rdr.
\end{split}
\]
It follows that
\[
p\mu_n^p(\delta-\delta'+o(1))\le\left(\log{\frac{s_n}{r_n}}\right)E_n(0,s_n).
\]
This  gives the  first assertion.  Next, noting the previous lemma, we choose a subsequence and a constant $\nu\in[0,1]$ so that  
\[
\lim_{n\to \infty}\frac{\log{\frac1{\ga_{0,n}}}}{\mu_n^p}=\frac{\nu}2.
\]
Then using the previous lemma again, if $\lim_{n\to \infty}(\mu_n^{-p}\log{\phi_n(r_n)})=0$, we get that 
\[
\frac{\log{\frac{s_n}{r_n}}}{\mu_n^p}= \frac{\log{\frac{1}{r_n}}-\log{\frac{1}{s_n}}}{\mu_n^p}\le \frac{-1+\delta^p+\nu}2-\frac{-1+\delta'^p+\nu}2+o(1)=\frac{\delta^p-\delta'^p}2+o(1).
\]
This proves the second conclusion. Clearly, we ensure the equality in the formula above  if we assume in addition $\lim_{n\to \infty}(\mu_n^{-p}\log{\phi_n(s_n)})=0$. 
 We finish the proof.
\end{proof}
\subsection{Bubbles with singular limit profiles}
Now let us start our main discussion.  Recall the definitions of $\delta_k$, $a_k$ in \eqref{eq:del1} and \eqref{eq:del2}, $E_k$ in Lemma \ref{lem:bg}, and $E_n(s,t)$ in \eqref{defE}. Moreover for any $k\in \mathbb{N}$, we  introduce the next condition (A$_k$).
\begin{enumerate}
\item[(A$_k$)]  
If $k=1$, we define sequences $(\ga_{0,n})$ of values  so that 
\[
p\la_n \mu_n^{p-1}f(\mu_n)\ga_{0,n}^2=1
\]
and $(z_{0,n})$ of functions  by 
\[
z_{0,n}(r)=p \mu_n^{p-1}(u_n(\ga_{0,n}r)-\mu_n)
\] 
for all $r\in[0,1/\ga_{0,n}]$. Then there exists a sequence $(r_{0,n})\subset (0,1)$ such that $r_{0,n}/\ga_{0,n}\to 2\sqrt{2}$, $u_n(r_{0,n})/\mu_n\to1$, and $\phi_n(r_{0,n})\to2$ as $n\to \infty$. Moreover, there exists a sequence $(\rho_{0,n})\subset (0,1)$ of values  such that $\rho_{0,n}/\ga_{0,n}\to \infty$, $\mu_n^{-p/2}\log{(\rho_{0,n}/\ga_{0,n})}\to0$, $u_n(\rho_{0,n})/\mu_n\to1$, $E_n(0,\rho_{0,n})\to 4$, $\phi_n(\rho_{0,n})\to0$ as $n\to \infty$, 
and 
\begin{equation}\label{as:z0}
z_{0,n}(r)\le -(4+o(1))\log{r}
\end{equation}
for all $r\in[\rho_{0,n}/\ga_{0,n},1/\ga_{0,n}]$ and $n\in \mathbb{N}$ where $o(1)\to0$ as $n\to \infty$ uniformly for all $r$ in the interval. If $k\ge2$, there exists a sequence $(r_{k-1,n})\subset (0,1)$ of values  such that $u_n(r_{k-1,n})/\mu_n\to \delta_{k-1}$ and $\phi_n(r_{k-1,n})\to a_{k-1}^2/2$ as $n\to \infty$, and further, if we put sequences $(\ga_{k-1,n})$ of positive values  so that 
\[
p\la_n u_n(r_{k-1,n})^{p-1}f(u_n(r_{k-1,n}))\ga_{k-1,n}^2=1
\]
and $(z_{k-1,n})$ of functions  by 
\[
z_{k-1,n}(r)=p u_n(r_{k-1,n})^{p-1}(u_n(\ga_{k-1,n}r)-u_n(r_{k-1,n}))
\]
for all $r\in[0,1/\ga_{k-1,n}]$ and $n\in \mathbb{N}$, then there exists a sequence $(\rho_{k-1,n})\subset (0,1)$ of values  such that  
  $\rho_{k-1,n}/r_{k-1,n}\to \infty$, $\mu_n^{-p/2}\log{(\rho_{k-1,n}/r_{k-1,n})}\to0$, $u_n(\rho_{k-1,n})/\mu_n\to\delta_{k-1}$, 
\[
E_n(0,\rho_{k-1,n})\to\sum_{j=0}^{k-1}E_j, 
\]
$\phi_n(\rho_{k-1,n})\to0$ as $n\to \infty$, and
\begin{equation}\label{as:zk}
z_{k-1,n}(r)\le -(2+a_{k-1}+o(1))\log{r}
\end{equation}
for all $r\in[\rho_{k-1,n}/\ga_{k-1,n},1/\ga_{k-1,n}]$ and $n\in \mathbb{N}$  where $o(1)\to0$ as $n\to \infty$ uniformly for all $r$ in the interval. 
\end{enumerate}
\begin{remark}\label{rmk:A} We remark that the sequences $(\rho_{0,n})$ and $(r_{0,n})$ in Lemmas \ref{lem21}, \ref{lem:thm1}, and Theorem \ref{thm1} complete the conditions in  (A$_1$). On the other hand, if (A$_k$) is satisfied for some  $k\ge1$, we have that $r_{k-1,n}/\ga_{k-1,n}=\sqrt{\phi_n(r_{k-1,n})}\to a_{k-1}/\sqrt{2}$ as $n\to \infty$. 
Especially, this and the last assertion in Lemma \ref{lem:gl} imply  $\ga_{k-1,n}\to0$ as $n\to \infty$.  
\end{remark}
After this, we always assume, in addition to the basic assumptions noted in the first paragraph of this section, that (A$_k$) holds true for some  $k\in\mathbb{N}$. If $p=2$, we only consider the case $k=1$ and put $\delta_1=0$. We first prove  the next lemma which corresponds to the former  formula in Lemma \ref{lem32} if $k=1$.
\begin{lemma}\label{lem:e1} 
Let $(r_n)\subset (\rho_{k-1,n},1)$, $(\delta_n)\subset (0,\delta_{k-1})$, and $\delta\in(0,\delta_{k-1})$ be sequences and a value such that $u_n(r_n)=\delta_n \mu_n$ for all $n\in \mathbb{N}$ and $\delta_n\to \delta$ as $n\to \infty$. Then we have
\[
\liminf_{n\to \infty}E_n(\rho_{k-1,n},r_n)\ge 2p\left\{ \frac{1-(\delta/\delta_{k-1})}{1-(\delta/\delta_{k-1})^p}-\frac{2+a_{k-1}}{2p}\right\}\frac1{\delta_{k-1}^{p-1}}.
\]
\end{lemma}
\begin{remark} Since $2+a_{k-1} \le 4$, we see that additional energy must appear in $(\rho_{k-1,n},r_n)$ if $p>2$ and $\delta>0$ is sufficiently small.
\end{remark}
\begin{proof} 
From \eqref{id2}, we get
\[
\begin{split}
u_n(\rho_{k-1,n})-u_n(r_n)
&\le \left(\log{\frac{r_n}{\rho_{k-1,n}}}\right)\int_{0}^{r_n} \la_nf(u_n)rdr. 
\end{split}
\]
It follows that
\begin{equation}\label{eq:a1}
\begin{split}
\delta_{k-1}-\delta+o(1)&\le
\left(\log{\frac{r_n}{\ga_{k-1,n}}}+\log{\frac{\ga_{k-1,n}}{\rho_{k-1,n}}}\right)\frac{E_n(0,r_n)}{p\mu_n^p}.
\end{split}
\end{equation}
On the other hand,  from \eqref{id1} and Lemma \ref{lem:h}, there exists a constant $c>0$ such that 
\[
\begin{split}
p\mu_n^p(1-\delta+o(1))
&=r_n^2p\mu_n^{p-1}\la_nf(u_n(r_n))\int_{0}^{1} \frac{f(u_n( r_n r))}{f(u_n(r_n))}r\log{\frac{1}{r}}dr\\
&\ge c r_n^2p\mu_n^{p-1}\la_nf(u_n(r_n))=\frac{c+o(1)}{\delta_{k-1}^{p-1}} \left(\frac{r_n}{\ga_{k-1,n}}\right)^2 \frac{f(u_n(r_n))}{f(u_n(r_{k-1,n}))}\\
&=\left(\frac{r_n}{\ga_{k-1,n}}\right)^2 e^{-\mu_n^p\left(\delta_{k-1}^p-\delta^p+o(1)\right)}
\end{split}
\]
by \eqref{ha}. It follows that
\[
\log{\left(\frac{r_n}{\ga_{k-1,n}}\right)}\le \frac{\delta_{k-1}^p-\delta^p+o(1)}{2}\mu_n^p.
\]
Using this for \eqref{eq:a1}, we have
\[
\delta_{k-1}-\delta+o(1)\le\frac{\delta_{k-1}^p-\delta^p+o(1)}{2p}E_n(0,r_n)
\]
where we noted (A$_k$) and Remark \ref{rmk:A}. Then, we derive
\[
E_n(0,r_n)\ge 2p \left(\frac{\delta_{k-1}-\delta}{\delta_{k-1}^p-\delta^p}\right)+o(1).
\]
Consequently, we obtain
\[
\begin{split}
E_n(\rho_{k-1,n},r_n)
&\ge 2p \left(\frac{\delta_{k-1}-\delta}{\delta_{k-1}^p-\delta^p}\right)-\frac{2+a_{k-1}}{\delta_{k-1}^{p-1}}+o(1)
\end{split}
\]
by (A$_k$) and Lemma \ref{lem:bg}. This proves the desired formula. We complete the proof.
\end{proof}
By next two  lemmas, we will find the maximal interval $(\rho_{k-1,n},\sig_n)$ where no additional bubble appears, that is,  $E_n(\rho_{k-1,n},\sig_n)\to0$ as $n\to \infty$. We begin with the next rough calculation.
\begin{lemma}\label{lem:e10} Let $\delta\in(0,\delta_{k-1})$ be any constant such that
\[
\delta>\frac{p-2}{p-1}\delta_{k-1}\text{ and }\left(\frac{\delta}{\delta_{k-1}}\right)^p>1-\frac{4pa_{k-1}}{(p-1)(2+a_{k-1})^2} 
\] 
and take a sequence $(r_n)\subset (\rho_{k-1,n},1)$ such that $u_n(r_n)=\delta \mu_n$ for all $n\in \mathbb{N}$. Then we have
\[
\lim_{n\to \infty}E_n(\rho_{k-1,n},r_n)=0.
\]
\end{lemma}
\begin{proof} We put $\mu_{k-1,n}=u_n(r_{k-1,n})$ for all $n\in \mathbb{N}$. From Lemma \ref{lem:22}, (A$_k$), and Remark \ref{rmk:A}, we get that 
\begin{equation}\label{eq:11}
\begin{split}
\log{\frac{r_n}{\ga_{k-1,n}}}&=\log{\frac{r_n}{\rho_{k-1,n}}}+\log{\frac{\rho_{k-1,n}}{\ga_{k-1,n}}}
\le \frac{1-(\delta/\delta_{k-1})^p+o(1)}{2}\mu_{k-1,n}^p.
\end{split}
\end{equation}
Moreover, by \eqref{as:zk}, there exists  a value $M>0$ such that 
\[
-p\left(1-\frac{\delta}{\delta_{k-1}}+o(1)\right)\mu_{k-1,n}^p\le z_{k-1,n}(r)\le -M
\]
for all $r\in[\rho_{k-1,n}/\ga_{k-1,n},r_n/\ga_{k-1,n}]$ and $n\in \mathbb{N}$. Then  putting $\tilde{r}_n:=r_n/\ga_{k-1,n}$ and $\tilde{\rho}_n:=\rho_{k-1,n}/\ga_{k-1,n}$ for all $n\in \mathbb{N}$ and using  Lemma \ref{lem:h3}, we get 
\[
\begin{split}
E_n(\rho_{k-1,n},r_n)&=(\delta_{k-1}^{1-p}+o(1))\int_{\tilde{\rho}_n}^{\tilde{r}_n}\frac{f(u_n(\ga_{k-1,n}r))}{f(\mu_{k-1,n})}rdr\\
&=
(\delta_{k-1}^{1-p}+o(1))\int_{\tilde{\rho}_n}^{\tilde{r}_n}e^{\mu_{k-1,n}^p\left\{\left(1+\frac{z_{k-1,n}}{p \mu_{k-1,n}^p}\right)^p-1\right\}+o\left(z_{k-1,n}\right)}rdr\\
&\le (\delta_{k-1}^{1-p}+o(1)) \int_{\tilde{\rho}_n}^{\tilde{r}_n}e^{(1+o(1))z_{k-1,n}+\frac{p-1}{2p \mu_{k-1}^p}z_{k-1,n}^2}rdr\\
&\le (\delta_{k-1}^{1-p}+o(1))
\int_{\tilde{\rho}_n}^{\tilde{r}_n}e^{-(2+a_{k-1}+o(1))\log{r}+\frac{(p-1)(2+a_{k-1}+o(1))^2}{2p\mu_{k-1,n}^p}\log^2{r}}rdr
\end{split}
\]
for all large $n\in \mathbb{N}$ where for the last inequality we used the former condition on $\delta$ and our assumptions \eqref{as:z0} and \eqref{as:zk}.  Changing the variable with $s=\log{r}$ and putting $\bar{r}_n:=\log{\tilde{r}_n}$, $\bar{\rho}_n:=\log{\tilde{\rho}_n}$, and 
\[
K_n:=\sqrt{\frac{(p-1)(2+a_{k-1}+o(1))^2}{2p\mu_{k-1,n}^p}},
\]
we get
\[
\begin{split}
&E_n(\rho_{k-1,n},r_n)=O\left(e^{-\left(\frac{a_{k-1}+o(1)}{2K_n}\right)^2}\int_{\bar{\rho}_n}^{\bar{r}_n}e^{\left(K_ns-\frac{a_{k-1}+o(1)}{2K_n}\right)^2}ds\right).
\end{split}
\]
Changing the variable again with $t=K_ns-(a_{k-1}+o(1))(2K_n)^{-1}$ and setting $\hat{r}_n=K_n\bar{r}_n-(a_{k-1}+o(1))(2K_n)^{-1}$ and $\hat{\rho}_n=K_n\bar{\rho}_n-(a_{k-1}+o(1))(2K_n)^{-1},$
we have
\begin{equation}\label{eq:ij}
\begin{split}
E_n(\rho_{k-1,n},r_n)&=O\left(\mu_{k-1,n}^{p/2}e^{-\left(\frac{a_{k-1}+o(1)}{2K_n}\right)^2}\int_{\hat{\rho}_n}^{\hat{r}_n}e^{t^2}dt\right).
\end{split}
\end{equation}
Note that  (A$_k$) and \eqref{eq:11} imply  $\hat{\rho}_n, \hat{r}_n=O(\mu_n^{p/2})$. Now, we choose a constant $\e_0>0$ so small that
\[
\frac{pa_{k-1}^2}{2(p-1)(2+a_{k-1})^2}-\e_0>0
\]
and put
\[
I_n=\mu_{k-1,n}^{p/2}e^{-\left(\frac{a_{k-1}+o(1)}{2K_n}\right)^2}\int_{[\hat{\rho}_n,\hat{r}_n]\cap \{|t|\le \sqrt{\e_0} \mu_{k-1,n}^{p/2}\}}e^{t^2}dt
\]
and
\[
J_n=\mu_{k-1,n}^{p/2}e^{-\left(\frac{a_{k-1}+o(1)}{2K_n}\right)^2}\int_{[\hat{\rho}_n,\hat{r}_n]\cap \{|t|> \sqrt{\e_0} \mu_{k-1,n}^{p/2}\}}e^{t^2}dt.
\] 
From the definition of $K_n$ and our choice of $\e_0>0$, we obtain  
\[
I_n\le \mu_{k-1,n}^{p/2}e^{-\left(\frac{p(a_{k-1}+o(1))^2}{2(p-1)(2+a_{k-1}+o(1))^2}-\e_0\right)\mu_{k-1,n}^{p}}(|\hat{\rho}_n|+|\hat{r}_n|)\to0
\]
as $n\to \infty$. On the other hand, since $|t|> \sqrt{\e_0} \mu_{k-1,n}^{ p/2}$ implies $1<|t|/(\sqrt{\e_0}\mu_{k-1,n}^{p/2})$, we get
\[
\begin{split}
J_n&=O\left(e^{-\left(\frac{a_{k-1}+o(1)}{2K_n}\right)^2}\int_{\hat{\rho}_n}^{\hat{r}_n}|t|e^{t^2}dt\right)\\
&=O\left(e^{-\left(\frac{a_{k-1}+o(1)}{2K_n}\right)^2+\hat{r}_n^2}\right)+O\left(e^{-\left(\frac{a_{k-1}+o(1)}{2K_n}\right)^2+\hat{\rho}_n^2}\right)+o(1).
\end{split}
\]
Here we calculate
\[
\begin{split}
&e^{-\left(\frac{a_{k-1}+o(1)}{2K_n}\right)^2+\hat{r}_n^2}\\
&=\exp{\left[-(a_{k-1}+o(1))\left(1-\frac{(p-1)(2+a_{k-1}+o(1))^2}{2p(a_{k-1}+o(1))\mu_{k-1,n}^p}\log{\tilde{r}_n}\right)\log{\tilde{r}_n}\right]}\\
&\to0
\end{split}
\]
as $n\to \infty$ by \eqref{eq:11} and the latter condition on $\delta$. On the other hand, from (A$_k$), we similarly compute 
\[
\begin{split}
&e^{-\left(\frac{a_{k-1}+o(1)}{2K_n}\right)^2+\hat{\rho}_n^2}
\to0
\end{split}
\]
as $n\to \infty$. Hence we prove $J_n\to0$ as $n\to \infty$. Using these conclusions for \eqref{eq:ij}, we deduce $E_n(\rho_{k-1,n},r_n)\to0$ as $n\to \infty$. This finishes the proof.
\end{proof}
We further extend the interval with no additional bubble. 
\begin{lemma}\label{lem:e2} 
There exists  a sequence $(\sig_n)\subset (\rho_{k-1,n},1)$ such that $u_n(\sig_n)/\mu_n\to \delta_k$, $\phi_n(\sig_n)\to0$, and
\[
E_n(\rho_{k-1,n},\sig_n)\to0
\]
as $n\to \infty$ up  to a subsequence. 
\end{lemma}
\begin{proof} We again put $u_n(r_{k-1,n})=\mu_{k-1,n}$ for all $n\in \mathbb{N}$. Choose a constant $\bar{\delta}\in(\delta_k,\delta_{k-1})$ and a sequence $(r_n)\subset (0,1)$ so that
\[
\bar{\delta}>\frac{p-2}{p-1}\delta_{k-1}\text{,\ \  } \left(\frac{\bar{\delta}}{\delta_{k-1}}\right)^p>1-\frac{4pa_{k-1}}{(p-1)(2+a_{k-1})^2},
\] 
and $u_n(r_n)=\bar{\delta} \mu_n$ for all $n\in \mathbb{N}$. Then the previous lemma shows  $E_n(\rho_{k-1,n},r_n)=o(1)$. Moreover for any $\delta\in(\delta_k,\bar{\delta})$, take a sequence $(s_n)\subset (r_n,1)$ so that $u_n(s_n)=\delta \mu_n$ for all $n\in \mathbb{N}$. Then  we get 
\begin{equation}\label{eq:a10}
\begin{split}
\frac{p(1-\bar{\delta}/\delta_{k-1}+o(1))}{2+a_{k-1}}&\le \frac{\log{\frac{r_n}{\ga_{k-1,n}}}}{\mu_{k-1,n}^p}\le \frac{\log{\frac{s_n}{\ga_{k-1,n}}}}{\mu_{k-1,n}^p}\le\frac{1-(\delta/\delta_{k-1})^p+o(1)}{2}
\end{split}
\end{equation}
for all $n\in\mathbb{N}$. We shall show the first and last inequalities. First, (A$_k$) implies  $E_n(0,r_n)= \sum_{i=0}^{k-1}E_i+o(1)$. Then, using Lemma \ref{lem:22} and (A$_k$), we get 
\[
\begin{split}
\log{\frac{r_n}{\ga_{k-1,n}}}&=
\log{\frac{r_n}{\rho_{k-1,n}}}+\log{\frac{\rho_{k-1,n}}{\ga_{k-1,n}}}
\ge\frac{p(1-\bar{\delta}/\delta_{k-1}+o(1))}{\delta_{k-1}^{p-1}\sum_{i=0}^{k-1}E_i}\mu_{k-1,n}^p
\end{split}
\]
for all $n\in \mathbb{N}$. This and Lemma \ref{lem:bg} prove the first inequality. Next, from the second formula in Lemma \ref{lem:22} with (A$_k$) and Remark \ref{rmk:A}, we have
\[
\begin{split}
\log{\frac{s_n}{\ga_{k-1,n}}}&=\log{\frac{s_n}{\rho_{k-1,n}}}+\log{\frac{\rho_{k-1,n}}{\ga_{k-1,n}}}\le \frac{\delta_{k-1}^p-\delta^p+o(1)}{2}\mu_n^p
\end{split}
\]
for all $n\in \mathbb{N}$. This shows the last one. Now if  $r\in[r_n,s_n]$, putting $R=r/\ga_{k-1,n}$ and noting Lemma \ref{lem:h3}, \eqref{as:z0}, \eqref{as:zk}, and \eqref{eq:a10},  we get 
\begin{equation}\label{eq:a2}
\begin{split}
\phi_n(r)&=R^2\frac{u_n(r)^{p-1}f_n(u_n(r))}{\mu_{k-1,n}^{p-1}f_n(\mu_{k-1,n})}\\
&\le\exp{\left[\mu_{k-1,n}^p\left\{\left(1+\frac{z_{k-1,n}(R)}{p\mu_{k-1,n}^p}\right)^p-1+o\left(\frac{z_{k-1,n}(R)}{\mu_{k-1,n}^p}\right)+\frac{2\log{R}}{\mu_{k-1,n}^p}\right\}\right]}\\
&= \exp{\left[\mu_{k-1,n}^p\left\{\left(1-\frac{2+a_{k-1}}{p}\frac{\log{R}}{\mu_{k-1,n}^p}\right)^p-1+\frac{2\log{R}}{\mu_{k-1,n}^p}+o(1)\right\}\right]}.
\end{split}
\end{equation}
Here we set a function
\[
\zeta(x)= \left(1-\frac{2+a_{k-1}}{p}x\right)^p-1+2x
\]
for all $x\in [0,p/(2+a_{k-1})]$. Then we readily see that $\zeta(0)=0$ and there exists a number $0<x^*<p/(2+a_{k-1})$ such that $\zeta(x)$ is decreasing for all $0<x<x^*$ and increasing for all $x^*<x<p/(2+a_{k-1})$. Moreover, we claim 
\[
\zeta\left(\frac{1-(\delta/\delta_{k-1})^p}{2}\right)<0.
\]
Indeed, a direct calculation shows that the formula above is equivalent to
\[
\frac{2p}{2+a_{k-1}}\left(1-\frac{\delta}{\delta_{k-1}}\right)-1+\left(\frac{\delta}{\delta_{k-1}}\right)^p<0.
\] 
Then we can readily confirm that this formula is true by  \eqref{eq:del1} and our choice $\delta_k<\delta<\delta_{k-1}$.
This shows the claim. Consequently, there exists a constant $\e>0$ such that 
\[
\zeta(x)\le -2\e\text{\ \  for all \ \ }\frac{p(1-\bar{\delta}/\delta_{k-1})}{2+a_{k-1}}\le x\le\frac{1-(\delta/\delta_{k-1})^p}{2}.
\]
Using this for \eqref{eq:a2} with \eqref{eq:a10}, we get
\begin{equation}\label{eq:a4}
\sup_{r\in[r_n,s_n]}\phi_n(r)\le e^{-\e \mu_{k-1,n}^p} 
\end{equation}
for all large $n\in \mathbb{N}$. Therefore, we obtain
\[
\begin{split}
E_n(r_n,s_n)&=\int_{r_n}^{s_n}\left(\frac{\mu_n}{u_n(r)}\right)^{p-1}\frac{\phi_n(r)}{r}dr\le \frac{e^{-\e \mu_{k-1,n}^p}}{\delta^{p-1}}\log{\frac{s_n}{r_n}}
\end{split}
\]
for all large $n\in \mathbb{N}$. Then since \eqref{eq:a10} implies
\[
\log{\frac{s_n}{r_n}}=\log{\frac{s_n}{\ga_{k-1,n}}}-\log{\frac{r_n}{\ga_{k-1,n}}}=O(\mu_{k-1,n}^p),
\]
we conclude that
\[
\lim_{n\to \infty}E_n(r_n,s_n)=0.
\]
Lastly, recalling that $\delta\in(\delta_k,\bar{\delta})$ is arbitrary,  we find a sequence $(\sig_n)\subset (r_n,1)$ such that $u_n(\sig_n)/\mu_n\to \delta_k$, $\phi_n(\sig_n)\to0$, and $E_n(r_n,\sig_n)\to0$ as $n\to \infty$.
 This completes the proof.  
\end{proof}
Before proceeding to the next step, we here prove Theorem \ref{thm:crit}. For this purpose, we give the next lemma.
\begin{lemma}\label{lem:crit1} Assume $p=2$.  Then, up to a subsequence,  there exists a sequence $(r_n)\subset (0,1)$ such that $u_n(r_n)/\mu_n\to 0$ and  $E_n(0,r_n)\to4$ as $n\to \infty$.
\end{lemma}
\begin{proof} As noted in Remark \ref{rmk:A}, we confirm that (A$_1$) holds true. Then, using Theorem \ref{thm1} and Lemma \ref{lem:e2} and recalling that $\delta_1=0$ if $p=2$, we get the desired conclusion. We finish the proof.
\end{proof}
Moreover,  we get the following. 
\begin{lemma}\label{lem:crit2} Suppose $p=2$. Then we have, up to a subsequence,
\[
\lim_{n\to \infty}\frac{\log{\frac{1}{\ga_{0,n}}}}{\mu_n^2}=\frac12.
\]
\end{lemma}
\begin{proof} In view of Lemma \ref{lem33}, it suffices to show that 
\[
\liminf_{n\to \infty}\frac{\log{\frac{1}{\ga_{0,n}}}}{\mu_n^2}\ge \frac12.
\]
Then, we select a sequence $(r_n)\subset(0,1)$ as in the previous lemma and compute from \eqref{id1} that
\[
\begin{split}
2\mu_n^2(1+o(1))&\le 2\mu_n\int_{0}^{r_n} \la_nf(u_n(r))r\log{\frac{1}{r}}dr\\
&=(4+o(1))\log{\frac{1}{\ga_{0,n}}}+\int_0^{\frac{r_n}{\ga_{0,n}}}\frac{f(u_n(\ga_{0,n}r))}{f(\mu_n)}r\log{\frac1r}dr.
\end{split}
\]
Since $r_n/\ga_{0,n}\to \infty$ as $n\to \infty$, as in the proof of Lemma \ref{lem:thm22}, we get
\[
2\mu_n^2(1+o(1))\le (4+o(1))\log{\frac{1}{\ga_{0,n}}}
\]
for all large $n\in \mathbb{N}$. This gives the desired estimate. We complete the proof.
\end{proof}
Now, we prove the theorem.
\begin{proof}[Proof of Theorem \ref{thm:crit}] Lemma \ref{lem:crit1} proves \eqref{cri1}. Moreover, using Lemma \ref{lem:crit2} and the definition of $\ga_{0,n}$, we confirm \eqref{cri2}  similarly to the proof of \eqref{sub2} in Theorem \ref{thm2}.  This completes the proof.
\end{proof}
After this, we focus on the case $p>2$. 
\begin{lemma}\label{lem:e3} Assume $p>2$. After extracting a subsequence if necessary, we choose the sequence  $(\sig_n)$ obtained in Lemma \ref{lem:e2} and any value  $e_0\in \mathbb{R}$ such that
\[
0<e_0<\min\left\{\frac{a_k}{\delta_k^{p-1}},2p\left(1-\frac{2+a_{k-1}}{2p}\right)\frac1{\delta_{k-1}^{p-1}}\right\}.
\]
Then there exist a sequence $(\tau_n)\subset (\sig_n,1)$ and a value $\tilde{\delta}_k\in(0,\delta_k]$ such that $u_n(\tau_n)/\mu_n\to\tilde{\delta}_k$, $E_n(\sig_n,\tau_n)\to e_0$ as $n\to \infty$,  
\begin{equation}\label{eq:st}
\lim_{n\to \infty}\sup_{r\in[\sig_n,\tau_n]}\psi_n(r)<2,
\end{equation}
and further,  there exists a constant $\eta>0$ such that $\phi_n(\tau_n)\to\eta$ as $n\to \infty$ up to a subsequence. Finally, we have that $\tau_n/\sig_n\to \infty$ as $n\to \infty$. 
\end{lemma}
\begin{proof} We first select a value $\delta\in(0,\delta_k)$ so that 
\[
2p\left( \frac{1-(\delta/\delta_{k-1})}{1-(\delta/\delta_{k-1})^p}-\frac{2+a_{k-1}}{2p}\right)\frac1{\delta_{k-1}^{p-1}}> e_0.
\]
Take a sequence  $(t_n)\subset (\sigma_n,1)$ so that $u_n(t_n)=\delta\mu_n$ for all $n\in \mathbb{N}$. Then from Lemmas \ref{lem:e1} and \ref{lem:e2}, we get 
\[
E_n(\sig_n,t_n)\ge 2p\left( \frac{1-(\delta/\delta_{k-1})}{1-(\delta/\delta_{k-1})^p}-\frac{2+a_{k-1}}{2p}\right)\frac1{\delta_{k-1}^{p-1}}+o(1)> e_0
\]
for all large $n$.  Hence there exist a sequence $(\tau_n)\subset (\sigma_n,t_n]$ and a value $\tilde{\delta}_k\in[\delta,\delta_k]$ such that $E_n(\sig_n,\tau_n)\to e_0$ and $u_n(\tau_n)/\mu_n\to \tilde{\delta}_k$ as $n\to \infty$ up to a subsequence. It follows from (A$_k$) that
\[
\begin{split}
\sup_{r\in[\sig_n,\tau_n]}\psi_n(r)&\le (\delta_k+o(1))^{p-1}E_n(0,\tau_n)=\delta_k^{p-1}\left(e_0+\sum_{i=0}^{k-1}E_i\right)+o(1).
\end{split}
\] 
Then  recalling our choice of $e_0$, Lemma \ref{lem:bg}, and \eqref{eq:del2}, we obtain
\[
\begin{split}
\lim_{n\to \infty}\sup_{r\in[\sig_n,\tau_n]}\psi_n(r)
<a_k+\left(\frac{\delta_k}{\delta_{k-1}}\right)^{p-1}(2+a_{k-1})=2
\end{split}
\] 
up to a subsequence. This proves \eqref{eq:st}. Moreover, we claim that $\phi_n(\tau_n)$ is uniformly bounded for all $n\in \mathbb{N}$. In fact, from \eqref{eq:st} and the first assertion in Lemma \ref{lem:h}, we obtain for all large $n\in \mathbb{N}$ that
\[
2> pu_n(\tau_n)^{p-1}\int_0^{\tau_n}\la_n f(u_n)rdr\ge \frac12\phi_n(\tau_n).
\]
This proves the claim. Hence by extracting a subsequence if necessary we find a constant $\eta\ge0$ such that $\phi_n(\tau_n)\to \eta$ as $n\to \infty$. It follows that $\eta>0$. We suppose $\eta=0$ on the contrary. From \eqref{eq:st} and Lemma \ref{lem:pp}, $\phi_n(r)$ is increasing for all $r\in[\sig_n,\tau_n]$ if $n$ is large enough. Then our assumption $\eta=0$ implies that $\sup_{r\in[\sig_n,\tau_n]}\phi_n(r)\to0$ as $n\to \infty$. Then it follows from Lemma \ref{lemD} that $E_n(\sig_n,\tau_n)\to0$ as $n\to \infty$ up to a subsequence. This is a contradiction. 
 Lastly from the first conclusion in Lemma \ref{lem:h} again, we confirm that
\[
o(1)=\frac{\phi_n(\sig_n)}{\phi_n(\tau_n)}=\left(\frac{\sig_n}{\tau_n}\right)^2\frac{u_n(\sig_n)^{p-1}f(u_n(\sig_n))}{u_n(\tau_n)^{p-1}f(u_n(\tau_n))}\ge \left(\frac{\sig_n}{\tau_n}\right)^2.
\]
This shows the final assertion. We complete the proof. 
\end{proof}
Now we detect the next bubble.
\begin{lemma}\label{lem:f1}  
Suppose $p>2$. Let $(\sig_n),(\tau_n)\subset (\rho_{k-1,n},1)$, $\tilde{\delta}_k\in(0,\delta_k]$, and $\eta>0$ be the sequences and constants in the previous lemma.  Then there exists a sequence $(r_{k,n})\subset (\sigma_n,1)$ such that $u_n(r_{k,n})/\mu_n\to \tilde{\delta}_k$, $\phi_n(r_{k,n})\to \tilde{a}_k^2/2$, and $\psi_n(r_{k,n})\to2$ as $n\to \infty$ up to a subsequence where $\tilde{a}_k=\sqrt{(2-\lim_{n\to \infty}\psi_n(\tau_n))^2+2\eta}\in(0,2)$. Moreover, if we put $\ga_{k,n}=r_{k,n}/\sqrt{\phi_n(r_{k,n})}$ and  
\[
z_{k,n}(r)=pu_n(r_{k,n})^{p-1}(u(\ga_{k,n} r)-u_n(r_{k,n}))
\] 
for all $r\in [0,1/\ga_{k,n}]$ and $n\in \mathbb{N}$, then we have that $\ga_{k,n}\to0$ and there exists a function $z_k$ such that $z_{k,n}\to z_k$ in $C^2_{\text{loc}}((0,\infty))$ as $n\to \infty$ up to a subsequence where $z_{k}$ satisfies
\[
\begin{cases}
-z_k''-\frac{1}{r}z_k'=e^{z_k}\ \ \text{ in }(0,\infty),\\
z_k(\tilde{a}_k/\sqrt{2})=0,\ -(\tilde{a}_k/\sqrt{2})z_k'(\tilde{a}_k/\sqrt{2})=2,
\end{cases}
\]
which implies
\[
z_k(r)=\log{\frac{2\tilde{a}_k^2 \tilde{b}_k }{r^{2-\tilde{a}_k}(1+\tilde{b}_k r^{\tilde{a}_k})^2}}
\]
with $\tilde{b}_k=(\sqrt{2}/\tilde{a}_k)^{\tilde{a}_k}$.
\end{lemma}
\begin{proof} 
First notice that  $\tau_n\to0$ as $n\to \infty$ by the last assertion in Lemma \ref{lem:gl}. Then putting $m_n=u_n(\tau_n)$, we  define 
\[
w_n(r)=pm_n^{p-1}(u_n(\tau_nr)-m_n)
\]
for all $r\in[0,1/\tau_n]$. It follows from \eqref{q} that $w_n$ satisfies
\begin{equation}\label{wn}
\begin{cases}
-w_n''-\frac{1}{r}w_n'=\phi_n(\tau_n)\frac{f(u_n(\tau_n\cdot))}{f(m_n)}=\phi_n(\tau_n)\frac{h(u_n(\tau_n\cdot))}{h(m_n)}e^{m_n^p\left\{\left(1+\frac{w_n}{pm_n^p}\right)^p-1\right\}}\text{ in }(0,1/\tau_n),\\
w_n(1)=0,\ w_n'(1)=-\psi_n(\tau_n).
\end{cases}
\end{equation}
 Here,  from (A$_k$) and Lemma \ref{lem:e2}, we get
\[
\begin{split}
 \psi_n(\tau_n)\ge (\tilde{\delta}_k+o(1))^{p-1}E_n(0,\sig_{n})=\tilde{\delta}_k^{p-1}\sum_{i=0}^{k-1}E_i+o(1).
\end{split}
\] 
Hence by \eqref{eq:st}, we have  
\begin{equation}\label{eq:ps0}
\psi_0:=\lim_{n\to \infty}\psi_n(\tau_n)\in(0,2)
\end{equation}
up to a subsequence. Now we shall show the uniform estimates for $(w_n)$. To this end, first fix any $r_0\in(0,1)$. We then claim that, 
\begin{equation}\label{wn'0}
\left(\tilde{\delta}_k^{p-1}\sum_{i=1}^{k-1} E_i+o(1)\right)\frac1r\le -w_n'(r)\le (\psi_0+o(1))\frac1r
\end{equation}
for all $r\in [r_0,1]$ and large $n\in \mathbb{N}$ where $o(1)\to0$ uniformly for all $r\in[r_0,1]$. To prove this, we confirm by Lemma \ref{lem:h} and (A$_k$) that 
\[
\left(\frac{\rho_{k-1,n}}{\tau_n}\right)^2\le \frac{\phi_n(\rho_{k-1,n})}{\phi_n(\tau_n)}=o(1).
\]
In particular, $\tau_nr_0>\rho_{k-1,n}$ for all large $n\in \mathbb{N}$. This implies that for all $r\in [r_0,1]$,  
\[
-rw_n'(r)\ge -r_0w_n'(r_0)
\ge  (\tilde{\delta}_k+o(1))^{p-1}E_n(0,\rho_{k-1,n})=\tilde{\delta}_k^{p-1}\sum_{i=1}^{k-1} E_i+o(1)
\]
if $n\in \mathbb{N}$ is sufficiently large. This shows the first inequality of \eqref{wn'0}. On the other hand, for any $r\le 1$, we confirm by \eqref{id0} that  
\[
\begin{split}
-rw_n'(r)\le pm_n^{p-1}\int_0^{\tau_n}\la_nf(u_n)sds=\psi_n(\tau_n)
\end{split}
\] 
for all $n\in \mathbb{N}$. This proves the second one. Hence we prove the claim. Then for any $r\in[r_0,1]$, integrating \eqref{wn'0} over $[r,1]$ and using $w_n(1)=0$, we get 
\begin{equation}\label{wn0}
\left(\tilde{\delta}_k^{p-1}\sum_{i=1}^{k-1} E_i+o(1)\right)\log{\frac1r} \le w_n(r)\le (\psi_0+o(1))\log{\frac1r}
\end{equation}
for all $n\in \mathbb{N}$. Next suppose $r\in[1,1/\tau_n]$. From the equation in \eqref{wn} and Lemma \ref{lem:h2}, we have 
\[
(-rw_n'(r))'\le (\eta+1) r
\]
for all large $n\in \mathbb{N}$. After integration over $[1,r]$ with $w_n'(1)=-\psi_n(\tau_n)$, we calculate
\begin{equation}\label{wn'inf}
0\le -w_n'(r)\le (\eta+1) \frac{r^2-1}{2r}+\psi_n(\tau_n)\frac1r
\end{equation}
for all $r\in[1,1/\tau_n]$ and large $n\in \mathbb{N}$. Integrating over $[1,r]$ with $w_n(1)=0$, we have
\begin{equation}\label{wninf}
0\le -w_n(r)\le (\eta+1)\left(\frac{r^2-1}{4}-\frac12\log{r}\right)+\psi_n(\tau_n)\log{r}
\end{equation}
for all $r\in[1,1/\tau_n]$ and large $n\in \mathbb{N}$. Consequently, \eqref{wn'0}, \eqref{wn0}, \eqref{wn'inf}, and \eqref{wninf} prove that $(w_n)$ is bounded in $C_{\text{loc}}^1((0,\infty))$. Then by the Ascoli-Arzel\`a theorem, we prove that there exists a function $w$ on $(0,\infty)$ such that $w_n\to w$ in $C_{\text{loc}}((0,\infty))$ as $n\to \infty$ by extracting a subsequence if necessary. \eqref{wn0} yields
\begin{equation}\label{sw}
\left(\tilde{\delta}_k^{p-1}\sum_{i=0}^{k-1}E_i\right)\log{\frac1r}\le w(r)\le \psi_0\log{\frac1r}
\end{equation}
for all $r\in(0,1)$.  Moreover, note that for each compact set $K\subset (0,\infty)$, there exists a constant $C>0$ such that $\sup_{r\in K}|w_n(r)|\le C$ for all $n\in \mathbb{N}$. It follows that  
\[
m_n-\frac{C}{pm_n^{p-1}}\le u_n(\tau_nr)\le m_n+\frac{C}{pm_n^{p-1}}
\]
for all $r\in K$ and $n\in \mathbb{N}$. Then \eqref{hb} proves that
\begin{equation}\label{eq:hh}
\frac{h(u_n(\tau_n\cdot))}{h(m_n)}\to1\text{ uniformly on }K
\end{equation}
as $n\to \infty$. Consequently, we get from the equation in \eqref{wn} and the Ascoli-Arzel\`a theorem  that $w_n\to w$ in $C^2_{\text{loc}}((0,\infty))$ up to a subsequence and thus, 
\begin{equation}\label{w}
\begin{cases}
-w''-\frac{1}{r}w'=\eta e^w\text{ in }(0,\infty),\\
w(1)=0,\ w'(1)=-\psi_0.
\end{cases}
\end{equation}
Now, put $z_{n}(r)=w_n(r/\sqrt{\phi_n(\tau_n)})$ for all $r\in(0,\sqrt{\phi_n(\tau_n)}/\tau_n)$ and $z(r)=w(r/\sqrt{\eta})$ for all $r\in (0,\infty)$. Then we clearly get that $z_{n}\to z$ in $C^2_{\text{loc}}((0,\infty))$ as $n\to \infty$. Moreover, it follows from \eqref{w} and \eqref{sw} that  
\begin{equation}\label{eq:z}
\begin{cases}
-z''-\frac{1}{r}z'= e^{z}\text{ in }(0,\infty),\\
z(\sqrt{\eta})=0,\ \sqrt{\eta}z'(\sqrt{\eta})=-\psi_0,
\end{cases}
\end{equation}
and 
\begin{equation}\label{eq:sz}
\left(\tilde{\delta}_k^{p-1}\sum_{i=0}^{k-1}E_i \right)\log{\frac{\sqrt{\eta}}r}\le z(r)\le \psi_0\log{\frac{\sqrt{\eta}}r}
\end{equation}
for all $r\in(0,\sqrt{\eta})$.  Integrating the equation in \eqref{eq:z} (see Proof of Proposition 3.1 in \cite{GGP}) shows that there exist constants $a,b>0$ such that   
\[
z(r)=\log{\frac{2a^2 b }{r^{2-a}(1+b r^{a})^2}}
\]
for all $r>0$.  \eqref{eq:sz} and \eqref{eq:ps0} yield $a\in(0,2)$.  Using $z(\sqrt{\eta})=0$ and $\sqrt{\eta}z'(\sqrt{\eta})=-\psi_0$, we derive $a=\sqrt{(2-\psi_0)^2+2\eta}$ and 
\[
b=\frac{a-2+\psi_0}{\sqrt{\eta}^a(a+2-\psi_0)}.
\]
Moreover, a direct calculation shows that
\[
\int_0^\infty e^zrdr=2a.
\] 
Let us conclude the proof. Suppose $R>0$. Noting $z_n(R)=z(R)+o(1)$ and \eqref{eq:hh}, we calculate that
\[
\begin{split}
&\phi_n(\tau_n R/\sqrt{\phi_n(\tau_n)})\\
&=R^2\left(1+\frac{z_n(R)}{pm_n^p}\right)^{p-1}\frac{h(u_n(\tau_n R/\sqrt{\phi_n(\tau_n)}))}{h(m_n)}e^{m_n^p\left\{\left(1+\frac{z_n(R)}{pm_n^p}\right)^p-1\right\}}\\
&=\frac{2a^2bR^{a}}{(1+bR^{a})^2}+o(1).
\end{split}
\]
Note that there exists a unique constant $R_*>0$ such that 
\[\max_{R>0}\frac{2a^2bR^{a}}{(1+bR^{a})^2}=\frac{2a^2bR_*^{a}}{(1+bR_*^{a})^2}=\frac{a^2}{2}.\]
Hence, there exists a sequence $(R_n)\subset (0,\infty)$ of values  such that putting $r_{k,n}=\tau_nR_n/\sqrt{\phi_n(\tau_n)}$ for all $n\in \mathbb{N}$, we have that $\phi_n'(r_{k,n})=0$ for all $n\in \mathbb{N}$, $\phi_n(r_{k,n})\to a^2/2$, and $R_n\to R_*$ as $n\to \infty$. Notice that since $z_n(R_n)=z(R_*)+o(1)$, it holds that $u_n(r_{k,n})/\mu_n\to \tilde{\delta}_k$ as $n\to \infty$. Then Lemma \ref{lem:pp} shows that $\psi_n(r_{k,n})\to 2$ as $n\to \infty$. Moreover, it follows from the final conclusion in  Lemma \ref{lem:gl} that $r_{k,n}\to0$ as $n\to \infty$. The last assertion in the previous lemma ensures $(r_{k,n})\subset (\sig_n,1)$ up  to a subsequence. Finally, we set $\ga_{k,n}=r_{k,n}/\sqrt{\phi_n(r_{k,n})}$ for all $n\in \mathbb{N}$. This implies $\ga_{k,n}^2p\la_n u_n(r_{k,n})^{p-1}f(u_n(r_{k,n}))=1$ for all $n\in \mathbb{N}$ and $\ga_{k,n}\to0$ as $n\to \infty$. Furthermore, we define  
\[z_{k,n}(r)=pu_n(r_{k,n})^{p-1}(u_n(\ga_{k,n}r)-u_n(r_{k,n}))\]
 for all $r\in [0,1/\ga_{k,n}]$ and $n\in \mathbb{N}$. It follows from \eqref{q} that
\begin{equation}\label{eq:zk}
\begin{cases}
-z_{k,n}''-\frac{1}{r}z_{k,n}'= \frac{f(u_n(\ga_{k,n}\cdot))}{f(u_n(r_{k,n}))}\text{ in }(0,1/\ga_{k,n}),\\
z_{k,n}(r_{k,n}/\ga_{k,n})=0,\ (r_{k,n}/\ga_{k,n})z_{k,n}'(r_{k,n}/\ga_{k,n})=-\psi_n(r_{k,n}).
\end{cases}
\end{equation}
Then the similar argument as above proves that there exists a smooth function $z_k$ on $(0,\infty)$ such that $z_{k,n}\to z_k$ in $C^2_{\text{loc}}((0,\infty))$ and 
\begin{equation}\label{eq:z}
\begin{cases}
-z_k''-\frac{1}{r}z_k'= e^{z_k}\text{ in }(0,\infty),\\
z_k(a/\sqrt{2})=0,\ (a/\sqrt{2})z_k'(a/\sqrt{2})=-2.
\end{cases}
\end{equation}
Again integrating the equation, we find constants $\tilde{a}\in(0,2)$ and $\tilde{b}$ such that 
\[
z_k(r)=\log{\frac{2\tilde{a}^2 \tilde{b} }{r^{2-\tilde{a}}(1+\tilde{b} r^{\tilde{a}})^2}}
\]
for all $r>0$. Using the latter conditions in \eqref{eq:z}, we get $\tilde{a}=a$ and $\tilde{b}=(\sqrt{2}/a)^{a}$. We finish the proof. 
\end{proof}
We extend the concentration region around $r=r_{k,n}$.   
\begin{lemma}\label{lem:e5} Assume $p>2$.  
Let $\tilde{a}_k,\tilde{\delta}_k>0$ be the constants, $(r_{k,n})$, $(\ga_{k,n})$, and $(z_{k,n})$ the sequences, and $z_k$ the function in the previous lemma. Then, there exist sequences $(\bar{\rho}_{k,n})\subset (\rho_{k-1,n},r_{k,n})$ and $(\rho_{k,n})\subset (r_{k,n},1)$ such that $\rho_{k-1,n}/\bar{\rho}_{k,n}\to0$, $\bar{\rho}_{k,n}/r_{k,n}\to 0$, $r_{k,n}/\rho_{k,n}\to 0$, $\mu_n^{-p/2}\log{(\bar{\rho}_{k,n}/r_{k,n})}\to 0$, $\mu_n^{-p/2}\log{(\rho_{k,n}/r_{k,n})}\to 0$,  $u_n(\bar{\rho}_{k,n})/\mu_n\to \tilde{\delta}_k$, $u_n(\rho_{k,n})/\mu_n\to \tilde{\delta}_k$, $\phi_n(\bar{\rho}_{k,n})\to0$, $\phi_n(\rho_{k,n})\to0$, 
\[
\|z_{k,n}-z_k\|_{C^2([\bar{\rho}_{k,n}/\ga_{k,n},\rho_{k,n}/\ga_{k,n}])}\to0,
\] 
\[
\sup_{r\in[\bar{\rho}_{k,n}/\ga_{k,n},\rho_{k,n}/\ga_{k,n}]}|rz_{k,n}'(r)-rz_k'(r)|\to0,
\]
\[
\sup_{r\in[\bar{\rho}_{k,n}/\ga_{k,n},\rho_{k,n}/\ga_{k,n}]}\left|\frac{h(u_n(\ga_{k,n}r))}{h(u_n(r_{k,n}))}-1\right|\to0,
\]
\[
E_n(\bar{\rho}_{k,n},\rho_{k,n})\to 2\tilde{a}_k/\tilde{\delta}_k^{p-1},
\]
and 
\[
E_n(\rho_{k-1,n},\bar{\rho}_{k,n})\to0
\]
as $n\to \infty$ up to a subsequence.  
\end{lemma}
\begin{proof} We put $\mu_{k,n}=u_n(r_{k,n})$ for all $n\in \mathbb{N}$. Since $z_{k,n}\to z_k$ in $C^2_{\text{loc}}((0,\infty))$ as $n\to \infty$, by \eqref{hb}, there exist sequences $(\bar{R}_n),(R_n)\subset (0,\infty)$ such that  $\bar{R}_n\to0$, $R_n\to\infty$, and 
\begin{equation}\label{hh0}
\|z_{k,n}-z_k\|_{C^2([\bar{R}_n,R_n])}\to0,
\end{equation}
\[
\sup_{r\in [\bar{R}_n,R_n]}|rz_{k,n}'(r)-rz_k'(r)|\to0,
\] 
\begin{equation}\label{hh}
\sup_{r\in[\bar{R}_n,R_n]}\left|\frac{h(u_n(\ga_{k,n}r))}{h(\mu_{k,n})}-1\right|\to0,
\end{equation}
\begin{equation}\label{eq:rr}
\mu_n^{-p/2}\log{\bar{R}_n}\to0,\ \mu_n^{-p/2}\log{R_n}\to0,   
\end{equation}
and 
\[
\begin{split}
E_n(\ga_{k,n}\bar{R}_n,\ga_{k,n}R_n)
&=\left(\frac{\mu_n}{\mu_{k,n}}\right)^{p-1}\int_{\bar{R}_n}^{R_n}\frac{h(u_n(\ga_{k,n}r))}{h(\mu_{k,n})}e^{\mu_{k,n}^p\left\{\left(1+\frac{z_{k,n}}{p\mu_{k,n}^p}\right)^p-1\right\}}rdr\\
&\to \frac1{\tilde{\delta}_k^{p-1}}\int_0^\infty e^{z_k}rdr,
\end{split}
\]
as $n\to \infty$. We set $\bar{\rho}_{k,n}=\ga_{k,n}\bar{R}_n$ and $\rho_{k,n}=\ga_{k,n} R_n$ for all $n\in \mathbb{N}$. These are the desired sequences. Indeed, since $r_{k,n}/\ga_{k,n}=\tilde{a}_k/\sqrt{2}+o(1)$, we get $\bar{\rho}_{k,n}/r_{k,n}\to0$ and $r_{k,n}/\rho_{k,n}\to0$ as $n\to \infty$.  Moreover, noting  $z_{k,n}(R_n)=z_k(R_n)+o(1)$, $z_{k,n}(\bar{R}_n)=z_k(\bar{R}_n)+o(1)$, and  \eqref{eq:rr}, we confirm $u_n(\rho_{k,n} )/\mu_n\to \tilde{\delta}_k$ and $u_n(\bar{\rho}_{k,n})/\mu_n\to \tilde{\delta}_k$ as $n\to \infty$. Then it clearly follows that $(\bar{\rho}_{k,n})\subset (\rho_{k-1,n},r_{k,n})$ and $(\rho_{k,n})\subset (r_{k,n},1)$ for all large $n\in \mathbb{N}$. Furthermore, we get from \eqref{hh}, \eqref{hh0}, and \eqref{eq:rr} that
\[
\begin{split}
&\phi_n(\rho_{k,n})\\
&=(1+o(1)) R_n^2e^{\mu_{k,n}^p\left\{\left(1+\frac{z_{k,n}(R_n)}{p\mu_{k,n}^p}\right)^p-1\right\}}\\
&=(1+o(1))\exp{\left[\mu_{k,n}^p\left\{\left(1-\frac{2+\tilde{a}_k+o(1)}{p\mu_{k,n}^p}\log{R_n}\right)^p-1+\frac{2\log{R_n}}{\mu_{k,n}^p}\right\}\right]}\\
&=(1+o(1))\exp{\left[-(\tilde{a}_k+o(1))\log{R_n})\right]} \to0
\end{split}
\]
as $n\to \infty$. Analogously, we also obtain  $\phi_n(\bar{\rho}_{k,n})\to0$ as $n\to \infty$. Finally, recall the sequences $(\sig_n)$ and $(\tau_n)$ in the previous lemma. If $(\bar{\rho}_{k,n})\subset (\rho_{k-1,n},\sig_n]$ for all large $n\in \mathbb{N}$, we obviously  have $E_n(\rho_{k-1,n},\bar{\rho}_{k,n})\to0$ by Lemma \ref{lem:e2}. On the other hand, suppose $(\bar{\rho}_{k,n})\subset (\sig_n,r_{k,n})$ up to a subsequence. From our choice of $(\ga_{k,n})$ and $(r_{k,n})$ in the proof of the previous lemma (recall the discussion above \eqref{eq:zk}), we get $\bar{\rho}_{k,n}/\tau_n=\bar{R}_n(\ga_{k,n}/\tau_n)\to0$ as $n\to \infty$. In particular, we see $\sigma_n<\bar{\rho}_{k,n}<\tau_n$ for all large $n\in \mathbb{N}$. Then from \eqref{eq:st} and Lemma \ref{lem:pp}, $\phi_n(r)$ is strictly increasing for all  $r\in[\sig_n,\bar{\rho}_{k,n}]$ and large $n\in \mathbb{N}$.  Hence we have 
 $\sup_{r\in[\sig_n,\bar{\rho}_{k,n}]}\phi_n(r)=\phi_n(\bar{\rho}_{k,n})\to0$ as $n\to \infty$. Therefore, again noting \eqref{eq:st}, we obtain $E_n(\sig_n,\bar{\rho}_{k,n})\to0$ as $n\to \infty$ by Lemma \ref{lemD}. Consequently, Lemma \ref{lem:e2} completes the final conclusion. Notice also that Lemma \ref{lem:22} proves that $\rho_{k-1,n}/\bar{\rho}_{k,n}\to0$ as $n\to \infty$. We finish the proof. 
\end{proof}
In the next two lemmas, we determine $\tilde{\delta}_k$ and $\tilde{a}_k$. First we discuss $\tilde{\delta}_k$.  
\begin{lemma}\label{lem:dl} We suppose $p>2$ and  $\tilde{\delta}_k\in(0,\delta_k]$ is the constant in the previous lemmas. Then, we get $\tilde{\delta}_k=\delta_k$.
\end{lemma}
\begin{proof} From \eqref{id2}, we have
\[
\begin{split}
u_n&(\rho_{k-1,n})-u_n(\bar{\rho}_{k,n})\\
&=\left(\log{\frac{\bar{\rho}_{k,n}}{\rho_{k-1,n}}}\right)\int_{0}^{\rho_{k-1,n}} \la_nf(u_n)rdr+O\left(\log{\frac{\bar{\rho}_{k,n}}{\rho_{k-1,n}}}\int_{\rho_{k-1,n}}^{\bar{\rho}_{k,n}} \la_nf(u_n)rdr\right).
\end{split}
\]
It follows that
\[
\begin{split}
\delta_{k-1}&-\tilde{\delta}_k+o(1)\\
&=\frac1{p\mu_n^p}\left(\log{\frac{\bar{\rho}_{k,n}}{\rho_{k-1,n}}}\right)E_n(0,\rho_{k-1,n})+O\left(\frac1{\mu_n^p}\log{\frac{\bar{\rho}_{k,n}}{\rho_{k-1,n}}}E_n(\rho_{k-1,n},\bar{\rho}_{k,n})\right).
\end{split}
\]
Here  Lemmas \ref{lem:22}, \ref{lem:e5}, and (A$_k$) imply 
\[
\begin{split}
\frac1{\mu_n^p}\log{\frac{\bar{\rho}_{k,n}}{\rho_{k-1,n}}}&=\frac1{\mu_n^p}\left(\log{\frac{r_{k,n}}{r_{k-1,n}}}+\log{\frac{\bar{\rho}_{k,n}}{r_{k,n}}}+\log{\frac{r_{k-1,n}}{\rho_{k-1,n}}}\right)\\
&=\frac{\delta_{k-1}^p-\tilde{\delta}_k^p}{2}+o(1).
\end{split}
\]
Using this for the previous formula, we get from (A$_k$) and Lemma \ref{lem:e5} that
\[
\begin{split}
\delta_{k-1}-\tilde{\delta}_k+o(1)
&=\frac{\delta_{k-1}^p-\tilde{\delta}_k^p}{2p}\sum_{i=0}^{k-1}E_i+o(1).
\end{split}
\]
It follows that
\[
\frac{2p}{\delta_{k-1}^{p-1}\sum_{i=0}^{k-1} E_i}\left(1-\frac{\tilde{\delta}_k}{\delta_{k-1}}\right)-1+\left(\frac{\tilde{\delta}_k}{\delta_{k-1}}\right)^p=0.
\]
Hence, we derive $\tilde{\delta}_k=\delta_k$ by Lemma \ref{lem:bg} and \eqref{eq:del2}. This finishes the proof. 
\end{proof}
We then decide $\tilde{a}_k$.  
\begin{lemma}\label{lem:ak}
Assume $p>2$ and $\tilde{a}_k\in (0,2)$ is the number  in the previous lemmas.  Then we have  $\tilde{a}_k=a_k$.
\end{lemma}
\begin{proof} We take the sequences $(\bar{\rho}_{k,n})$, $(r_{k,n})$, $(\rho_{k,n})$  obtained in Lemmas \ref{lem:f1} and \ref{lem:e5}. Then from Lemmas \ref{lem:f1}, \ref{lem:e5}, \ref{lem:dl},  and (A$_k$), we get
\[
\begin{split}
2&=\lim_{n\to \infty}\psi_n(r_{k,n})=\delta_k^{p-1}\lim_{n\to \infty}(E_n(0,\bar{\rho}_{k,n})+E_n(\bar{\rho}_{k,n},r_{k,n}))\\
&=\delta_k^{p-1}\sum_{i=0}^{k-1}E_i+\tilde{a}_k.
\end{split}
\]
It follows that 
\[
\tilde{a}_k=2-\delta_k^{p-1}\sum_{i=0}^{k-1}E_i=a_k
\]
by Lemma \ref{lem:bg} and \eqref{eq:del2}. This completes the proof.
\end{proof}
Finally, we give a pointwise estimate for $z_{k,n}$.
\begin{lemma}\label{lem:ff} Let  $p>2$ and $(\ga_{k,n})$, $(\rho_{k,n})$ and $(z_{k,n})$ the sequences in Lemmas \ref{lem:f1} and  \ref{lem:e5}. Then  we have that  
\[
z_{k,n}(r)\le -(2+a_k+o(1))\log{r}
\]
for all $r\in[\rho_{k,n}/\ga_{k,n},1/\ga_{k,n}]$ and $n\in \mathbb{N}$ where $o(1)\to 0$ uniformly for all $r$ in the interval.
\end{lemma}
\begin{proof} Put $R_n=\rho_{k,n}/\ga_{k,n}$ for all $n\in \mathbb{N}$. From Lemmas \ref{lem:e5} and \ref{lem:ak}, we have $R_n\to \infty$, $z_{k,n}(R_n)=-(2+a_k+o(1))\log{R_n}$, and $R_n z_{k,n}'(R_n)=-(2+a_k+o(1))$. Then for any $r\in[R_n,1/\ga_{k,n}]$, we get 
\[
\begin{split}
-rz_{k,n}'(r)&\ge -R_n z_{k,n}'(R_n)
=2+a_k+o(1).
\end{split}
\]
Multiplying by $1/r$ and integrating over $[R_n,r]$ give 
\[
\begin{split}
z_{k,n}(r)&\le z_k(R_n)-(2+a_k+o(1))\log{\frac{r}{R_n}}+o(1)\\
&\le -(2+a_k+o(1))\log{r}
\end{split}
\]
for all $r\in[R_n,1/\ga_{k,n}]$ and $n\in \mathbb{N}$. This finishes the proof.
\end{proof}

\subsection{Proofs of main theorems for $p> 2$}\label{sec:prf}
Finally, we show our theorems for $p>2$. We begin with Theorem \ref{thm30}.
\begin{proof}[Proof of Theorem \ref{thm30}] 
We first claim that (A$_k$) is true for all $k\in \mathbb{N}$. To see this we argue by induction. If $k=1$, we take the sequences $(r_{0,n}),(\rho_{0,n})$ in Theorem \ref{thm1} and Lemma \ref{lem21}. Then noting also the last assertion in Lemma \ref{lem:thm1}, we conclude that (A$_1$) is true.  We assume for a number $k\in\mathbb{N}$,  (A$_k$) is verified. Then from Lemmas \ref{lem:f1}, \ref{lem:e5}, \ref{lem:dl}, \ref{lem:ak}, and \ref{lem:ff}, we ensure that (A$_{k+1}$) is also satisfied. This proves the claim.  Now fix any $k\in \mathbb{N}$. 
Then from Lemmas \ref{lem:f1}, \ref{lem:e5}, \ref{lem:dl}, and \ref{lem:ak}, we get all the assertions of the theorem except for  \eqref{eq:en2}. We ensure the former conclusion in \eqref{eq:en2} since   
\[
p\int_{\rho_{k-1,n}}^{\bar{\rho}_{k,n}}\la_n u_n^{p-1}f(u_n)rd\le E_n(\rho_{k-1,n},\bar{\rho}_{k,n})\to0
\]
as $n\to \infty$ by \eqref{eq:en1}. On the other hand, putting $u_n(r_{k,n})=\mu_{k,n}$ as before, we calculate that  
\[
\begin{split}
&\left(\frac{u_n(\bar{\rho}_{k,n})}{\mu_n}\right)^{p-1}E_n(\bar{\rho}_{k,n},\rho_{k,n})\\
&\ \ \ \ \ge p\int_{\bar{\rho}_{k,n}}^{\rho_{k,n}}\la_n u_n^{p-1} f(u_n)rdr\\
&\ \ \ \ =\int_{\bar{\rho}_{k,n}/\ga_{k,n}}^{\rho_{k,n}/\ga_{k,n}}\left(1+\frac{z_{k,n}}{p\mu_{k,n}^p}\right)^{p-1}\frac{h(u_n(\ga_{k,n}r))}{h(\mu_{k,n})}e^{\mu_{k,n}^p\left\{\left(1+\frac{z_{k,n}}{p\mu_{k,n}^p}\right)^p-1\right\}}rdr.
\end{split}
\]
Then \eqref{eq:en1} and the Fatou lemma with Lemma \ref{lem:e5} show the latter one. We finish the proof. 
\end{proof}
Next, we shall  prove Theorem \ref{thm3}. To do this, we study some more properties of the sequences $(\delta_k)$ and $(a_k)$ defined in \eqref{eq:del1} and \eqref{eq:del2}. (We remark that most of the properties proved here are investigated in Section 2 in \cite{McMc} in a similar way. Since the context and expressions for \eqref{eq:del1} and \eqref{eq:del2} are different, we give our proofs here for the sake of completeness and readers' convenience.)  We first confirm the monotonicity and convergence of $(a_k)$.    
\begin{lemma}\label{lem:bf2}
Assume $p>2$ and put $d_k:=\delta_k/\delta_{k-1}$. Then we have $a_{k-1}>a_k$ and $d_k<d_{k+1}$ for any $k\in \mathbb{N}$. Moreover, we get  $a_k\to0$ and $d_k\to1$ as $k\to \infty$. 
\end{lemma}
\begin{proof} 
Set $k\in \mathbb{N}$. Using $d_k$, we write \eqref{eq:del1} and \eqref{eq:del2} as
\begin{equation}\label{s1}
\frac{2p}{2+a_{k-1}}\left(1-d_k\right)-1+d_k^p=0
\end{equation}
and
\begin{equation}\label{s2}
a_k=2-d_k^{p-1}(2+a_{k-1}).
\end{equation} 
From \eqref{s2}, we see
\begin{equation}\label{w1}
d_k=\left(\frac{2-a_k}{2+a_{k-1}}\right)^{\frac1{p-1}}.
\end{equation}
Substituting this into \eqref{s1}, we get
\begin{equation}\label{w2}
\frac{2p}{2+a_{k-1}}\left\{1-\left(\frac{2-a_k}{2+a_{k-1}}\right)^{\frac1{p-1}}\right\}-1+\left(\frac{2-a_k}{2+a_{k-1}}\right)^{\frac p{p-1}}=0.
\end{equation}
Then we define a function 
\[
g(x)=\frac{2p}{2+a_{k-1}}\left\{1-\left(\frac{2-x}{2+a_{k-1}}\right)^{\frac1{p-1}}\right\}-1+\left(\frac{2-x}{2+a_{k-1}}\right)^{\frac p{p-1}}
\]
for all $x\in [0,2]$. It clearly follows that $g(x)$ is strictly increasing on $[0,2]$, $g(a_k)=0$, and $g(2)>0$. Moreover, we claim $g(a_{k-1})>0$. Once this is proved, we get the first conclusion $a_k<a_{k-1}$ from the monotonicity of $g$. Although the proof of the claim is still elementary, we show it for readers' convenience.  We put a function
\[
\tilde{g}(x)=\frac{2p}{2+x}\left\{1-\left(\frac{2-x}{2+x}\right)^{\frac1{p-1}}\right\}-1+\left(\frac{2-x}{2+x}\right)^{\frac p{p-1}}
\]
for all $x\in [0,2]$. Notice that $\tilde{g}(0)=0$. Moreover, we observe that
\[
\tilde{g}'(x)=\frac{2p}{(2+x)^3}\left(\frac{2-x}{2+x}\right)^{-\frac{p-2}{p-1}}\left\{ 2-\frac{p-3}{p-1}x-(2+x)\left(\frac{2-x}{2+x}\right)^{\frac {p-2}{p-1}}\right\}.
\]
It follows that $\tilde{g}'(0)=0$. Furthermore, setting
\[
\hat{g}(x)=2-\frac{p-3}{p-1}x-(2+x)\left(\frac{2-x}{2+x}\right)^{\frac {p-2}{p-1}}
\]
for all $x\in [0,2]$, a direct calculation shows that $\hat{g}(0)=0$, $\hat{g}'(0)=0$, and $\hat{g}''(x)>0$ for all $x\in [0,2)$. These facts imply that $\tilde{g}(x)>0$ for all $x\in(0,2]$. Then, since $a_{k-1}\in(0,2]$, we see $g(a_{k-1})=\tilde{g}(a_{k-1})>0$. This proves the claim.   Then \eqref{w1} ensures $d_k<d_{k+1}$. Consequently, there exist constants $a_*\in[0,2)$ and $d_*\in (0,1]$ such that $a_k\to a_*$ and $d_k\to d_*$ as $k\to \infty$. From \eqref{w1} and \eqref{w2}, we obtain
\[
d_*=\left(\frac{2-a_*}{2+a_*}\right)^{\frac1{p-1}}.
\]
and
\[
\frac{2p}{2+a_*}\left\{1-\left(\frac{2-a_*}{2+a_*}\right)^{\frac1{p-1}}\right\}-1+\left(\frac{2-a_*}{2+a_*}\right)^{\frac p{p-1}}=0.
\] 
Then recalling that $\tilde{g}(x)>0$ for all $x\in(0,2)$, we confirm $a_*=0$ by the latter formula.  Then the former one proves $d_*=1$. This completes the proof.
\end{proof}
We next give a preliminary lemma.
\begin{lemma}\label{pre0} Let $(\e_k)\subset (0,1)$ be any sequence such that $\e_k\to0$ as $k\to \infty$ and there exists a value $\alpha>0$ such that 
\begin{equation}\label{pre1}
\e_{k+1}\ge \e_k-\alpha \e_k^2
\end{equation}
for all $k\in \mathbb{N}$. Then there exist numbers $k_0\in \mathbb{N}$ and $\beta>0$ such that $\e_k\ge \beta/k$ for all $k\ge k_0$.
\begin{proof}
We first claim $k\e_k\not \to 0$ as $k\to \infty$. If not, there exists a sequence $(\xi_k)$ of positive values  such that $\xi_k\to0$ as $k\to \infty$ and $\e_k=\xi_k/k$ for all $k\in\mathbb{N}$. Then, we find  a sequence $(k_n)$ of natural numbers  such that $k_n\to \infty$ as $n\to \infty$  and $\xi_{k_n}\ge \xi_{k_n+1}$ for all $n\in\mathbb{N}$. Hence we get from \eqref{pre1} that
\[
1\ge \frac{\xi_{k_n+1}}{\xi_{k_n}}
\ge 1+\frac1{k_n}+o\left(\frac1{k_n}\right)
\]
as $n\to \infty$. This is a contradiction.  Next, from the assumption and the previous claim, there exist  numbers  $k_0 \in \mathbb{N}$  and $\beta\in(0,1/\alpha)$ such that $\e_k<1/(2\alpha)$ for all $k\ge k_0$, $\e_{k_0}\ge \beta /k_0$, and $k_0(1-\alpha \beta)-\alpha \beta\ge0$. We claim that  $\e_k\ge \beta/k$ for all $k\ge k_0$. We prove this by induction. In fact, it is clearly true for $k=k_0$. Moreover assume it is verified  for some $k\ge k_0$. Then since $\beta/k\le \e_k<1/(2\alpha)$, we get  from \eqref{pre1} that
\[
\begin{split}
\e_{k+1}&\ge \e_k(1-\alpha\e_k)\ge \frac \beta k\left(1-\frac {\alpha \beta }k\right)= \frac{\beta}{k+1}+\frac{\beta \left\{k(1-\alpha \beta)-\alpha \beta\right\}}{k^2 (k+1)}\ge \frac{\beta}{k+1}
\end{split}
\]
as $k \ge k_0$. This proves the claim. We complete the proof.  
\end{proof}
\end{lemma}
We then check the following.  
\begin{lemma}\label{lem:div} Suppose $p>2$. Then we have $\lim_{k\to \infty}\delta_k=0$ and $\sum_{i=0}^\infty a_i=\infty$. 
\end{lemma}
\begin{proof} 
Let $(d_k)\subset (0,1)$ be a sequence of values  in Lemma \ref{lem:bf2}.   
We choose a sequence $(\e_k)\subset (0,1)$ so that $d_k=1-\e_k$ for all $k\in\mathbb{N}$. It follows that $\e_{k+1}<\e_k$ for all $k\in \mathbb{N}$ and $\e_k\to0$ as $k\to \infty$. Moreover, from \eqref{s1} and \eqref{s2}, we have 
\begin{equation}\label{eq:del3}
\frac{2p}{2+a_{k-1}}\e_k-1+(1-\e_k)^p=0
\end{equation}
and 
\begin{equation}\label{eq:del4}
\begin{split}
a_k
&=2-(1-\e_k)^{p-1} (2+a_{k-1}).
\end{split}
\end{equation}
From \eqref{eq:del3}, we get
\begin{equation}\label{eq:d5}
\begin{split}
a_{k-1}=(p-1)\e_k+\frac{(p-1)(p+1)}{6}\e_k^2+O(\e_k^3)
\end{split}
\end{equation}
where the last term on the right-hand side means that there exists a constant $C>0$ such that $|O(\e_k^3)|\le C \e_k^3$ for all $k\in \mathbb{N}$.  Substituting this into \eqref{eq:del4}, we obtain 
\begin{equation}\label{t1}
\begin{split}
a_k&=(p-1)\e_k-\frac{(p-1)(p-5)}{6}\e_k^2+O(\e_k^3).
\end{split}
\end{equation}
Then replacing $k$ with $k+1$ in \eqref{eq:d5} and substituting \eqref{t1} into it, we get 
\[
\e_k-\frac{(p-5)}{6}\e_k^2+O(\e_k^3)=\e_{k+1}+\frac{(p+1)}{6}\e_{k+1}^2+O(\e_{k+1}^3).
\]
Using the fact $\e_{k+1}\le \e_k$, we calculate
\[
\begin{split}
\e_{k+1}&=\e_k-\frac{p-5}{6}\e_k^2-\frac{p+1}{6}\e_{k+1}^2+O(\e_k^3)\\
&=\e_k-\frac{p-5}{6}\e_k^2-\frac{p+1}{6}(\e_k+O(\e_k^2))^2+O(\e_k^3)\\
&=\e_k-\frac{p-2}{3}\e_k^2+O(\e_k^3).
\end{split}
\]
Then from Lemma \ref{pre0}, there exist numbers $k_0\in \mathbb{N}$ and $\beta>0$ such that $\e_k\ge \beta/k$ for all $k\ge k_0$. Consequently, taking $k_0$ larger if necessary,  we get for all $k> k_0$ that
\[
\begin{split}
\log{\left(\frac{\delta_{k_0}}{\delta_k}\right)}&= \sum_{i=k_0+1}^k\log{\frac1{d_i}}=\sum_{i=k_0+1}^k \log{\left(\frac{1}{1-\e_i}\right)}\ge \frac12\sum_{i=k_1}^k \e_i\ge\frac \beta2\sum_{i=k_0}^k\frac1i\to\infty
\end{split}
\]
as $k\to \infty$. Hence we obtain $\delta_k\to0$ as $k\to \infty$. Moreover, recalling \eqref{t1} and choosing $k_0$ larger if necessary again, we obtain that for all $k\ge k_0$, 
\[
\sum_{i=0}^ka_i\ge \frac{p-1}{2}\sum_{i=k_0}^k  \e_i  \ge \frac{\beta(p-1)}{2}\sum_{i=k_0}^k  \frac1i\to \infty
\]
as $k\to \infty$.   This completes the proof.
\end{proof}
Let us complete the proof of our theorem.
\begin{proof}[Proof of Theorem \ref{thm3}] 
Let $(r_n)\subset (0,1)$ be any sequence such that $u_n(r_n)/\mu_n\to0$ as $n\to \infty$. Fix any $k\in \mathbb{N}$. Then, we choose the sequence $(\rho_{k,n})\subset (0,1)$ of values  from Theorem \ref{thm30}. Since $u_n(\rho_{k,n})/\mu_n\to \delta_k$ as $n\to \infty$, we have $\rho_{k,n}<r_n$ for all large $n\in \mathbb{N}$ which implies
\[
p\int_0^{r_n}\la_n u_n^{p-1}f(u_n)rdr>p\int_0^{\rho_{k,n}}\la_n u_n^{p-1}f(u_n)rdr. 
\]
It follows from \eqref{eq:en2} that 
\[
\liminf_{n\to \infty}E_n(0,r_n)\ge \liminf_{n\to \infty}p\int_0^{r_n}\la_n u_n^{p-1}f(u_n)rdr\ge \sum_{i=1}^k a_i.
\]
Since $k\in \mathbb{N}$ is arbitrary, we get the first conclusions  by Lemma \ref{lem:div}.  
  Next, noting Lemma \ref{lem34}, after extracting a subsequence if necessary, we find a constant $\nu\in [0,1]$ such that  
\[
\lim_{n\to \infty}\frac{\log{\frac1{\ga_{0,n}}}}{\mu_n^p}=\frac\nu 2.
\]
We shall prove $\nu=1$. We assume  $\nu\in[0,1)$ on the contrary. Then since $\delta_k\to0$ as $k\to \infty$ by Lemma \ref{lem:div}, there exists a number $k_0\in \mathbb{N}$ such that $\delta_{k_0}^p-1+\nu<0$. We choose a sequence $(r_{k_0,n})\subset (0,1)$ from Theorem \ref{thm30}. It follows that $r_{k_0,n}\to 0$, $u_n(r_{k_0,n})/\mu_n\to \delta_{k_0}$, and $\phi_n(r_{k_0,n})\to a_{k_0}^2/2$ as $n\to \infty$. Then Lemma \ref{lem34} yields
\[
\lim_{n\to \infty}\frac{\log{\frac1{r_{k_0,n}}}}{\mu_n^p}=\frac{\delta_{k_0}^p-1+\nu} 2<0.
\]
This is impossible. Hence we get
 \begin{equation}\label{eq:jj}
\lim_{n\to \infty}\frac{\log{\frac1{\ga_{0,n}}}}{\mu_n^p}=\frac1 2.
\end{equation}
Then, from the definition of $\ga_{0,n}$ and \eqref{ha}, we confirm \eqref{sup3}. Moreover, \eqref{eq:rk2} with $k=0$ is shown by \eqref{eq:jj} since $r_{0,n}/\ga_{0,n}\to 2\sqrt{2}$ as $n\to \infty$. \eqref{eq:rk2} with $k\ge1$ is proved by the  last assertion in Lemma \ref{lem34} with $\nu=1$.  
Lastly, take a sequence $(r_n)\subset (0,1)$ so that $u_n(r_n)\to \infty$ and $u_n(r_n)/\mu_n\to0$ as $n\to \infty$. The final conclusion  in Lemma \ref{lem:gl} implies $r_n\to0$ as $n\to \infty$. Then for any $r\in (0,1)$, we use \eqref{id2} and get
\[
p\mu_n^{p-1}u_n(r)\ge E_n(0,r)\log{\frac1r}\ge E_n(0,r_n)\log{\frac1r}
\]
for all large $n\in \mathbb{N}$. Hence, \eqref{sup222} proves \eqref{sup4}. This completes the proof. 
\end{proof}
\section{Concentration and oscillation}\label{sec:osc}
In this section, we directly deduce various oscillation estimates  from the concentration ones in the previous sections.  This is  our second main discussion in this paper.  In the following, for any $k\in \mathbb{N}\cup\{0\}$, we recall numbers $a_k$ and $\delta_k$, and sequences $(\ga_{k,n})$, $(r_{k,n})$, $(\rho_{k,n})$, and $(\bar{\rho}_{k,n})$, with regarding $\bar{\rho}_{0,n}=0$, in our main theorems in Section \ref{sec:intr}.  
\subsection{Oscillation estimates}
We first give several asymptotic formulas which contain the delicate information on the shapes of the graphs of blow-up solutions. This allows us to observe  the several oscillation behaviors depending on $p>0$. Particularly, we will find  that the infinite sequence of bubbles causes the infinite oscillation of the solutions in the supercritical case $p>2$. 
\begin{theorem}\label{thm:osc} Assume $p>0$ and (H1). Let $\{(\la_n,\mu_n,u_n)\}$ be a sequence of solutions of \eqref{q} with $\mu_n\to \infty$ as $n\to \infty$. Then for any sequence $(r_n)\subset (0,\rho_{0,n}]$ of values, taking a sequence $(R_n)\subset (0,\infty)$ so that $r_n=R_n\ga_{0,n}$ for all $n\in \mathbb{N}$,  we have
\begin{equation}\label{de0}
p\la_nr_n^2u_n^{p-1}(r_n)f(u_n(r_n))=\frac{64R_n^2}{(8+R_n^2)^2}(1+o(1))
\end{equation}
as $n\to \infty$ up  to a subsequence. 
If, in addition,  $\mu_n^{-p}\log{(r_n/\ga_{0,n})}\to0$ as $n\to \infty$, then the above formula implies that if $p\in(0,2)$,
\begin{equation}\label{de01}
\begin{split}
u_n(r_n)=\left\{\left(\frac{4}{p}+o(1)\right)\log{\frac1{r_n}}\right\}^{\frac1p}
\end{split}
\end{equation}
and if $p\ge2$, 
\begin{equation}\label{de02}
\begin{split}
u_n(r_n)=\left\{(2+o(1))\log{\frac1{r_n}}\right\}^{\frac1p}
\end{split}
\end{equation}
as $n\to \infty$. 
 Moreover, by extracting a subsequence if necessary, we obtain the following.
\begin{enumerate}
\item[(i)] Suppose $0<p<2$. Then  for any sequence $(r_n)\subset [\rho_{0,n},1)$, taking a sequence $(\delta_n)\subset (0,1)$ so that   $u_n(r_n)=\delta_n\mu_n$ for all $n\in \mathbb{N}$, we get
\begin{equation}\label{bo0}
u_n(r_n)=\left\{\left(\frac{4\delta_n^{p-1}}p+o(1)\right)\log{\frac1{r_n}}\right\}^{\frac1p}
\end{equation}
as  $n\to \infty$. 
\item[(ii)] Assume $p=2$. Then for any sequence $(r_n)\subset [\rho_{0,n},1)$ and constant $\delta\in(0,1]$ such that $u_n(r_n)/\mu_n\to \delta$ as $n\to \infty$, we have
\begin{equation}\label{bo1}
u_n(r_n)=\left\{(2\delta+o(1))\log{\frac1{r_n}}\right\}^{\frac12}
\end{equation}
as  $n\to \infty$.
\item[(iii)] Let $p>2$ and fix any number $k\in \mathbb{N}\cup \{0\}$. Then, for any sequence $(r_n)\subset [\rho_{k,n},\bar{\rho}_{k+1,n}]$ and value $\delta\in [\delta_{k+1},\delta_k]$ such  that   $u_n(r_n)/\mu_n\to\delta$ as $n\to \infty$, we obtain
\begin{equation}\label{bo2}
u_n(r_n)=\left\{(\beta_k(\delta)+o(1))\log{\frac1{r_n}}\right\}^{\frac1p}
\end{equation}
as  $n\to \infty$ where we defined the function $\beta_k: [\delta_{k+1},\delta_k]\to \mathbb{R}$ by  
\[
\beta_k(\delta)=\frac{2(2+a_k)(\delta/\delta_k)^p}{2+a_k-2p(1-\delta/\delta_k)}.
\]
Moreover, we assume $k\in \mathbb{N}$. Then, if  $(r_n)\subset [\bar{\rho}_{k,n},\rho_{k,n}]$, taking a sequence $(R_n)\subset (0,\infty)$ so that $r_n=R_n\ga_{k,n}$ for all $n\in \mathbb{N}$, 
 we get
\begin{equation}\label{de1}
p\la_nr_n^2u_n(r_n)^{p-1}f(u_n(r_n))=\frac{2a_k^2b_k R_n^{a_k}}{(1+b_kR_n^{a_k})^2}(1+o(1))
\end{equation}
which implies 
\begin{equation}\label{de11}
\begin{split}
u_n(r_n)=\left\{(2+o(1))\log{\frac1{r_n}}\right\}^{\frac1p}
\end{split}
\end{equation}
as  $n\to \infty$.
\end{enumerate}
\end{theorem}
\begin{remark}\label{rmk:osc1}
Note that if $p\ge2$, using \eqref{de02} for \eqref{de0} and \eqref{de11} for \eqref{de1} respectively, we can deduce more precise asymptotic formulas.  For example, if $p\ge2$ and $f(t)=u^me^{t^p+\alpha t^q}$ for all large $t\ge0$ where $m,\alpha\in \mathbb{R}$ and $0<q<p$ are given constants, we get from \eqref{de0} with \eqref{de02}  (and also \eqref{de1} with \eqref{de11} if $p>2$) that 
\[
\begin{split}
 u_n&(r_n)\\
&=\left\{2\log{\frac1{r_n}}-\alpha u_n(r_n)^q-(p-1+m)\log{u_n(r_n)}+\log{\frac{l(R_n)}{p\la_n}}+o(1)\right\}^{\frac1p}\\
&=\Bigg\{2 \log{\frac1{r_n}}-\alpha \left((2+o(1)) \log{\frac1{r_n}}\right)^{\frac qp}
\\
&\ \ \ \ \ \ \ \ \ \ \ \ \ \ \ \ \ \ \ \ \ -\left(1-\frac{1-m}{p}\right) \log{\left(2\log{\frac1{r_n}}\right)}+\log{\frac{l(R_n)}{p\la_n}}+o(1)
\Bigg\}^{\frac1p}
\end{split}
\]
as $n\to \infty$ where 
\[
l(R)=\frac{64R^2}{(8+R^2)^2 }\ \left(\text{and }\frac{2a_k^2b_kR^{a_k}}{(1+b_kR^{a_k})^2 }\text{ respectively}\right)\ \ (R>0).
\]
 Moreover, recalling \eqref{eq:r0}, \eqref{eq:rk}, and the facts
\[
\max_{R>0}\frac{64 R^2}{(8+R^2)^2}=2\ \text{ and } 
\max_{R>0}\frac{2a_k^2b_k R^{a_k}}{(1+b_kR^{a_k})^2}=\frac{a_k^2}{2}
\]
for all $k\in \mathbb{N}$, we find that the limit value of $l(R_n)$ in the previous formula is maximized when $(r_n)=(r_{k,n})$. These observations suggest that $u_n(r)$ attains the top of the oscillation at the center of the concentration  $r=r_{k,n}$. See more explanation below. 
\end{remark}
\begin{remark} The assertion (i) for $0<p<2$ can not be simply extended to the case $p=2$. This is because of the fact that  \eqref{sub1} is not true for $p=2$ in general. But,  as noted below Theorem \ref{thm:crit}, if $p=2$ and $h(t)=te^{\alpha t^q}$ for all $t\ge0$ with $\alpha\ge0$ and $0<q<2$, \eqref{sub1} still holds  true and thus, the same assertion with (i) is valid. See the proof of the theorem below.
\end{remark}
This theorem describes several oscillation behaviors of graphs of solutions by means of the curve 
\[
U_\beta(r)=\left(\beta \log{\frac1r}\right)^{\frac1p} \ \ (r\in(0,1])
\]
with $\beta>0$. 
 Let us begin with the standard case $1<p\le2$. In this case, we can interpret \eqref{de01}, \eqref{de02}, \eqref{bo0}, \eqref{bo1}, and the argument in the remarks above as follows. First notice that \eqref{de01} and \eqref{de02} show that $u_n(r)$ touches the curve $U_{4/p+o(1)}(r)$ at the center of the first concentration $r=r_{0,n}$. This  implies that starting from the finite value $\mu_n$ at the origin, as $r$ increases $0$ to $r_{0,n}$, $u_n(r)$ crosses $U_{\beta}(r)$ for all $0<\beta<4/p$ and finally arrives on the highest curve $U_{4/p+o(1)}(r)$ at $r=r_{0,n}$. After that, as $r$ increases again, then $u_n(r)/\mu_n$ decreases and thus, \eqref{bo0} and \eqref{bo1} imply that $u_n(r)$ passes the lower curve $U_{\beta}(r)$ for all $4/p>\beta>0$ again. This up-and-down behavior can be understood as the graph of $u_n(r)$ oscillates at least once as $r$ increases from $0$  to $1$. More precisely, if $1<p<2$, it oscillates only once while if $p=2$, the number of the oscillations may be more than one depending on the choice of $h$. Notice also that one bubble causes one oscillation here. On the other hand, in the case $0<p<1$, the conclusion is very different. After arriving on the curve $U_{4/p}(r)$ at $r=r_{0,n}$ as in the previous case, as $r$ increases again and then $u_n(r)/\mu_n$ decreases and thus, since the exponent  $p-1$ on $\delta_n$ in \eqref{bo0} is negative, $u_n(r)$ further passes the higher curve $U_{\beta}(r)$ for all $\beta\ge 4/p$. In particular, it crosses $U_{\beta}(r)$ for all $\beta>0$ as $r$ increases from $0$ to $1$ and  exhibits no oscillation behavior. This behavior can be understood as the result of the entire blow-up proved by \eqref{sub3}. Finally,  let us consider the most interesting case $p>2$. To do this, fix any $k\in \mathbb{N}\cup\{0\}$. 
 Then from  \eqref{de02} and \eqref{de11}, we first see that  the graph of $u_n(r)$ reaches the curve $U_{2+o(1)}(r)$ at $r=r_{k,n}$ which  is the center of the $(k+1)$th bubble. Next notice that $\beta_k(\delta_k)=\beta_k(\delta_{k+1})=2$ by \eqref{eq:del1} and $\beta_k(\delta)$ admits the unique minimum point $\delta_k^*$ in $(\delta_{k+1},\delta_k)$. Here, we have that
\[
\delta_k^*=\left(1-\frac{a_k}{2(p-1)}\right)\delta_k
\] 
and
\[
\beta_k^*:=\beta_k(\delta_k^*)=(2+a_k)\left(1-\frac{a_k}{2(p-1)}\right)^{p-1}\in(0,2).
\] 
Hence we can take a sequence $(r_{k,n}^*)$ in the interval $(\rho_{k,n},\bar{\rho}_{k+1,n})$, where no bubble appears, so that $u_n(r_{k,n}^*)/\mu_n\to \delta_k^*$ as $n\to \infty$. Then, \eqref{bo2} ensures that the graph of $u_n(r)$ touches the curve $U_{\beta_k^*+o(1)}(r)$ at $r=r_{k,n}^*$. As a result, we can understand that starting from the highest curve $U_{2+o(1)}(r)$ at $r=r_{k,n}$, as $r$ increases, $u_n(r)$ crosses the lower curve $U_{\beta}(r)$ for all $2>\beta>\beta_k^*$ and finally arrives on the lowest curve $U_{\beta_k^*+o(1)}(r)$ at $r=r_{k,n}^*$. Moreover, as $r$ again increases, it passes the higher curve $U_\beta(r)$ for all $\beta_k^*<\beta<2$ and, at last, reaches the highest curve $U_{2+o(1)}(r)$ at $r=r_{k+1,n}$ again. Since $k\in \mathbb{N}\cup\{0\}$ is arbitrary, this up-and-down behavior is repeated infinitely many times as $k\to \infty$. This shows  the infinite oscillation of the graph of $u_n(r)$. Notice that each oscillation has one to one correspondence to the appearance of each bubble. Hence we conclude that this infinite oscillation is caused by the infinite sequence of bubbles. As a remark, notice also that  since $(a_k)$ and $(\beta_k^*)$ monotonically converges to 0 and 2 respectively by Lemma \ref{lem:bf2}, the amplitudes of the oscillations get smaller and smaller as $k\to \infty$ and finally become vanishingly small for all large $k$. 
 Let us next summarize this observation as the following intersection property. Before that we remark on the argument  in \cite{McMc}.
\begin{remark}\label{rmk:mm} As remarked in Section \ref{sec:intr}, in the earlier work \cite{McMc}, McLeod-McLeod point out the oscillation behavior, bouncing process, of solutions which corresponds to our observation above. See the arguments under Theorem 1.1  there. Interestingly, we further encounter several correspondences among their heuristic discussions and our rigorous ones based on the concentration analysis. For example, the energy function $H$ they introduce in the beginning of Section 2 there corresponds to our $\log{\phi_n}$ where $(\phi_n)$ is the sequence of functions defined in Section \ref{sec:pre} in this paper. Notice that ours is introduced as a direct extension of a key  tool developed in \cite{D} which comes from  the scaling  structure of  the problem. Note also that in their discussion, the critical point of $H$ is interpreted as a point of bounce while in our argument, that of $\phi_n$ is the center $r_{k,n}$ of the $(k+1)$th bubble for all $k\in \mathbb{N}\cup\{0\}$. (Recall our choice of $r_{k,n}$ in the final part of the proof of Lemma \ref{lem:f1}.) In the previous discussion, we have already explained that they coincide with each other. Furthermore, one can confirm that the system of the recurrence formulas \eqref{eq:del1} and \eqref{eq:del2}, which completely characterize each bubble, in the present  paper are equivalent to (2.7) and  the combination of (2.11) and (2.12) there. Their derivation is still a heuristic way based on the analysis of the energy $H$ while ours is  the precise computation of the balance between adjacent two bubbles via the identity \eqref{id2}.  In this way, we observe several correspondences between these two independent works. From the viewpoint of their result, we may also say that  our present work justifies their heuristic observation, moreover, we characterize the bouncing process as a result of the infinite concentration phenomenon. Finally, we emphasize that our proof successfully provides much more precise information on the oscillation behavior as above. Thanks to this, we can readily proceed to the next steps to study the intersection properties and also the oscillations of the bifurcation diagrams.
\end{remark}
\subsection{Intersection properties}
We next deduce various intersection properties between blow-up solutions and singular functions from Theorem \ref{thm:osc}. Notice that in the results below, we give not only the estimates for the intersection numbers but also the precise information on the points where the intersections occur. For any interval $I\subset \mathbb{R}$ and function $u: I\to \mathbb{R}$,  we define $Z_I[(u)]$ as the number of zero points of $u$ on $I$. 
 We first study the intersection properties between  blow-up solutions and the function $U_\beta(r)$ defined in the previous argument.
\begin{corollary}\label{thm:int1} Suppose $p>0$, (H1), and $\{(\la_n,\mu_n,u_n)\}$ is a sequence of solutions of \eqref{q} with $\mu_n\to \infty$ as $n\to \infty$. 
 Then  setting $0<\beta<4/p$ if $0<p\le 2$ and $0<\beta<2$ if $p> 2$, we get, up to a subsequence,
\[
Z_{(0,r_{0,n}]}[u_n-U_{\beta}]\ge1
\]
for all large $n\in \mathbb{N}$.  Moreover,  we have the following. 
\begin{enumerate}
\item[(i)] Assume $0<p<1$. Then,  for any $\beta\ge p/4$, choosing $\delta\in(0,1)$ so that $4\delta^{p-1}/p=\beta$, we have a sequence $(r_n)\subset (0,1)$ such that ${u}_n(r_n)/\mu_n\to \delta$ as $n\to \infty$ and ${u}_n(r_n)=U_{\beta}(r_n)$ for all large $n\in \mathbb{N}$.  In particular, we get 
\[
Z_{(0,r_n]}[{u}_n-U_{\beta}]\ge1
\]
for all large $n\in \mathbb{N}$. 
\item[(ii)]
Suppose $1<p\le 2$. Then, for every $0<\beta<4/p$, selecting $\delta\in(0,1)$ so that $4\delta^{p-1}/p=\beta$, we find a sequence $(r_n)\subset ({r}_{0,n},1)$ such that ${u}_n(r_n)/\mu_n\to \delta$  as $n\to \infty$  and ${u}_n(r_n)=U_\beta(r_n)$ for all large $n\in \mathbb{N}$. Especially,  we obtain  
\[
Z_{(0,r_n]}[{u}_n-U_{\beta}]\ge2
\]
for all large $n\in \mathbb{N}$.
\item[(iii)] Assume $p>2$, $k\in \mathbb{N}\cup \{0\}$, and  $\beta_k^*$ and $\delta_k^*$ are the constants defined above. Then, for each  $\beta\in (\beta_k^*,2)$, taking $\delta_k>\delta>\delta_k^*>\delta'>\delta_{k+1}$ so that $\beta_k(\delta)=\beta_k(\delta')=\beta$, we get sequences $(r_n),(r_n')\subset ({r}_{k,n},{r}_{k+1,n})$ such that  ${u}_n(r_n)/\mu_n\to \delta$ and ${u}_n(r_n')/\mu_n\to \delta'$ as $n\to \infty$ and ${u}_n(r_n)=U_\beta(r_n)$ and ${u}_n(r_n')=U_\beta(r_n')$ for all large $n\in \mathbb{N}$. Particularly, 
 we have
\[
Z_{(0,r_n']}[{u}_n-U_{\beta}]\ge 1+2(k+1)
\]
for all large $n\in \mathbb{N}$.
\end{enumerate}
\end{corollary}
Note that in the case (i), $u_n$ intersects $U_\beta$ at least once for any $\beta>0$. Moreover, the intersection number $2$ in (ii) is a consequence of the one time oscillation caused by one bubble. Finally, (iii) picks up two intersection points caused by down-and-up behavior between $(k+1)$th and $(k+2)$th bubble.
 
Next, let us next check  the delicate case $p\ge2$ and $\beta=2$ where the intersection numbers may diverge to infinity. To simplify the statement, we assume that $(\la_n)$ converges to a finite value  which is reasonable with the assumption  (H2) by the last assertion in Theorem \ref{thm:gl}.  Moreover, for the later application, we  focus on the following case.
\begin{enumerate}
\item[(H3)]  There exist values $c_0,\tau_0>0$ and $m\in \mathbb{R}$ such that $f(t)=c_0t^me^{t^p}$ for all $t\ge \tau_0$. 
\end{enumerate} 
Note that (H3) implies (H1). Under this condition,  for any $L\in \mathbb{R}$,  we set a function 
\[
\begin{split}
V_L(r)=\left\{2 \log{\frac1{r}}-\left(1-\frac{1-m}{p}\right) \log{\left(2\log{\frac1{r}}\right)}+L\right\}^{\frac1p}
\end{split}
\]
for all $r>0$. We also recall the function $U_\beta$ in the previous discussion. First we consider the case $p=2$.  
\begin{corollary}\label{thm:int2} Assume (H3) with $p=2$. Suppose $\{(\la_n,\mu_n,u_n)\}$ is a sequence of solutions of \eqref{q} with $\mu_n\to \infty$ and there exists a value $ \la_*\in[0,\infty)$ such that $\la_n\to \la_*$ as $n\to \infty$. Let $U(r)$ be a continuous function defined for all small $r>0$. Moreover, we assume the next (i) or (ii).
\begin{enumerate} \item[(i)] $\la_*\not=0$ and for any $\e\in(0,2)$, there exist  constants $r_0\in(0,1)$ and $L<\log{(1/(c_0\la_*))}$ such that 
\[
U_{2-\e}(r)\le U(r)\le V_L(r)
\]
for all $r\in(0,r_0)$.
\item[(ii)] $\la_*=0$ and for any $\e\in(0,2)$, there exist values $r_0\in(0,1)$ and $L\in \mathbb{R}$ such that
\[
U_{2-\e}(r)\le U(r)\le V_L(r)
\]
for all $r\in(0,r_0)$. 
\end{enumerate}
Then, up to a subsequence,  there exist sequences $(r_{0,n}^{\pm})\subset (0,1)$ such that $u_n(r_{0,n}^{\pm})/\mu_n\to 1$ as $n\to \infty$, $r_{0,n}^-<r_{0,n}<r_{0,n}^+$, and $u_n(r_{0,n}^\pm)=U(r_{0,n}^\pm)$ for all large $n\in \mathbb{N}$. In particular, we get
\[
Z_{(0,r_{0,n}^+]}[u_n-U]\ge2 
\]
for all large $n\in \mathbb{N}$.
\end{corollary}
In this theorem, we observe the two intersection points caused by the up-and-down behavior around the top of the oscillation or equivalently, around the center of the concentration. We next study the case $p>2$.  
\begin{corollary}\label{thm:int3} Suppose (H3) with $p>2$. Let $\{(\la_n,\mu_n,u_n)\}$ be a sequence of solutions of  \eqref{q} with $\mu_n\to \infty$ and there exists a number $ \la_*\in[0,\infty)$ such that $\la_n\to \la_*$ as $n\to \infty$. Assume $U(r)$ is a continuous function defined for all small $r>0$. Furthermore, we suppose  (i) or (ii) below.
\begin{enumerate} \item[(i)] $\la_*\not=0$ and for any $\e\in(0,2)$ and $L\in \mathbb{R}$, there exists a constant $r_0\in(0,1)$ such that 
\[
U_{2-\e}(r)\le U(r)\le V_L(r)
\]
for all $r\in(0,r_0)$.
\item[(ii)] $\la_*=0$ and for any $\e>0$, there exist values $r_0\in(0,1)$ and $L\in \mathbb{R}$  such that
\[
U_{2-\e}(r)\le U(r)\le V_L(r)
\]
for all $r\in(0,r_0)$. 
\end{enumerate}
Then,  up to a subsequence, for any $k\in \mathbb{N}\cup\{0\}$, there exist sequences $(r_{k,n}^{\pm})\subset (0,1)$ such that $u_n(r_{k,n}^{\pm})/\mu_n\to \delta_k$ as $n\to \infty$, $r_{k,n}^-<r_{k,n}<r_{k,n}^+$,  and $u_n(r_{k,n}^\pm)=U(r_{k,n}^\pm)$ for all large $n \in\mathbb{N}$. In particular, we get 
\[
\lim_{n\to \infty}Z_{(0,r_{k,n}^+]}[u_n-U]\ge 2(k+1)
\]
for all large $n\in \mathbb{N}$ and further, for any sequence $(r_n)\subset (0,1)$ with $u_n(r_n)/\mu_n\to0$ as $n\to \infty$, it holds that
\[
\lim_{n\to \infty}Z_{(0,r_n)}[u_n-U]=\infty.
\]
\end{corollary}
In this corollary, we similarly find that the two intersection points appear around the center of $(k+1)$th bubble for all $k\in \mathbb{N}\cup\{0\}$. As a consequence, we finally show the divergence of the intersection numbers. For example, suppose $f(t)=e^{t^p}$ for all large $t>0$ with $p>2$. For any $\theta>1-1/p$, if we choose  a continuous function $U(r)$ so that 
\[
U(r)=\left\{2\log{\frac1r}-\theta \log{\left(2\log{\frac1{r}}\right)}+o\left(\log{\left(2\log{\frac1{r}}\right)}\right)\right\}^{\frac1p}
\]
as $r\to 0^+$. Then the previous theorem proves that   
\[
\lim_{n\to \infty}Z_{(0,1)}[u_n-U]=\infty.
\]

Finally, for the next application, we give a similar intersection result for any sequence $\{(\bar{r}_n,\bar{\mu}_n,\bar{u}_n)\}$ which solves  the next problem,  
\begin{equation}\label{q1}
\begin{cases}
-\bar{u}_n''-\frac1r\bar{u}_n'=f(\bar{u}_n),\ \bar{u}_n>0\text{ in }(0,\bar{r}_n),\\
\bar{u}_n(0)=\bar{\mu}_n,\ \bar{u}_n'(0)=0=\bar{u}_n(\bar{r}_n).
\end{cases}
\end{equation}
We get the following.
\begin{corollary}\label{thm:int4} Assume $p>2$ and (H3). Let  $\{(\bar{r}_n,\bar{\mu}_n,\bar{u}_n)\}$ be a sequence of solutions of  \eqref{q1} with $\bar{\mu}_n\to \infty$ as $n\to \infty$. Moreover, suppose that $\bar{U}(r)$ is a continuous function defined for all small $r>0$ and that for any values $\e\in (0,2)$ and $L\in \mathbb{R}$, there exists a constant $r_0\in(0,1)$ such that
\[
U_{2-\e}(r)\le \bar{U}(r)\le V_L(r)
\] 
for all  $r\in(0,r_0)$. Then  for any sequence $(r_n)\subset (0,\bar{r}_n]$ such that $\bar{\mu}_n^{-p}\log{(1/r_n)}\to0$ as $n\to \infty$,  we have that
\[
\lim_{n\to \infty}Z_{(0,r_n)}[\bar{u}_n-\bar{U}]=\infty
\]
up to a subsequence.
\end{corollary}
We apply this result to prove  the infinite oscillation of the bifurcation diagram of \eqref{p}.
\subsection{Oscillations of bifurcation diagrams}\label{sec:app}
Finally, we shall demonstrate how our oscillation estimates work for getting the oscillation of the bifurcation diagram  of \eqref{p}. To this end, we  consider the condition,  inspired by  Lemma 5.2 in \cite{GGP2} and  Proposition 2.1 in \cite{IRT}, 
\begin{enumerate}
\item[(H4)] $p>2$, $f\in C^2([0,\infty))$, and there exists a value $\tau_0>0$ such that  
\[
f(t)=\frac{4(p-1)}{p^2}t^{1-2p}e^{t^p}
\]
for all $t\ge \tau_0$. 
\end{enumerate}
Notice that   $(H4)$ implies  (H3). Moreover, we can obviously construct  specific examples of $f$ satisfying (H4).  This assumption  is reasonable for our aim since we can easily check the existence of a solution with a suitable singularity at the origin. See Lemma \ref{lem:sg} below. Moreover, under (H4), we confirm that for any $\mu>0$, there exists  a  unique pair $(\la(\mu),u(\mu,\cdot))$ of a number and a function in $C^{2,\alpha}([0,1])$  such that $(\la,u)=(\la(\mu),u(\mu,\cdot))$ solves \eqref{p} and $u(\mu,0)=\mu$. For the proof, see Theorem 3 in \cite{AtPe} or Lemma 2.1 in \cite{AKG} with   Theorem 2.1 in \cite{NT} and the standard regularity theory. Then, combining this fact with Theorem 2.1 in \cite{K}, one sees that the map $\mu \mapsto (\la(\mu),u(\mu,\cdot))$ draws a $C^1$ curve in $\mathbb{R}\times C^{2,\alpha}([0,1])$. Let us show the  oscillation behavior of this solutions curve as $\mu\to \infty$.  

For this purpose, for each $\mu>0$, we set $\bar{u}(\mu,r)=u(\mu,r/\sqrt{\la})$ for all $r\in[0,\sqrt{\la}]$. Then $\bar{u}=\bar{u}(\mu,\cdot)$ satisfies 
\begin{equation}\label{p1}
\begin{cases}
-\bar{u}''-\frac1r\bar{u}'=f(\bar{u})\text{ in }(0,\bar{r}),\\
\bar{u}(0)=\bar{\mu},\ \bar{u}'(0)=0=\bar{u}(\bar{r}),
\end{cases}
\end{equation}
with $\bar{r}=\sqrt{\la}$ and $\bar{\mu}=\mu$.  Furthermore, for any $\bar{R}>0$, we consider the next probelm, 
\begin{equation}\label{p**}
\begin{cases} -\bar{U}''-\frac1r \bar{U}'=f(\bar{U}),\ \bar{U}> 0\text{ in }(0,\bar{R}),\\
\lim_{r\to 0^+}\bar{U}(r)=\infty,\ \bar{U}(\bar{R})=0.
\end{cases}
\end{equation}
Under (H4), we will prove in Lemma \ref{lem:sg} that there exists a pair of a value $\bar{R}^*$ and a smooth function $\bar{U}^*$on $(0,1]$  such that $\bar{U}^*(r)=(2\log{(1/r)})^{1/p}$  for all small $r>0$ and $(\bar{R},\bar{U})=(\bar{R}^*,\bar{U}^*)$ satisfies \eqref{p**}. Note  also that if we put  $\la^*=(\bar{R}^*)^2$ and $U^*(r)=\bar{U}^*(\sqrt{\la^*}r)$ for all $r\in(0,1]$, $(\la,U)=(\la^*,U^*)$ solves the next problem,
\begin{equation}\label{p*}
\begin{cases} -{U}''-\frac1r {U}'=\la f({U}),\ U> 0\text{ in }(0,1),\\
\lim_{r\to 0^+}{U}(r)=\infty,\ {U}(1)=0.
\end{cases}
\end{equation}
 Corollary \ref{thm:int4} succeeds in  proving that the intersection number between any solution $\bar{u}$ of \eqref{p1} and $\bar{U}^*$ diverges to infinity as $\bar{\mu}\to \infty$.  Then we may apply the argument in \cite{Mi} and get the desired result as follows. 
\begin{theorem}\label{thm:app2} Assume  (H4). Let $\mu \mapsto (\la(\mu),u(\mu,\cdot))$ be the solutions $C^1$-curve of \eqref{p} in $\mathbb{R}\times C^{2,\alpha}([0,1])$ with $u(\mu,0)=\mu$  and $(\la^*,U^*)$ the pair of the number and function obtained as above. Then, $\la(\mu)$ oscillates around $\la^*$ infinitely many times as $\mu\to \infty$. In particular,  we get the following.
\begin{enumerate}
\item[(i)] $(\la,U)=(\la^*,U^*)$ solves \eqref{p*}.
\item[(ii)] There exist a sequence $(\mu_n)\subset (0,\infty)$ and a nonincreasing function $u^*$ on $(0,1]$ such that $\mu_n\to \infty$, $\la(\mu_n)\to \la^*$,  $u(\mu_n,\cdot)\to u^*$ in $C^2_{\text{loc}}((0,1])$ as $n\to \infty$, and $(\la^*,u^*)$ solves \eqref{q*}.  
\item[(iii)] For each $N\in \mathbb{N}$, there exists a value $\e\in(0,\la^*)$ such that \eqref{p} admits at least $N$ distinct solutions for all $\la^*-\e<\la<\la^*+\e$ 
 and infinitely many solutions for $\la=\la^*$.
\item[(iv)] If (H2) holds, we have that 
\[
\lim_{\mu\to \infty}Z_{(0,1)}[u(\mu,\cdot)-U^*]=\infty.
\]
\end{enumerate}   
\end{theorem}
\begin{remark} In this theorem, we mean by the oscillation of $\la(\mu)$  that  there exist sequences $(\mu_n),(\mu_n')\subset (0,\infty)$ such that $\mu_n,\mu_n'\to\infty$ as $n\to \infty$ and $\la(\mu_n')<\la^*<\la(\mu_n)$ for all $n\in \mathbb{N}$. 
\end{remark}
As stated in the present theorem, our oscillation estimates successfully allow us to prove the infinite oscillation of the bifurcation diagram.  Surely, it will enable  us to find more nonlinearities $f$ which yield the infinite oscillation of the bifurcation diagram of \eqref{p}  once the existence and the asymptotic behavior of the solutions of \eqref{p**} are confirmed. Fortunately, this is done for the general  nonlinearity by \cite{FIRT} very recently.  Thanks to this, we can confirm that the same oscillation results hold true for the standard  examples of $f$. See Remark \ref{rmk:B} below.  Finally, we remark that, in (ii) of the previous theorem, we do not prove  that  $u^*=U^*$ or $(\la(\mu),u(\mu,\cdot))\to (\la^*,U^*)$ as $\mu\to \infty$ entirely  since the proofs seem to  require more arguments and to be  beyond the main aim of this paper. Hence, we leave them open here as an interesting question in the future. Let us end this subsection by remarking on a fruitful consequence of the interaction between the existence result in \cite{FIRT} and our oscillation estimates. 
\begin{remark}\label{rmk:B}
As noted above, Fujishima-Ioku-Ruf-Terraneo \cite{FIRT} prove the existence of solutions of \eqref{p**} with precise asymptotic formulas under some general assumptions on $f$. See Theorems 2.1 and 2.2. Moreover, in Examples 4.1 and 4.2, they pick up two typical examples $f(t)=t^me^{t^p}$ and $f(t)=e^{t^p+t^q}$ for large $t>0$, where $m\in\mathbb{R}$, $p>1$, and $0<2q<p$. Then, if $p>2$, combining their results with our oscillation estimates, in each case,  we can check that  the intersection numbers between any blow-up solutions of \eqref{p1} and their singular solution diverge to infinity. Consequently, we get the same assertions with those in Theorem \ref{thm:app2}, that is, we observe the infinite oscillation of the bifurcation diagram for each typical nonlinearity. We will explain more details about the intersection property used here in Remark \ref{rmk:D} below. 
\end{remark}
In the following, we prove the theorems and corollaries above.
\subsection{Proof of Theorem \ref{thm:osc}}
We first give the next two lemmas for the case $p>2$. Let $\{(\la_n,\mu_n,u_n)\}$ be a sequence of solutions of \eqref{q} as usual. The next one  proves \eqref{de1}.
\begin{lemma}\label{lem:osc1} Assume $p>2$, (H1), and $k\in \mathbb{N}$. Let $(r_n)\subset [\bar{\rho}_{k,n},\rho_{k,n}]$ and $(R_n)\subset (0,\infty)$ be sequences of values such that $r_n=R_n \ga_{k,n}$ for all $n\in \mathbb{N}$. Then we get
\[
p\la_nr_n^2 u_n(r_n)^{p-1}f(u_n(r_n)) =\frac{2a_k^2 b_k R_n^{a_k}}{(1+b_k R_n^{a_k})^2}(1+o(1))
\]
as $n\to \infty$.
\end{lemma}
\begin{proof} Put $u_n(r_{k,n})=\mu_{k,n}$ for all $n\in \mathbb{N}$. Similarly to the proof of Lemma \ref{lem:thm1}, using the results in Lemma \ref{lem:e5}, we compute
\[
\begin{split}
p \la_nr_n^2 u_n(r_n)^{p-1}f(u_n(r_n))&=\left(\frac{u_n(r_n)}{\mu_{k,n}}\right)^{p-1}R_n^2\frac{f(u_n(r_n))}{f(\mu_{k,n})}
\\
&=(1+o(1))\frac{2a_k^2 b_k R_n^{a_k}}{(1+b_k R_n^{a_k})^2}.
\end{split}
\]
This proves the lemma. 
\end{proof}
The next one shows \eqref{bo2}.
\begin{lemma}\label{lem:osc2} Assume $p>2$, (H1), and $k\in \mathbb{N}\cup\{0\}$. Then for any sequence $(r_n)\subset [\rho_{k,n},\bar{\rho}_{k+1,n}]$ and value $\delta\in [\delta_{k+1},\delta_k]$ such that $u_n(r_n)/\mu_n\to \delta$ as $n\to \infty$, we obtain
\[
u_n(r_n)=\left\{(\beta_k(\delta)+o(1))\log{\frac1{r_n}}\right\}^{\frac1p}
\]
as $n\to \infty$. 
\end{lemma}
\begin{proof} From \eqref{id2}, we get 
\[
 u_n(\rho_{k,n})-u_n(r_n)\ge\left(\log{\frac{r_n}{\rho_{k,n}}}\right)\int_{0}^{\rho_{k,n}} \la_nf(u_n(r))rdr,
\]
and also
\[
 u_n(\rho_{k,n})-u_n(r_n)
\le \left(\log{\frac{r_n}{\rho_{k,n}}}\right)\int_{0}^{\bar{\rho}_{k+1,n}} \la_nf(u_n(r))rdr.
\]
Hence using the first assertion in \eqref{eq:en1}, we have
\[
p\mu_n^{p-1}\left( u_n(\rho_{k,n})-u_n(r_n)\right)=\left(\log{\frac{1}{\rho_{k,n}}}-\log{\frac{1}{r_n}}\right)\left(E_n(0,\rho_{k,n})+o(1)\right).
\]
Here, recalling our choice of $(\rho_{k,n})$ in Lemmas \ref{lem21} and \ref{lem:e5} and noting \eqref{eq:rk2}, \eqref{eq:en1}, and Lemma \ref{lem:bg}, we get
\[
\frac{\log{\frac{1}{\rho_{k,n}}}}{\mu_n^p}=\frac{\log{\frac{1}{r_{k,n}}}+\log{\frac{r_{k,n}}{\rho_{k,n}}}}{\mu_n^p}=\frac{\delta_k^p}{2}+o(1)\text{\ \ and \ }E_n(0,\rho_{k,n})= \frac{2+a_k}{\delta_k^{p-1}}+o(1).
\]
Using these facts and our assumption $\mu_n=(\delta^{-1}+o(1))u_n(r_n)$ for the previous formula, we calculate 
\[
\begin{split}
p\left(\frac{\delta_k }{\delta}\right)^p&\left(1-\frac{\delta}{\delta_k}+o(1)\right)u_n(r_n)^p\\
&=\left\{\frac{1+o(1)}{2}\left(\frac{\delta_k}{\delta}\right)^pu_n(r_n)^p-\log{\frac{1}{r_n}}\right\}(2+a_k+o(1)).
\end{split}
\]
Solving this formula with respect to $u_n(r_n)$, we get the conclusion. This finishes the proof.  
\end{proof}
Let us show Theorem \ref{thm:osc}
\begin{proof}[Proof of Theorem \ref{thm:osc}]
Assume $p>0$.  First, select any sequence $(r_n)\subset(0,\rho_{0,n}]$. Then Lemma \ref{lem:thm1} proves \eqref{de0}. Suppose, in addition, $\mu_n^{-p}\log{(r_n/\ga_{0,n})}\to0$ as $n\to \infty$. Then since $\log{f(u_n(r_n))}=u_n(r_n)^p(1+o(1))$  from \eqref{ha} and $u_n(r_n)/\mu_n=1+o(1)$, we get \eqref{de01} by \eqref{de0} and \eqref{sub2} and obtain \eqref{de02} by  \eqref{de0}, \eqref{cri2},  and  \eqref{sup3}. Next let us check (i), (ii), and (iii). Assume $0<p<2$ and take any sequences $(r_n)\subset[\rho_{0,n},1)$ and $(\delta_n)\subset (0,1)$ as in (i). 
It follows from \eqref{id2} that 
\[
\left(\log{\frac{1}{r_n}}\right)\int_{0}^{r_n} \la_nf(u_n(r))rdr\le u_n(r_n)\le \left(\log{\frac{1}{r_n}}\right)\int_{0}^{1} \la_nf(u_n(r))rdr.
\]
Then using \eqref{eq:sq} and \eqref{sub1}, we get   
\[
p\mu_n^{p-1}u_n(r_n)=(4+o(1))\log{\frac1{r_n}}.
\]
Hence, substituting $\mu_n=u_n(r_n)\delta_n^{-1}$, we confirm \eqref{bo0}.  
 Next assume $p=2$. Let $(r_n)\subset [\rho_{0,n},1)$ and $\delta\in(0,1]$ be a sequence and a value such that $u_n(r_n)/\mu_n\to \delta$ as $n\to \infty$. Note that $E(0,r_n)=4+o(1)$ by \eqref{eq:sq} and \eqref{cri1}. Then using  \eqref{id2}, similarly to the argument in the proof of Lemma \ref{lem:osc2}, we obtain
\[
2\mu_n(u_n(\rho_{0,n})-u_n(r_n))=\left(\log{\frac{1}{\rho_{0,n}}}-\log{\frac{1}{r_n}}\right)(4+o(1)).
\]
Here, recalling Lemmas \ref{lem:crit2} and \ref{lem21}, we have
\[
\log{\frac1{\rho_{0,n}}}=\log{\frac1{\ga_{0,n}}}+\log{\frac{\ga_{0,n}}{\rho_{0,n}}}=\left(\frac12+o(1)\right)\left(\frac{u_n(r_n)}{\delta}\right)^2.
\]
Substituting this  into the previous formula, we obtain 
\[
2\left(\frac{u_n(r_n)}{\delta}\right)^2(1-\delta+o(1))=\left\{\left(\frac12+o(1)\right)\left(\frac{u_n(r_n)}{\delta}\right)^2-\log{\frac{1}{r_n}}\right\}(4+o(1)).
\]
Solving this with respect to $u_n(r_n)$, we prove \eqref{bo1}. Finally, suppose $p>2$. Then, \eqref{bo2} is proved by Lemma \ref{lem:osc2}. \eqref{de1} follows from Lemma \ref{lem:osc1}. Choose any number $k\in \mathbb{N}$ and sequence $(r_n)\subset [\bar{\rho}_{k,n},\rho_{k,n}]$. Then Lemma \ref{lem:e5} shows that $\mu_n^{-p}\log{(r_n/\ga_{k,n})}=o(1)$  and $u_n(r_n)/\mu_n=\delta_k+o(1)$. Consequently, we get \eqref{de11} by \eqref{de1}, \eqref{ha}, and \eqref{sup3}.  This completes the proof.
\end{proof}
Next we prove the corollaries. In the following proofs, we recall the constants $\delta_k^*$ and $\beta_k^*$ for all $k\in \mathbb{N}\cup\{0\}$ defined below Theorem \ref{thm:osc}.
 First we show Corollary \ref{thm:int1}. 
\begin{proof}[Proof of Corollary \ref{thm:int1}] First assume $p>0$. Set $\beta\in(0,4/p)$ if $0<p\le2$ and $\beta\in(0,2)$ if $p>2$. Recall that  $r_{0,n}/\ga_{0,n}\to 2\sqrt{2}$ as $n\to \infty$ by Theorem \ref{thm1}. Then \eqref{de01} and \eqref{de02} ensure that  if $0<p\le2$,
\begin{equation}\label{de011}
{u}_n(r_{0,n})^p=\left(\frac4p+o(1)\right)\log{\frac{1}{r_{0,n}}}>U_\beta(r_{0,n})^p
\end{equation}
and if $p>2$,
\begin{equation}\label{de021}
{u}_n(r_{0,n})^p=\left(2+o(1)\right)\log{\frac{1}{r_{0,n}}}>U_\beta(r_{0,n})^p
\end{equation}
for all large $n\in \mathbb{N}$. Then since $u_n(0)<\infty$ for all $n\in \mathbb{N}$ and $U_\beta(r)\to \infty$ as $r\to 0^+$, there exists a sequence $(r_n)\subset (0,r_{0,n})$ such that $u_n(r_n)=U_\beta(r_n)$ for all large $n\in\mathbb{N}$. This proves the first assertion. 

 Next we consider the case (i). Assume $0<p<1$. Set $\beta >4/p$ and choose a constant $\delta\in(0,1)$ so that $4\delta^{p-1}/p=\beta$. Then, for any value $\delta'\in(0,\delta)$ and sequence $(r_n')\subset (r_{0,n},1)$ such that $u_n(r_n')/\mu_n\to \delta'$ as $n\to \infty$, we get from \eqref{de01} and \eqref{bo0}  that
\[
u_n(\rho_{0,n})<U_\beta (\rho_{0,n})\ \text{ and } \ 
u_n(r_n')>U_\beta (r_n')
\]
for all large $n \in\mathbb{N}$.  Hence there exists a sequence $(r_n)\subset (\rho_{0,n},r_n')$ such that  $u_n(r_n)=U_\beta(r_n)$ for all large $n\in\mathbb{N}$. \eqref{bo0} yields $u_n(r_n)/\mu_n\to \delta$ as $n\to \infty$. This proves the case $\beta>4/p$. Moreover, note that this sequence satisfies  $U_{4/p}(r_n)<u_n(r_n)$ for all large $n\in \mathbb{N}$. Then recalling the first conclusion of this corollary, we  prove the case $\beta=4/p$. This finishes (i). 
 
Next let us show (ii). Assume $1<p\le2$. For any $\beta\in(0,4/p)$, take a constant $\delta\in(0,1)$ so that $4\delta^{p-1}/p=\beta$. Choose any value $\delta'\in(0,\delta)$ and  sequence $(r_n')\subset (\rho_{0,n},1)$ such that $u_n(r_n')/\mu_n\to \delta'$ as $n\to \infty$. Then we have from \eqref{de01}, \eqref{de02}, \eqref{bo0}, and \eqref{bo1} that
\[
u_n(\rho_{0,n})>U_\beta(\rho_{0,n})\text{ and }u_n(r_n')<U_\beta(r_n')
\] 
for all large $n\in \mathbb{N}$. Hence we find  a sequence $(r_n)\subset (\rho_{0,n},r_n')$ such that $u_n(r_n)=U_\beta(r_n)$ for all large $n\in \mathbb{N}$. \eqref{bo0} and \eqref{bo1} confirm $u_n(r_n)/\mu_n\to \delta$. This finishes (ii). 

Finally, assume $p>2$ and $k\in \mathbb{N}\cup\{0\}$. For any $\beta\in(\beta_k^*,2)$, take values $\delta_k>\delta>\delta_k^*>\delta'>\delta_{k+1}$ so that $\beta_k(\delta)=\beta_k(\delta')=\beta$. Choose any sequence $(r_{k,n}^*)\subset (\rho_{k,n},\bar{\rho}_{k+1,n})$ such that $u_n(r_{k,n}^*)/\mu_n\to \beta_k^*$ as $n\to \infty$. Then  from \eqref{de02}, \eqref{de11} and \eqref{bo2}, we get 
\[
u_n(\rho_{k,n})>U_\beta(\rho_{k,n}),\ u_n(r_{k,n}^*)<U_\beta(r_{k,n}^*),\text{ and }u_n(\bar{\rho}_{k+1,n})>U_\beta(\bar{\rho}_{k+1,n})
\]
for all large $n\in\mathbb{N}$. Therefore, there exist sequences $(r_n)\subset (\rho_{k,n},r_{k,n}^*)$ and $(r_n')\subset (r_{k,n}^*,\bar{\rho}_{k+1,n})$ such that $u_n(r_n)=U_\beta(r_n)$ and $u_n(r_n')=U_\beta(r_n')$ for all large $n\in\mathbb{N}$. \eqref{bo2} implies  that $u_n(r_n)/\mu_n\to \delta$ and  $u_n(r_n')/\mu_n\to \delta'$ as $n\to \infty$. This completes (iii). We finish the proof. 
\end{proof}
To prove Corollaries \ref{thm:int2} and \ref{thm:int3}, we give the next lemma. 
\begin{lemma}\label{lem:int23} Assume $p\ge2$, (H3), and $k=0$ if $p=2$ and $k\in \mathbb{N}\cup\{0\}$ if $p>2$.  Then we get
\[
\begin{split}
u_n(r_{k,n})=\left\{2\log{\frac1{r_{k,n}}}-\left(1-\frac{1-m}{p}\right) \log{\left(2\log{\frac1{r_{k,n}}}\right)}+\log{\frac{a_k^2}{2c_0p \la_n}}+o(1)\right\}^{\frac1p}
\end{split}
\]
as $n\to \infty$.
\begin{proof} The conclusion readily follows from \eqref{eq:r0} with \eqref{de02} if $k=0$ and \eqref{eq:rk} with \eqref{de11} if $k\in \mathbb{N}$.
\end{proof}
\end{lemma}
Let us complete the proof.  
\begin{proof}[Proof of Corollary \ref{thm:int2}] Choose a sequence $(r_n)\subset (r_{0,n},1)$ so that $u_n(r_n)/\mu_n\to 1/2$ as $n\to \infty$. Notice that $r_n\to0$ as $n\to \infty$ by the last assertion in Lemma \ref{lem:gl}. Then under the conditions in  (i) and (ii) respectively, Lemma \ref{lem:int23} and \eqref{bo1} prove that $U(r_{0,n})\le V_L(r_{0,n})<u_n(r_{0,n})$ and $U(r_n)\ge U_{3/2}(r_n)>u_n(r_n)$ for all large $n\in \mathbb{N}$. Hence noting also $u_n(0)<\infty$ for all $n\in \mathbb{N}$ and $U(r)\to \infty$ as $r\to0^+$, we find  sequences $(r_{0,n}^{\pm})\subset (0,1)$ such that $0<r_{0,n}^-<r_{0,n}<r_{0,n}^+<r_n$ and $u_n(r_{0,n}^\pm)=U(r_{0,n}^\pm)$ for all large $n\in \mathbb{N}$. Then  \eqref{bo1} yields that $u_n(r_{0,n}^\pm)/\mu_n\to1$ as $n\to \infty$. This completes the proof.  
\end{proof}
\begin{proof}[Proof of Corollary \ref{thm:int3}] Assume $k\in \mathbb{N}\cup\{0\}$. Choose a sequence $(r_{k,n}^*)\subset (\rho_{k,n},\bar{\rho}_{k+1,n})$ such that $u_n(r_{k,n}^*)/\mu_n\to \delta_k^*$ as $n\to \infty$. In addition, if $k\not=0$, select a sequence $(r_{k-1,n}^*)\subset (\rho_{k-1,n},\bar{\rho}_{k,n})$ so that $u_n(r_{k-1,n}^*)/\mu_n\to \beta_{k-1}^*$ as $n\to \infty$. Then fix  a number $\beta\in(\beta_k^*,2)$. Moreover, in the case of (i), we  choose any number $L<\log{(a_k^2/(2c_0p\la_*)}$. Then we have by the assumption that $U_\beta(r)\le U(r)\le V_L(r)$ for all small $r>0$. On the other hand, in the case of (ii), noting the assumption, we select a value $L>0$ so that $U_\beta(r)\le U(r)\le V_L(r)$ for all small $r>0$. Then, in each case, we get  from Lemma \ref{lem:int23} and \eqref{bo2} that $U(r_{k,n})\le V_L(r_{k,n})<u_n(r_{k,n})$ and $U(r_{k,n}^*)\ge U_{\beta}(r_{k,n}^*)>u_n(r_{k,n}^*)$ for all large $n\in \mathbb{N}$. Furthermore, if $k\ge1$, since $\beta_{k-1}^*<\beta_k^*$, we have that $U(r_{k-1,n}^*)\ge U_\beta (r_{k-1,n^*})>u_n(r_{k-1,n}^*)$ for all large $n\in \mathbb{N}$. Consequently,  we obtain sequences $(r_{k,n}^{\pm})\subset (0,1)$ such that $r_{k-1,n}^*<r_{k,n}^-<r_{k,n}<r_{k,n}^+<r_{k,n}^*$, where we regard $r_{k-1,n}^*=0$ when $k=0$, and  $u_n(r_{k,n}^{\pm})=U(r_{k,n}^{\pm})$ for all large $n\in \mathbb{N}$. Furthermore,  $u_n(r_{k,n}^\pm)/\mu_n\to \delta_k$ as $n\to \infty$ by \eqref{bo2}. This proves the former conclusion. Finally, take any sequence $(r_n)\subset (0,1)$ such that $u_n(r_n)/\mu_n\to0$ as $n\to \infty$. Then the former assertion implies that for any $k\in \mathbb{N}\cup\{0\}$,  we get that $r_{k,n}^+>r_n$ for all large $n\in \mathbb{N}$ and thus, 
\[
Z_{(0,r_n)}[u_n-U]\ge 2(k+1)
\]
for all large $n\in \mathbb{N}$. Since $k$ is arbitrary, we finish the proof.   
\end{proof}
Lastly, we prove Corollary \ref{thm:int4}.
\begin{proof}[Proof of  Corollary \ref{thm:int4}] 
Put $u_n(r)=\bar{u}_n(\bar{r}_nr)$ for all $r\in[0,1]$, $\la_n=\bar{r}_n^2$, and $\mu_n=\bar{\mu}_n$ for all $n\in \mathbb{N}$. Then $\{(\la_n,\mu_n,u_n)\}$ is a sequence of solutions of \eqref{q}. It follows from  Theorems \ref{thm1}, \ref{thm30}, and \eqref{bo2} that, for any $k\in \mathbb{N}\cup \{0\}$, there exist sequences $(r_{k,n}),(r_{k,n}^*)\subset (0,1)$ such that $u_n(r_{k,n})/\mu_n\to \delta_k$, $u_n(r_{k,n}^*)/\mu_n\to \delta_k^*$, 
\[
p\la_nr_{k,n}^2  u_n(r_{k,n})^{p-1}f(u_n(r_{k,n}))= \frac{a_k^2}{2}+o(1),
\]   
and
\[
u_n(r_{k,n}^*)=\left\{(\beta_k^*+o(1))\log{\frac1{r_{k,n}^*}}\right\}^{\frac1p}
\]
as $n\to \infty$. Putting $\bar{r}_{k,n}=\sqrt{\la_n}r_{k,n}$ and $\bar{r}_{k,n}^*=\sqrt{\la_n}r_{k,n}^*$ for all $n\in \mathbb{N}$, these two formulas imply
\begin{equation}\label{ww0}
p\bar{r}_{k,n}^2  \bar{u}_n(\bar{r}_{k,n})^{p-1}f(\bar{u}_n(\bar{r}_{k,n}))=\frac{a_k^2}{2}+o(1)
\end{equation}
and 
\begin{equation}\label{ww1}
\bar{u}_n(\bar{r}_{k,n}^*)=\left\{(\beta_k^*+o(1))\log{\frac{1}{\bar{r}_{k,n}^*}}\right\}^{\frac1p}
\end{equation}
where we used \eqref{sup3}, \eqref{eq:rk2}, and the fact that $r_{k,n}<r_{k,n}^*<r_{k+1,n}$ for all large $n\in \mathbb{N}$. Furthermore, similarly to the proof of Lemma \ref{lem:int23}, we deduce from \eqref{ww0} that 
\begin{equation}\label{ww2}
\begin{split}
&\bar{u}_n(\bar{r}_{k,n})
=\left\{2 \log{\frac1{\bar{r}_{k,n}}}
-\left(1-\frac{1-m}{p}\right) \log{\left(2\log{\frac1{\bar{r}_{k,n}}}\right)}+\log{\frac{a_k^2}{2c_0p }}+o(1)\right\}^{\frac1p}.
\end{split}
\end{equation}
Then by  \eqref{ww1}, \eqref{ww2}, and our assumption, we get $\bar{U}(\bar{r}_{k,n})<\bar{u}_n(\bar{r}_{k,n})$ and $\bar{u}_n(\bar{r}_{k,n}^*)<\bar{U}(\bar{r}_{k,n}^*)$ for all large $n\in \mathbb{N}$. Hence there exists a sequence $(\bar{s}_{k,n})\subset(\bar{r}_{k,n},\bar{r}_{k,n}^*)$ such that $\bar{u}_n(\bar{s}_{k,n})=\bar{U}(\bar{s}_{k,n})$. Finally, choose any sequence $(r_n)\subset (0,1)$ so that $\bar{\mu}_n^{-p}\log{(1/r_n)}\to 0$  as $n\to \infty$. Then, we have that $\log{(1/r_n)}/\log{(1/\bar{s}_{k,n})}
\to0$ as $n\to \infty$ by \eqref{sup3} and \eqref{eq:rk2}. In particular, $\bar{s}_{k,n}<r_n$ for all large $n\in \mathbb{N}$. This implies that $Z_{(0,r_n)}[\bar{u}_n-\bar{U}]\ge k+1$ for all large $n\in \mathbb{N}$. Since $k$ is arbitrary, we get $Z_{(0,r_n)}[\bar{u}_n-\bar{U}]\to \infty$ as $n\to \infty$. This finishes the proof. 
\end{proof}
\subsection{Proof of  Theorem \ref{thm:app2}}
Finally, we shall show Theorem \ref{thm:app2}. To this end, we first prove the existence of a solution of \eqref{p**}.
\begin{lemma}\label{lem:sg} Assume (H4). Then there exists a pair of a value $\bar{R}^*>e^{-\tau_0^p/2}$ and a $C^2$ function $\bar{U}^*$ on $(0,\bar{R}^*]$ such that $\bar{U}^*(r)=(2\log{(1/r)})^{1/p}$ for all $r\in(0,e^{-\tau_0^p/2}]$ and $(\bar{R},\bar{U})=(\bar{R}^*,\bar{U}^*)$ satisfies \eqref{p**}. 
\end{lemma} 
\begin{proof} We refer to the idea in the proof of  Proposition 2.1 in \cite{IRT} and Lemma 2.1 in \cite{AKG}. Set $R_0=e^{-\tau_0^p/2}$. First  it is easy to check that $-\{(2\log{(1/r)})^{1/p}\}''-\{(2\log{(1/r)})^{1/p}\}'/r=f((2\log{(1/r)})^{1/p})$ and $(2\log{(1/r)})^{1/p}>0$ for all $r\in(0,R_0]$. Next 
 we solve the initial value problem,
\begin{equation}\label{sg1}
\begin{cases}
-\bar{U}''-\frac1r \bar{U}'=f(\bar{U}) \text{ in }(0,\infty),\\
\bar{U}(R_0)=\tau_0, \bar{U}'(R_0)=-2/(p\tau_0^{p-1}R_0).
\end{cases}
\end{equation}
By (H4), there exists a value $\e>0$ such that \eqref{sg1} admits a  unique solution $\bar{U}$ on $(0,R_0+\e)$ and $\bar{U}(r)=(2\log{(1/r)})^{1/p}$ for all $r\in (0,R_0]$. Define the  value 
\[
\begin{split}
R^*:=\sup\{r>R_0\ |\ &\text{the solution $\bar{U}(s)$  of \eqref{sg1} exists and $\bar{U}(s)>0$ }\\
&\ \ \ \ \ \ \  \ \ \ \ \ \ \ \ \ \ \ \ \ \ \ \ \ \ \ \ \ \ \ \ \ \ \ \ \ \ \ \ \ \ \ \text{for all $s\in (0,r)$}\}.
\end{split}
\]
Then, for any $r\in(R_0,R^*)$,  integrating the equation gives  
\[
-r\bar{U}'(r)=-R_0\bar{U}'(R_0)+\int_{R_0}^{r}f(\bar{U})rdr>0
\]
by the initial condition and (H0).  In particular, $\bar{U}(r)$ is strictly decreasing on $(0,R^*)$.  Moreover, we have that $R^*<\infty$. If not, $\bar{U}>0$ on $(0,\infty)$. Then,  for any $r\in(R_0+1,\infty)$,  integrating the equation, we get that
\[
-\bar{U}'(r)
\ge\left(-R_0\bar{U}'(R_0)+\int_{R_0}^{R_0+1}f(\bar{U})sds\right)\frac1r
\]
by (H0). For any $r\in(R_0+1,\infty)$, integrating again gives
\[
0<\bar{U}(r)\le\bar{U}(R_0+1)-\left(-R_0\bar{U}'(R_0)+\int_{R_0}^{R_0+1}f(\bar{U})sds\right)\log{\frac{r}{R_0+1}}.
\]
Since the right-hand side diverges to $-\infty$ as $r\to \infty$, we get  a contradiction. Then noting the monotonicity, we obtain $\lim_{r\to R^{*-}}\bar{U}(r)=0$. Therefore, setting $\bar{R}^*=R^*$ and $\bar{U}^*=\bar{U}$ on $(0,\bar{R}^*)$ with $\bar{U}^*(\bar{R}^*)=0$, we get the desired solution. This finishes the proof.   
\end{proof}
Now, we can show the oscillation of the bifurcation diagram by combining our intersection estimates and the idea in the proof of Lemma 3.5 in \cite{Mi}. 
\begin{lemma}\label{lem:app1} Suppose (H4). Let $\mu \mapsto (\la(\mu),u(\mu,\cdot))$ be the solutions $C^1$-curve of \eqref{p} and $\la^*$ the constant  as in the assumption of Theorem \ref{thm:app2}.  Then we get that  $\la(\mu)$ oscillates around $\la^*$ infinitely many times as $\mu\to \infty$.  
\end{lemma}
\begin{proof} 
 For every $\mu>0$, put $r(\mu)=\sqrt{\la(\mu)}$ and  $\bar{u}(\mu,r)=u(\mu,r/\sqrt{\la(\mu)})$ for all $r\in[0,r(\mu)]$. Note that from our assumption, $r(\mu)$ and $\bar{u}(\mu,r)$ are $C^1$ functions with respect to each variable. Moreover, let $(\la^*,U^*)$ be the pair of the number and  function  in the assumption of Theorem \ref{thm:app2}. Set $r^*=\sqrt{\la^*}$ and  $\bar{U}^*(r)=U^*(r/\sqrt{\la^*})$ for all $r\in (0,r^*]$.  Then for any $\mu>0$,  $\bar{u}=\bar{u}(\mu,\cdot),\bar{U}^*$  are solutions of the equation, 
\begin{equation}\label{ap1}
-\bar{u}''-\frac1r \bar{u}'=f(\bar{u}) \text{ in }(0,\min\{r(\mu),r^*\}].
\end{equation}
Put $v(\mu,r)=\bar{U}^*(r)-\bar{u}(\mu,r)$ for all $\mu>0$ and $r\in(0,\min\{r(\mu),r^*\}]$. Then, $v(\mu,r)$  is  a $C^1$ function with respect to each variable. It also holds that, for each $\mu>0$,
\begin{equation}\label{ap10}
\lim_{r\to 0^+}v(\mu,r)=\infty.
\end{equation}
 Moreover, we obtain that 
\begin{equation}\label{ap2}
\lim_{\mu\to \infty}Z_{(0,\min\{r(\mu),r^*\}]}[v(\mu,\cdot)]=\infty.
\end{equation}
If not, there exists a sequence $(\mu_n)\subset (0,\infty)$ such that $\mu_n\to \infty$ as $n\to \infty$ and 
\[
\sup_{n\in \mathbb{N}}Z_{(0,\min\{r(\mu_n),r^*\}]}[v(\mu_n,\cdot)]<\infty.
\]
Then noting  $\mu_n^{-p}\log{(\min\{r(\mu_n),r^*\})}\to 0$ as $n\to \infty$ by \eqref{sup3}, we get a contradiction  with Corollary \ref{thm:int4}. This proves \eqref{ap2}. Next, we claim the following. \vspace{0.1cm}\\
\textbf{Claim 1.} For each $\mu> 0$, $Z_{(0,\min\{r(\mu),r^*\}]}[v(\mu,\cdot)]$ is finite. If this is not true, for some $\mu>0$, noting \eqref{ap10}, we find  a value $r_0\in(0,\min\{r(\mu),r^*\}]$ and a sequence $(r_n)\subset (0,\min\{r(\mu),r^*\}]$ such that $r_n\not=r_0$ and  $v(r_n)=0$ for all $n\in \mathbb{N}$ and $r_n\to r_0$ as $n\to \infty$. It follows that $v(\mu,r_0)=0$ and $v_r(\mu,r_0)=0$ where $v_r$ denotes the partial derivative of $v$ with respect to $r$. This  contradicts the uniqueness of solutions of the initial value problem of \eqref{ap1}. This proves Claim 1. \vspace{0.1cm}

Next we prove the following. 
\vspace{0.1cm}\\ 
\textbf{Claim 2.} For any value $\mu_0>0$, let $\rho_0$ be any zero point of $v(\mu_0,\cdot)$ on $(0,\min\{r(\mu_0),r^*\})$. Then there exist a value $\e>0$ and a $C^1$ function $\rho(\mu)$ defined for all $\mu\in[\mu_0-\e,\mu_0+\e]$  such that $\rho(\mu_0)=\rho_0$ and that if $\mu\in[\mu_0-\e,\mu_0+\e]$ and $r\in[\rho_0-\e,\rho_0+\e]$, then $v(\mu,r)=0$ if and only if $r=\rho(\mu)$. In fact, from the uniqueness on \eqref{ap1} again, we get $v_r(\mu_0,\rho_0)\not=0$ and then, the implicit function theorem proves the claim. In addition, we claim that  if  $r(\mu_0)=r^*$ for some $\mu_0>0$, then there exists a value $\e>0$ such that $Z_{(0,\min\{r(\mu),r^*\}]\cap [r^*-\e,r^*]}[v(\mu,\cdot)]\le 1$ for all $\mu\in[\mu_0-\e,\mu_0+\e]$. If the claim does not hold, then there exist  sequences $(\mu_n)$ in some neighborhood  of $\mu_0$ and $(\rho_n),(\rho_n')\subset(0,\min\{r(\mu_n),r^*\}]$ such that $\rho_n<\rho_n'$ and  $v(\mu_n,\rho_n)=v(\mu_n,\rho_n')=0$ for all $n\in\mathbb{N}$ and $\mu_n\to \mu_0$ and $r^*\ge\rho_n'\ge\rho_n\to r^*$ as $n\to \infty$. It follows that  $v_r(\mu_0,r^*)=0$. But, this is impossible again by the uniqueness assertion on  \eqref{ap1}. 
\vspace{0.1cm} 

Claim 2 implies that, as $\mu$ increases or decreases, any zero point of $v(\mu,\cdot)$ on $(0,\min\{r(\mu),r^*\}]$ does not split into more than two ones and any two zero points do not collide with each other. Moreover noting also \eqref{ap10},  we have that any change of the number of zero points on $(0,\min\{r(\mu),r^*\}]$ occurs only at $r=r(\mu)$ and  only when $r(\mu)$ leaves  or touches $r^*$, and at the moment, the number changes by at most one. In particular, if  $r(\mu_0)<r^*$ ($r(\mu_0)>r^*$), then by \eqref{ap2} and Claim 1, we find  a value $\bar{\mu}>\mu_0$ such that $r(\mu)<r^*$ ($r(\mu)>r^*$ respectively) and $Z_{(0,\min\{r(\mu),r^*\}]}[v(\mu,\cdot)]=Z_{(0,\min\{r(\mu_0),r^*\}]}[v(\mu_0,\cdot)]$ for all $\mu_0\le \mu<\bar{\mu}$ and $r(\bar{\mu})=r^*$.  On the other hand, if  $r(\mu_0)=r^*$ for some $\mu_0>0$, we put $\tilde{\mu}=\sup\{\mu\ge \mu_0\ |\ r(m)=r^*\text{ for all }\mu_0\le m\le \mu\}$.  Then similarly we get $\tilde{\mu}<\infty$. Note also that $r(\tilde{\mu})=r^*$ by continuity and that $Z_{(0,r^*]}[v(\mu,\cdot)]=Z_{(0,r^*]}[v(\mu_0,\cdot)]$ for all $\mu_0\le \mu\le \tilde{\mu}$ by Claim 2.  Now we proceed to the conclusion of the proof.  To this end, we next claim the following. 
\vspace{0.1cm}\\
\textbf{Claim 3.} There exists a sequence $(\mu_n)\subset (0,\infty)$ such that $\mu_n\to \infty$ as $n\to \infty$ and $r(\mu_n)>r^*$ for all $n\in\mathbb{N}$. Otherwise, we have a value $\mu_0>0$ such that $\min\{r(\mu),r^*\}=r(\mu)$ for all $\mu\ge \mu_0$.  It suffices to consider  the next two cases A and B. In the following, for all $\mu>0$, let $\bar{r}(\mu)\in(0,r(\mu))$  be the largest zero point of $v(\mu,\cdot)$ on $(0,r(\mu))$. \vspace{0.1cm}\\
\textbf{Case A.} Assume $r(\mu_0)<r^*$. Then recall the number $\bar{\mu}>\mu_0$ above.  There are only two cases. The first one is that $\sup_{\mu\in [\mu_0,\bar{\mu}]}\bar{r}(\mu)<r^*$. This implies that a new zero point is added when $\mu$ reaches  $\bar{\mu}$ from below. In particular, we get  
 $Z_{(0,r(\bar{\mu})]}[v(\bar{\mu},\cdot)]=Z_{(0,r(\mu_0)]}[v(\mu_0,\cdot)]+1$. Moreover, since $v(\mu,r)>0$ for all $r\in(\bar{r}(\mu),r(\mu))$ and $\mu\in[\mu_0,\bar{\mu})$, we have that  $v(\bar{\mu},r)>0$ for all $\bar{r}(\bar{\mu})<r<r(\bar{\mu})=r^*$. In particular,  $v_r(\bar{\mu},r^*)<0$. The second case is that $r^*>r(\mu)>\bar{r}(\mu)\to r^*$ as $\mu\to \bar{\mu}$ from below. In this case, by Claim 2, $Z_{(0,r(\mu)]}[v(\mu,\cdot)]=Z_{(0,r(\mu_0)]}[v(\mu_0,\cdot)]$ for all $\mu\in[\mu_0,\bar{\mu}]$. Moreover, since $v(\mu,r)<0$ for all $r\in(\tilde{r}(\mu),\bar{r}(\mu))$ and $\mu\in[\mu_0, \bar{\mu})$ and $\sup_{\mu\in[\mu_0, \bar{\mu})}\tilde{r}(\mu)<r^*$, where $\tilde{r}(\mu)>0$ is the largest zero point of $v(\mu,\cdot)$ on $(0,\bar{r}(\mu))$, we have that $v(\bar{\mu},r)<0$ for all $\bar{r}(\bar{\mu})<r<r^*$. Especially, we have $v_r(\bar{\mu},r^*)>0$.\vspace{0.1cm} \\
\textbf{Case B.} Suppose $r(\mu_0)=r^*$. Set the constant  $\tilde{\mu}\ge \mu_0$ as above. From the uniqueness on \eqref{ap1}, we have only two cases $v_r(\mu_0,r^*)<0$ and $v_r(\mu_0,r^*)>0$. In the former case,  from the uniqueness again,  we get that $v_r(\mu,r^*)<0$ for all $\mu_0\le \mu\le \tilde{\mu}$. This yields $v(\tilde{\mu},r)>0$ for all  $\bar{r}(\tilde{\mu})<r<r(\tilde{\mu})$. Moreover, from the definition of $\tilde{\mu}$ and our assumption, when $\mu$ increases from $\tilde{\mu}$, $r(\mu)$ must begin to leave $r^*$ for the left direction. At the moment, the zero point at $r=r(\tilde{\mu})$ vanishes and thus,  there exists a number $\hat{\mu}>\tilde{\mu}$ such that $r(\hat{\mu})<r^*$ and $Z_{(0,r(\hat{\mu})]}[v(\hat{\mu},\cdot)]=Z_{(0,r(\tilde{\mu})]}[v(\tilde{\mu},\cdot)]-1=Z_{(0,r(\mu_0)]}[v(\mu_0,\cdot)]-1$. In the latter case, again from the uniqueness, $v_r(\mu,r^*)>0$ for all $\mu_0\le \mu\le \tilde{\mu}$. This implies that $v(\tilde{\mu},r)<0$ for all $\bar{r}(\tilde{\mu})<r<r^*$. Then, again when $\mu$ increases from $\tilde{\mu}$, $r(\mu)$ must start leaving $r^*$ for the left direction, and in this moment, we have that $v(\mu,r(\mu))>0$.
 Then, the zero point at $r=r(\tilde{\mu})$ does not vanish and continuously moves to the left direction as $\mu$ increases from $\tilde{\mu}$. In particular, we have a value $\hat{\mu}>\tilde{\mu}$ such that $r(\hat{\mu})<r^*$ and $Z_{(0,r(\hat{\mu})]}[v(\hat{\mu},\cdot)]=Z_{(0,r(\tilde{\mu})]}[v(\tilde{\mu},\cdot)]=Z_{(0,r(\mu_0)]}[v(\mu_0,\cdot)]$.\vspace{0.1cm}

Roughly speaking, in view of the consequences in Cases A and B, as $\mu$ increases from $\mu_0$ to infinity, $r(\mu)$ repeats the behaviors in these two cases. This yields that $Z_{(0,r(\mu)]}[v(\mu,\cdot)]\le Z_{(0,r(\mu_0)]}[v(\mu_0,\cdot)]+1$ for all $\mu\ge \mu_0$. In particular, by Claim 1,  $Z_{(0,r(\mu)]}[v(\mu,\cdot)]$ is uniformly  bounded for all $\mu\ge\mu_0$. Hence we get  a contradiction with  \eqref{ap2}. This proves Claim 3. More rigorous proof is the following. \vspace{0.1cm}

\noindent \textbf{Conclusion for Claim 3.} We claim that  $Z_{(0,r(\mu)]}[v(\mu,\cdot)]\le Z_{(0,r(\mu_0)]}[v(\mu_0,\cdot)]+1$ for all $\mu\ge \mu_0$. To prove this, noting the argument above and replacing $\mu_0$ with a larger value if necessary, we may assume $r(\mu_0)<r(\mu)$ without loss of  generality. Let $\mu^*:=\sup\{\mu\ge \mu_0\ |\ Z_{(0,r(m)]}[v(m,\cdot)]\le Z_{(0,r(\mu_0)]}[v(\mu_0,\cdot)]+1\text{ for all }\mu_0\le m \le \mu\}$. Suppose that $\mu^*<\infty$ on the contrary. Then, in view of Claim 2, we have that  $r(\mu^*)=r^*$ and $Z_{(0,r(\mu^*)]}[v(\mu^*,\cdot)]=Z_{(0,r(\mu_0)]}[v(\mu_0,\cdot)]+2$. From the continuity of $r(\mu)$, we have $\mu^*>\mu_0$. Moreover, in view of  the definition of $\mu^*$ and Claim 2, when $\mu$ decreases from $\mu^*$, $r(\mu)$ must leave $r^*$ for the left direction. Hence there exists a value $\mu_0<\mu_*<\mu^*$ such that $r(\mu_*)<r^*$ and $Z_{(0,r(\mu_*)]}[v(\mu_*,\cdot)]=Z_{(0,r(\mu_0)]}[v(\mu_0,\cdot)]+1$. Set $\mu_{**}:=\inf\{\mu\in [\mu_0,\mu_*]\ |\ Z_{(0,r(m)]}[v(m,\cdot)]= Z_{(0,r(\mu_0)]}[v(\mu_0,\cdot)]+1\text{ for all }\mu\le m \le \mu_*\}$. Claim 2, the continuity of $r(\mu)$, and the conclusions in Case B imply $r(\mu_{**})=r^*$, $\mu_0<\mu_{**}< \mu_*$, and  $Z_{(0,r(\mu_{**})]}[v(\mu_{**},\cdot)]= Z_{(0,r(\mu_0)]}[v(\mu_0,\cdot)]+1$. If $v_r(\mu_{**},r^*)<0$, recalling the former discussion in Case B, we find  a value $\mu_{**}<\mu^{**}<\mu_*$ such that $Z_{(0,r(\mu^{**})]}[v(\mu^{**},\cdot)]= Z_{(0,r(\mu_0)]}[v(\mu_0,\cdot)]$. This contradicts the definition of $\mu_{**}$. On the other hand, if  $v_r(\mu_{**},r^*)>0$, similarly to the latter argument in Case B, we can show that as $\mu$ slightly increases or decreases from $\mu_{**}$, the zero point at $r=r(\mu_{**})$ stays at $r^*$ or begins to  move  to the left direction but does not vanish. In particular, $Z_{(0,r(\mu)]}[v(\mu,\cdot)]=Z_{(0,r(\mu_0)]}[v(\mu_0,\cdot)]+1$ for all $\mu$ in some neighborhood of $\mu_{**}$. This again contradicts the definition of $\mu_{**}$. Consequently, we get $\mu^*=\infty$. Therefore, by Claim 1, we arrive at a contradiction with \eqref{ap2}.  This completes the proof of Claim 3.   
 \vspace{0.1cm}

We next show  the following. \vspace{0.1cm}\\
\textbf{Claim 4.} There exists a sequence $(\mu_n')\subset (0,\infty)$ such that $\mu_n'\to \infty$ as $n\to \infty$ and $r(\mu_n')<r^*$ for all $n\in \mathbb{N}$. Assume that there exists a value $\mu_0>0$ such that $r(\mu)\ge r^*$ for all $\mu\ge \mu_0$ on the contrary. Then we get the desired contradiction similarly to the argument for Claim 3. \vspace{0.1cm}

Consequently, Claims 3 and 4 show that $r(\mu)$ oscillates around $r^*$ infinitely many times as $\mu\to \infty$. Then, recalling that $\la(\mu)=r(\mu)^2$ and $\la^*=(r^*)^2$, we prove the desired conclusion of the lemma. This completes the proof.
\end{proof}
We shall show Theorem \ref{thm:app2}
\begin{proof}[Proof of Theorem \ref{thm:app2}] The former conclusion follows from Lemma \ref{lem:app1}. Moreover, (i)  is clear. (ii) is proved by the first conclusion and Theorem \ref{thm:gl}. (iii) is confirmed  by the continuity and the oscillating behavior of $\la(\mu)$ proved above. (iv) is a consequence of Corollary \ref{thm:int3}. We finish the proof.
\end{proof}
We end this section by giving the proof for the discussion in Remark \ref{rmk:B}.
\begin{remark}\label{rmk:D} Finally, let us check the intersection properties for the two examples in Remark \ref{rmk:B}. First assume $f$ satisfies $f(t)=t^m e^{t^p}$ for all large $t>0$ with $m\in \mathbb{R}$ and $p>2$  and some suitable conditions. In this case, by \cite{FIRT},  we confirm that \eqref{p**} admits a solution $\bar{U}$ such that
\[
\bar{U}(r)=\left\{2\log{\frac1r}-\left(2-\frac{1-m}{p}\right)\log{\left(2\log{\frac1{r}}\right)}+\log{\frac{4(p-1)}{p^2}}\right\}^{\frac1p}+o(1)
\] 
as $r\to0^+$. See Theorema 2.1, 2.2, and Example 4.1 there. Hence by Corollary \ref{thm:int4}, the intersection numbers between blow-up solutions of \eqref{q1} and $\bar{U}$ diverge to infinity. This allows us to obtain the same conclusion with \eqref{ap2} in the proof of Lemma \ref{lem:app1} and thus, we complete the desired oscillation assertions. Next suppose $f$ verifies $f(t)=e^{t^p+t^q}$ for all large $t>0$ with $p>2$ and $0<2q<p$ and some appropriate assumptions. Put 
\[
w(r)=\frac{p'}{4}r^2\left(\log{\frac1{r^2}}+1\right)
\]
where $1/p+1/p'=1$. Then from \cite{FIRT}, we ensure  the existence of solutions $\bar{U}$ of \eqref{p**} such that
\[
\begin{split}
\bar{U}(r)&=\left\{-\log{w}-(-\log{w})^{\frac qp}-\left(1-\frac{1}{p}\right)\log{(-\log{w})}+\log{\frac1p}\right\}^{\frac1p}+o(1)
\end{split}
\] 
as $r\to 0^+$. Check Example 4.2 with Remark 4.1 there. It follows that 
\begin{equation}\label{dd1}
\begin{split}
\bar{U}&(r)^p+\bar{U}(r)^q\\
&= 2\log{\frac1r}-\log{\left(2\log{\frac1r}\right)}-\left(1-\frac{1}{p}\right)\log{\left\{\log{\left(\frac1{r^2\left(\log{\frac1{r^2}}+1\right)}\right)}\right\}}+O(1)
\end{split}
\end{equation}
as $r\to0^+$. On the other hand, let $\{(\bar{r}_n,\bar{\mu}_n,\bar{u}_n)\}$ be any sequence of solutions of \eqref{q1} such that $\bar{\mu}_n\to \infty$ as $n\to \infty$. Then similarly to the proof of \eqref{ww2}, for each $k\in \mathbb{N}\cup\{0\}$, we find  a sequence $(\bar{r}_{k,n})$ such that $\bar{r}_{k,n}\to0$ and
\begin{equation}\label{dd2}
\begin{split}
&\bar{u}_n(\bar{r}_{k,n})^p+\bar{u}_n(\bar{r}_{k,n})^q=2 \log{\frac1{\bar{r}_{k,n}}}
-\left(1-\frac{1}{p}\right) \log{\left(2\log{\frac1{\bar{r}_{k,n}}}\right)}+O(1)
\end{split}
\end{equation}
as $n\to \infty$. In particular, $\bar{U}(\bar{r}_{k,n})<\bar{u}_{k,n}(\bar{r}_{k,n})$ for all large $n\in \mathbb{N}$. Hence using \eqref{dd1} and \eqref{dd2}, similarly to Corollary \ref{thm:int4}, we get that the intersection number between $\bar{u}_n$ and $\bar{U}$ diverges to infinity as $n\to \infty$. Then analogously to the former case, we arrive at the desired oscillation result. We finish the proof. 
\end{remark}
\section{Remarks on limit cases $p\to2^+$ and $p\to \infty$}\label{sec:lim}
As a final remark, we study the limit cases $p\to 2^+$ and $p\to \infty$. For each $p>2$, choose a function $h$ satisfying (H1). We write the dependence of $h$ on $p$ as $h=h_p$. For any sequence $(p_n)\subset (2,\infty)$, we put $f_n(t)=h_{p_n}(t)e^{t^{p_n}}$ and consider a sequence $\{(\la_n,\mu_n,v_n)\}$ of solutions  of the next problem,  
\begin{equation}\label{eq:q2}
\begin{cases}
-v_n''-\frac1r v_n'=\la_n f_n(v_n),\ v_n>0\text{ in }(0,1),\\
v_n(0)=\mu_n,\ v_n(1)=0=v_n'(0).
\end{cases}
\end{equation}
We shall remark on the behavior of $(v_n)$ when $\mu_n\to \infty$ as $n\to \infty$. 
\subsection{Infinite sequence of bubbles}
We begin with discussing what happens on the infinite sequence $(z_k)$ of the limit profiles, obtained in Theorem \ref{thm30}, in the limits $p\to 2^+$ and $\infty$ respectively. To do this, we recall the sequences $(\delta_k)$ and $(a_k)$ of numbers  defined by \eqref{eq:del1} and \eqref{eq:del2} and  $(\beta_k^*)$ below Theorem \ref{thm:osc}. For each $k\in \mathbb{N}\cup\{0\}$, we  write the dependence on $p>2$ by $\delta_k=\delta_k(p)$, $a_k=a_k(p)$, $\beta_k^*=\beta_k^*(p)$, and $z_k=z_k^{(p)}$. First we check the case $p\to2^+$. In this case, the singular limit profile turns into a regular one. We get the following.
\begin{proposition}[The case $p\to 2^+$]\label{prop:p2} Assume $p>2$. Then we have that for any $k\in \mathbb{N}$, $\delta_k(p)/\delta_{k-1}(p)\to0$ and  $a_k(p)\to2$ as $p\to 2$. In particular,  we get $\delta_k(p)\to0$ and $z_k^{(p)}(r)\to \tilde{z}_0(r)$ for any $r>0$ as $p\to 2$ where 
\begin{equation}\label{eq:tilz}
\tilde{z}_0(r)=\log{\frac{16}{(2+r^2)^2}}
\end{equation}
which satisfies
\[\begin{cases}
-\tilde{z}_0''-\frac1r \tilde{z}_0'=e^{\tilde{z}_0}\text{ in }(0,\infty),\\
\tilde{z}_0(\sqrt{2})=0,\ \tilde{z}_0'(\sqrt{2})=-\sqrt{2},
\end{cases}
\]
and 
\[
\int_0^\infty e^{\tilde{z}_0}rdr=4.
\] 
\end{proposition}
Next, we consider the case $p\to \infty$. In this case, $z_k^{(p)}$ converges to another singular profile. To see this, we put $\hat{a}_0=2$ and $\hat{c}_0=1$.  Then, for all $k\in \mathbb{N}$, we define constants $\hat{a}_k\in(0,2)$ and $\hat{c}_k\in(0,\hat{c}_{k-1})$ by the relations,
\begin{equation}\label{eq:*1}
\frac{2}{2+\hat{a}_{k-1}}\log{\frac{\hat{c}_{k-1}}{\hat{c}_k}}-1+\frac{\hat{c}_k}{\hat{c}_{k-1}}=0
\end{equation}
and
\begin{equation}\label{eq:*2}
\hat{a}_k=2-\frac{\hat{c}_k}{\hat{c}_{k-1}}(2+\hat{a}_{k-1}).
\end{equation}
These numbers are well-defined. See Lemma \ref{lem:p1} below. Then we have the following. 
\begin{proposition}[The case $p\to \infty$]\label{prop:p*}
Suppose $p>2$. Then we get $\delta_k(p)^{p-1}\to \hat{c}_{k}$ and $a_{k}(p)\to \hat{a}_{k}$ as $p\to \infty$. Particularly, we obtain $\delta_k(p)\to1$ and  $z_k^{(p)}(r)\to \hat{z}_k(r)$ for all $r>0$ as $p\to \infty$ where  $\hat{z}_k$ is defined by 
\begin{equation}\label{hatz}
\hat{z}_k(r)=\log{\frac{2\hat{a}_k^2 \hat{b}_k}{r^{2-\hat{a}_k}(1+\hat{b}_kr^{\hat{a}_k})^2}}
\end{equation}
with $\hat{b}_k=(\sqrt{2}/\hat{a}_k)^{\hat{a}_k}$ and satisfies
\[
\begin{cases}
-\hat{z}_k''-\frac1r \hat{z}_k'=e^{\hat{z}_k}\text{ in }(0,\infty),\\
\hat{z}_k(\hat{a}_k/\sqrt{2})=0,\ (\hat{a}_k/\sqrt{2})\hat{z}_k'(\hat{a}_k/\sqrt{2})=2,
\end{cases}
\]
and
\[
\int_0^\infty e^{\hat{z}_k}rdr=2\hat{a}_k.
\]
Moreover, we have that $\hat{a}_k$ is strictly decreasing with respect to $k\in \mathbb{N}$, $\hat{a}_k\to0$ as $k\to \infty$, and  $\sum_{k=0}^\infty \hat{a}_k=\infty$. Especially, we deduce  
\[
\hat{z}_k(r)-\log{\frac{\hat{a}_k^2}{2}}\to 2\log{\frac1r}
\]
for any $r>0$ as $k\to \infty$.
\end{proposition}
Noting the previous propositions, in each case of $p_n\to 2^+$ and $p_n\to \infty$ in \eqref{eq:q2}, we can expect the existence of an energy-unbounded sequence of blow-up solutions  which holds infinite sequence of bubbles characterized by the regular profile $\tilde{z}_0$ and by the sequence $(\hat{z}_k)$ of singular ones respectively. We remark that, in the former case, since $\delta_k(p_n)/\delta_{k-1}(p_n)\to0$ as $n\to \infty$ for every $k\in \mathbb{N}$, each bubble tend to be relatively away from the previous one while in the latter case, the fact $\delta_k(p_n)\to1$ suggests that all the bubbles tend to gather around the origin.   Let us finally describe  such a blow-up behavior of $(v_n)$.  We first give the statement for the case $p\to 2^+$. 
\begin{corollary}\label{thm4} 
Assume that for all $p>2$ which is sufficiently closed to $2$, $h_p$ satisfies (H1) and there exists a value $\mu_{p}>0$ such that \eqref{p} admits a solution $(\la,u)$ with $u(0)=\mu$ for all $\mu\ge\mu_{p}$. Let $(p_n)\subset (2,3)$ be any sequence of values such that $p_n\to2$ as $n\to \infty$. Then, for all large $n\in \mathbb{N}$, there exists a sequence $\{(\la_n,\mu_n,v_n)\}$ of solutions of  \eqref{eq:q2} such that $\mu_n\to \infty$ as $n\to \infty$ and for all $k\in \mathbb{N}\cup\{0\}$, there exists a sequence $(\tilde{r}_{k,n})\subset (0,1)$ such that $\tilde{r}_{k,n}\to0$, $v_n(\tilde{r}_{k,n})\to \infty$, $v_n(\tilde{r}_{k,n})/(\delta_k(p_n)\mu_n)\to 1$, 
and further, if we put sequences $(\tilde{\ga}_{k,n})$ of positive values  and $(\tilde{z}_{k,n})$  of functions  so that if $k=0$,
\[
p_n\la_nv_n(0)^{p-1}f_n(v_n(0))\tilde{\ga}_{0,n}^2=1
\]
and
\[
\tilde{z}_{0,n}(r)=p_n v_n(0)^{p-1}(v_n(\tilde{\ga}_{0,n}r)-v_n(0))
\]
for all $r\in [0,1/\tilde{\ga}_{0,n}]$ and $n\in \mathbb{N}$ and if $k\ge1$,
\[p_n\la_nv_n(\tilde{r}_{k,n} r)^{p-1}f_n(v_n(\tilde{r}_{k,n}))\tilde{\ga}_{k,n}^2=1
\]
and
\[
\tilde{z}_{k,n}(r)=p_n v_n(\tilde{r}_{k,n})^{p-1}(v_n(\tilde{\ga}_{k,n} r)-v_n(\tilde{r}_{k,n})) 
\]
for all $r\in [0,1/\tilde{\ga}_{k,n}]$  and $n\in \mathbb{N}$, then we have for all $k\in \mathbb{N}\cup\{0\}$ that  $\tilde{\ga}_{k,n}\to0$, $\tilde{z}_{0,n}\to z_0$ in $C^1_{\text{loc}}([0,\infty))$, and $\tilde{z}_{k,n}\to \tilde{z}_0$ in $C^2_{\text{loc}}((0,\infty))$ if $k\ge1$ where $\tilde{z}_0$ is the function defined by \eqref{eq:tilz}. Moreover, for all $k\in \mathbb{N}\cup\{0\}$, there exist sequences $(\tilde{\bar{\rho}}_{k,n}),(\tilde{\rho}_{k,n})\subset (0,1)$ of values  such that $\tilde{\bar{\rho}}_{0,n}=0$, $v_n(\tilde{\rho}_{0,n})/\mu_n\to1$,  
\[ 
p_nv_n(0)^{p_n-1}\int_0^{\tilde{\rho}_{0,n}}\la_n f_n(v_n)rdr \to 
4,
\]
and if $k\ge1$, $\tilde{\rho}_{k-1,n}/\tilde{\bar{\rho}}_{k,n}\to 0$, $\tilde{\bar{\rho}}_{k,n}/\tilde{r}_{k,n}\to0$,  $\tilde{r}_{k,n}/\tilde{\rho}_{k,n}\to 0$, $v_n(\tilde{\rho}_{k,n})/(\delta_k(p_n)\mu_n)\to1$, $v_n(\tilde{\bar{\rho}}_{k,n})/(\delta_k(p_n)\mu_n)\to1$, 
\[
p_nv_n(0)^{p_n-1}\int_{\tilde{\rho}_{k-1,n}}^{\tilde{\bar{\rho}}_{k,n}}\la_n f_n(v_n)rdr\to 0,
\]
and
\[
p_nv_n(\tilde{r}_{k,n})^{p_n-1}\int_{\tilde{\bar{\rho}}_{k,n}}^{\tilde{\rho}_{k,n}}\la_n f_n(v_n)rdr\to 4
\]
as $n\to \infty$.  Furthermore,  we obtain
\[
\lim_{n\to \infty} p_n\int_0^1\la_n v_n^{p_n-1}f_n(v_n)rdr=\infty
\]
and
\[
\lim_{n\to \infty}\frac{\log{\frac1{\la_n}}}{v_n(0)^{p_n-1}}=0.
\]
Finally, we get that for all $k\in \mathbb{N}\cup\{0\}$, 
\[
v_n(\tilde{r}_{k,n})=\left\{(2+o(1))\log{\frac1{\tilde{r}_{k,n}}}\right\}^{\frac1{p_n}}
\]
and there exists a sequence $(\tilde{r}_{k,n}^*)\subset (\tilde{r}_{k,n},\tilde{r}_{k+1,n})$ such that 
\[
v_n(\tilde{r}_{k,n}^*)=\left(o\left(1\right)\log{\frac1{\tilde{r}_{k,n}^*}}\right)^{\frac1{p_n}}
\]
as $n\to \infty$.
\end{corollary} 
We remark that, as noted in the first part of Section \ref{subsec:ce}, the conditions on $h_p$ in the present corollary admit the typical examples. Moreover, we notice some different behaviors from those  in Theorem \ref{thm30}. In fact, recalling $\delta_k(p)/\delta_{k-1}(p)\to 0$ as $p\to 2^+$, we get $v_n(\tilde{r}_{k,n})/v_n(\tilde{r}_{k-1,n})\to0$ as $n\to \infty$ for all $k\in \mathbb{N}$  which implies that each concentrating bubble is located relatively farther away from the previous one. Furthermore, the limit profile and energy of each bubble is given by the regular function $\tilde{z}_0$ and $\int_0^\infty e^{\tilde{z}_0}rdr$ respectively as expected from Proposition \ref{prop:p2}. Hence such behaviors are  closer to those observed in the critical case $p=2$ by \cite{D} while the number and total energy of bubbles diverge to infinity in our case which is the crucial difference.   In addition, the final asymptotic formula in the previous corollary comes from \eqref{bo2} and the fact $\beta_k^*(p)\to 0$ as $p\to 2^+$. As a result, the last two formulas  suggest that the amplitude of each oscillation of $v_n$ gets larger and larger as $p_n\to2^+$. 

Lastly, we give the result for the case $p\to \infty$.
\begin{corollary}\label{cor5} 
We suppose that for all $p>2$ which is sufficiently large, $h_p$ verifies (H1) and there exists a value $\mu_{p}>0$ such that \eqref{p} permits  a solution $(\la,u)$ with $u(0)=\mu$ for all $\mu\ge\mu_{p}$. Take any sequence  $(p_n)\subset (2,\infty)$ of values  such that $p_n\to\infty$ as $n\to \infty$. Then, for all large $n\in \mathbb{N}$, there exists a sequence $\{(\la_n,\mu_n,v_n)\}$ of solutions of  \eqref{eq:q2} such that $\mu_n\to \infty$ as $n\to \infty$ and for all $k\in \mathbb{N}\cup\{0\}$, there exists a sequence $(\hat{r}_{k,n})\subset (0,1)$ such that $\hat{r}_{k,n}\to0$, $v_n(\hat{r}_{k,n})\to \infty$, $v_n(\hat{r}_{k,n})/\mu_n\to 1$, 
and further, setting sequences $(\hat{\ga}_{k,n})$  of positive numbers  and $(\hat{z}_{k,n})$ of functions   so that if $k=0$,
\[
p_n\la_nv_n(0)^{p-1}f_n(v_n(0))\hat{\ga}_{0,n}^2=1
\]
and
\[
\hat{z}_{0,n}(r)=p_n v_n(0)^{p-1}(v_n(\hat{\ga}_{0,n}r)-v_n(0))
\]
for all $r\in [0,1/\hat{\ga}_{0,n}]$ and $n\in\mathbb{N}$ and if $k\ge1$,
\[
p_n\la_nv_n(\hat{r}_{k,n} r)^{p-1}f_n(v_n(\hat{r}_{k,n}))\hat{\ga}_{k,n}^2=1
\]
and
\[
\hat{z}_{k,n}(r)=p_n v_n(\hat{r}_{k,n})^{p-1}(v_n(\hat{\ga}_{k,n} r)-v_n(\hat{r}_{k,n})) 
\]
for all $r\in [0,1/\hat{\ga}_{k,n}]$ and $n\in\mathbb{N}$, then we obtain for all $k\in \mathbb{N}\cup\{0\}$ that  $\hat{\ga}_{k,n}\to0$, $\hat{z}_{0,n}\to z_0$ in $C^1_{\text{loc}}([0,\infty))$, and $\hat{z}_{k,n}\to \hat{z}_k$ in $C^2_{\text{loc}}((0,\infty))$ if $k\ge1$ where $\hat{z}_k$ is  defined by \eqref{hatz}. Furthermore, for all $k\in \mathbb{N}\cup\{0\}$, there exist sequences $(\hat{\bar{\rho}}_{k,n}),(\hat{\rho}_{k,n})\subset (0,1)$  of values such that $\hat{\bar{\rho}}_{0,n}=0$, $v_n(\hat{\rho}_{0,n})/\mu_n\to1$,  
\[ 
p_nv_n(0)^{p_n-1}\int_0^{\hat{\rho}_{0,n}}\la_n f_n(v_n)rdr \to 
4,
\]
and if $k\ge1$, $\hat{\rho}_{k-1,n}/\hat{\bar{\rho}}_{k,n}\to 0$, $\hat{\bar{\rho}}_{k,n}/\hat{r}_{k,n}\to0$,  $\hat{r}_{k,n}/\hat{\rho}_{k,n}\to 0$, $v_n(\hat{\rho}_{k,n})/\mu_n\to1$, $v_n(\hat{\bar{\rho}}_{k,n})/\mu_n\to1$, 
\[
p_nv_n(0)^{p_n-1}\int_{\hat{\rho}_{k-1,n}}^{\hat{\bar{\rho}}_{k,n}}\la_n f_n(v_n)rdr\to 0,
\]
and
\[
p_nv_n(\hat{r}_{k,n})^{p_n-1}\int_{\hat{\bar{\rho}}_{k,n}}^{\hat{\rho}_{k,n}}\la_n f_n(v_n)rdr\to 2\hat{a}_k
\]
as $n\to \infty$. In addition, we get
\[
\lim_{n\to \infty} p_n\int_0^1\la_n v_n^{p_n-1}f_n(v_n)rdr=\infty
\]
and
\[
\lim_{n\to \infty}\frac{\log{\frac1{\la_n}}}{v_n(0)^{p_n}}=0.
\]
Lastly, we obtain that  for all $k\in \mathbb{N}\cup\{0\}$, 
\[
v_n(\hat{r}_{k,n})=\left\{(2+o(1))\log{\frac1{\hat{r}_{k,n}}}\right\}^{\frac1{p_n}}
\]
and there exists a sequence $(\hat{r}_{k,n}^*)\subset (\hat{r}_{k,n},\hat{r}_{k+1,n})$ such that 
\[
v_n(\hat{r}_{k,n}^*)= \left\{\left(\hat{\beta}_k^*+o(1)\right)\log{\frac1{\hat{r}_{k,n}^*}}\right\}^{\frac1{p_n}}
\]
as $n\to \infty$ where $\hat{\beta}_k^*=(2+\hat{a}_k)e^{-\hat{a}_k/2}<2$.
\end{corollary} 
A different point from the previous case is that $v_n(\hat{r}_{k,n})/\mu_n\to1$ as $n\to \infty$ for all $k\in \mathbb{N}\cup\{0\}$ by the fact $\delta_k(p)\to1$ as $p\to \infty$ in Proposition \ref{prop:p*}. Hence each bubble appears relatively closer to the top of the graph. Moreover, the last formula is a consequence of the fact that for each $k\in \mathbb{N}$, $\beta_k^*(p)\to \hat{\beta}_k^*$ as $p\to\infty$. Notice also that  $\hat{\beta}_k^*\to 2$ as $k\to \infty$. Hence in the present limit case, the oscillation behavior is similar to that in the case of fixed $p>2$ while the fact $\lim_{p\to \infty}\delta_k(p)= 1$ yields that each oscillation occurs much closer to the origin.    

In the next sections,  we prove the propositions and corollaries above. 
\subsection{Proofs}
Let us first show Proposition \ref{prop:p2}. We have the following. 
\begin{lemma}\label{lem:ad} Assume $p>2$.  Then we have that for all $k\in \mathbb{N}$, $\delta_k(p)/\delta_{k-1}(p)\to 0$ and $a_k(p)\to2$ as $p\to2$.
\end{lemma}
\begin{proof}  For each $k\in \mathbb{N}$, put $d_k(p)=\delta_k(p)/\delta_{k-1}(p)$. We argue by induction. First, we consider the case $k=1$. From \eqref{eq:del1}, we see
\begin{equation}\label{eq:i1}
\frac p2(1-d_1(p))-1+d_1^p(p)=0.
\end{equation}
It follows that $d_1(p)<(1/2)^{1/(p-1)}$. We claim $d_1(p)\to 0$ as $p\to2$. If not, there exist a sequence $(p_n)\subset (2,3)$ and a constant $d_{1,*}\in(0,1/2]$ such that  $p_n\to2$ and $d_1(p_n)\to d_{1,*}$ as $n\to \infty$. Then \eqref{eq:i1} yields hat
\[
-d_{1,*}+d_{1,*}^2=0.
\]
This is impossible. As a consequence, we get by \eqref{eq:del2} that
\[
a_1(p)=2-4d_1(p)^{p-1} \to2
\]
as $p\to2$. This is the proof for the case $k=1$. We assume that the assertions are true for a number $k\ge 1$. From \eqref{eq:del1}, we see
\begin{equation}\label{eq:i2}
\frac {2p}{2+a_k(p)}(1-d_{k+1}(p))-1+d_{k+1}(p)^p=0.
\end{equation}
It follows that
\begin{equation}\label{eq:i3}
d_{k+1}(p)<\left(\frac{2}{2+a_k(p)}\right)^{\frac1{p-1}}.
\end{equation}
We claim $d_{k+1}(p)\to0$ as $p\to2$. If not, noting the assumption that $a_k(p)\to 2$ as $p\to2$, we find a sequence $(p_n)\subset (2,3)$ and a value $d_{k+1,*}\in (0,1/2]$ such that $p_n\to2$ and $d_{k+1}(p_n)\to d_{k+1,*}$ as $n\to \infty$.  Moreover, \eqref{eq:i2} yields that
\[
-d_{k+1,*}+d_{k+1,*}^2=0
\]
which  is again impossible. Then \eqref{eq:del2} shows $a_{k+1}(p)\to2$ as $p\to 2$. This proves the desired assertion. We finish the proof. 
\end{proof}
Then we prove Proposition \ref{prop:p2}.
\begin{proof}[Proof of Proposition \ref{prop:p2}] Using Lemma \ref{lem:ad} and recalling the definition of $z_k$ in Theorem \ref{thm30}, we readily get the proof. 
\end{proof}
Next we shall show Proposition \ref{prop:p*}. We begin with the following. 
\begin{lemma}\label{lem:p1} Suppose $p>2$, $\hat{a}_0=2$, and $\hat{c}_0=1$. Then for all $k\in \mathbb{N}$, $\hat{a}_k\in(0,2)$ and $\hat{c}_k\in(0,\hat{c}_{k-1})$ are well-defined by the relations \eqref{eq:*1} and \eqref{eq:*2}. 
\end{lemma}
\begin{proof} We argue by induction. For $k=1$, we clearly see that there exists a unique positive value $x_1$ such that $x_1<1$ and  
\[
0=\frac{1}{2}\log{\frac{1}{x_1}}-1+x_1.
\]
Moreover, we easily see that $x_1<1/2$. Hence, the constants $\hat{c}_1\in(0,1)$ and $\hat{a}_1\in(0,2)$ are well-defined by \eqref{eq:*1} and \eqref{eq:*2} with $k=1$.  Next, we suppose $\hat{a}_k\in(0,2)$ and $\hat{c}_k\in(0,\hat{c}_{k-1})$ are well-defined by \eqref{eq:*1} and \eqref{eq:*2} for some $k\ge1$. Then similarly we easily check that there exists a unique positive constant $x_{k+1}$ such that $x_{k+1}<1$ and 
\[
0=\frac{2}{2+\hat{a}_k}\log{\frac 1{x_{k+1}}}-1+x_{k+1}. 
\]
It obviously follows that  $x_{k+1}<2/(2+\hat{a}_k)$. Putting $\hat{c}_{k+1}=\hat{c}_kx_{k+1}$, the constant $\hat{c}_{k+1}\in(0,\hat{c}_k)$ satisfying \eqref{eq:*1} is uniquely determined. Then   $\hat{a}_{k+1}\in (0,2)$ is uniquely defined by \eqref{eq:*2}. 
 This finishes the proof. 
\end{proof}
We use the next lemma to check the convergence of $a_k(p)$, $\delta_k(p)$, and $\delta_k^{p-1}(p)$.
\begin{lemma}\label{lem:pre} Assume $p>2$. Let $A_p,A_*\in(2,4]$ be constants such that $A_p\to A_*\in(2,4]$ as $p\to \infty$ and put a function
\[
g_p(x)=\frac{2p}{A_p}(1-x)-1+x^p
\]
for all $x\in[0,1]$. Then there exists a unique value  $x_p$  such that $0<x_p<1$ and $g_p(x_p)=0$. Moreover, there exists a number  $x_*\in(0,2/A_*)$ such that $x_p\to 1$, $x_p^{p-1}\to x_*$ as $p\to \infty$, and 
\[
\frac2A_*\log{\frac1{x_*}}-1+x_*=0.
\] 
 \end{lemma}
\begin{proof}
First, obviously  there exists the unique value $x_p$ such that $x_p\in(0,1)$ and $g_p(x_p)=0$. 
Moreover, noting that $g_p(x)>-(2p/A_p)x+2p/A_p-1$ for all $x\in(0,1)$, we get
\[
1-\frac {A_p}{2p}<x_p<\left(\frac2A_p\right)^{\frac1{p-1}}. 
\]
It follows that 
\begin{equation}\label{eq:pre1}
\lim_{p\to \infty}x_p= 1\text{ and }\  e^{-A_*/2}\le \liminf_{p\to \infty}x_p^{p-1}\le\limsup_{p\to \infty}x_p^{p-1}\le 2/A_*<1.
\end{equation}
 To finish the proof, we set $\tilde{g}_p(t)=g_p(t^{1/(p-1)})$, that is,
\[
\tilde{g}_p(t)=\frac{2p}{A_p}\left(1-t^{\frac1{p-1}}\right)-1+t^{\frac{p}{p-1}}
\]
for all $t\in [0,1]$. Setting $t_p=x_p^{p-1}$, we have that $t_p$ is the unique constant  such that   $t_p\in(0,1)$ and $\tilde{g}_p(t_p)=0$. Moreover, we obtain that $\tilde{g}_p(t)\to \tilde{g}_*(t)$ as $p\to \infty$ locally uniformly in $(0,1]$ where 
\[
\tilde{g}_*(t)=\frac2{A_*}\log{\frac1t}-1+t.
\]
Hence noting \eqref{eq:pre1}, we confirm that $t_p\to t_*$ as $p\to \infty$ where $t_*$ is the unique value such that $t_*\in(0,1)$ and $\tilde{g}_*(t_*)=0$. We readily check that  $t_*\in(0,2/A_*)$. This completes the proof. 
 \end{proof}
Then we get the desired convergence results. 
\begin{lemma}\label{lem:ad2} Assume $p>2$.  We have that for each $k\in \mathbb{N}$, $\delta_{k}(p)^{p-1}\to \hat{c}_k$ and $a_k(p)\to \hat{a}_k$ as $p\to \infty$.  
\end{lemma}
\begin{proof}
We again put 
 $d_k=\delta_k/\delta_{k-1}$ for all $k\in \mathbb{N}$. Let us give the proof  by induction. If $k=1$,  \eqref{eq:del1} and  \eqref{eq:del2} become 
\begin{equation}\label{sss1}
\frac{2p}{4}\left(1-d_1(p)\right)-1+d_1(p)^p=0
\end{equation}
and
\begin{equation}\label{sss2}
a_1(p)=2-4d_1(p)^{p-1}.
\end{equation} 
Then using Lemma \ref{lem:pre}, we get that  $d_1(p)\to 1$ and  $d_1(p)^{p-1}\to \hat{c}_{1}$ as $p\to \infty$. 
 Moreover, it follows from \eqref{sss2} and \eqref{eq:*2} that $a_{1}(p)\to \hat{a}_1$ as $p\to \infty$.
This proves the case $k=1$. We assume the assertions are true for some $k\ge1$. From \eqref{eq:del1} and \eqref{eq:del2}, we see 
\begin{equation}\label{eq:sts1}
\frac{2p}{2+a_{k}(p)}\left(1-d_{k+1}(p)\right)-1+d_{k+1}(p)^p=0
\end{equation}
and
\begin{equation}\label{eq:sts2}
a_{k+1}(p)=2-d_{k+1}(p)^{p-1}(2+a_{k}(p)).
\end{equation}
Then again using Lemma \ref{lem:pre}, we obtain from \eqref{eq:sts1} that there exists a value $\hat{c}_{k+1}^*\in(0,2/(2+\hat{a}_k))$ such that $d_{k+1}(p)\to 1$, $d_{k+1}^{p-1}\to \hat{c}_{k+1}^*$ as $p\to \infty$, and 
\[
\frac2{2+\hat{a}_k}\log{\frac1{\hat{c}_{k+1}^*}}-1+\hat{c}_{k+1}^*=0.
\]
It follows that  $\delta_{k+1}(p)^{p-1}=d_{k+1}^{p-1}\delta_k(p)^{p-1}\to \hat{c}_{k+1}^*\hat{c}_k<\hat{c}_k$ as $p\to \infty$. Using the previous formula and \eqref{eq:*1}, we get $\hat{c}_{k+1}^*\hat{c}_k= \hat{c}_{k+1}$. Then \eqref{eq:sts2} completes the proof.   We finish the proof. 
\end{proof}
The next lemma shows the behavior of $\hat{a}_k$ and $\hat{c}_k$ as $k\to \infty$.
\begin{lemma}\label{lem:p2} We have $\hat{a}_{k-1}>\hat{a}_k$ and $\hat{c}_k/\hat{c}_{k-1}<\hat{c}_{k+1}/\hat{c}_k$ for all $k\in \mathbb{N}$ and $\hat{a}_k\to 0$ and $\hat{c}_k/\hat{c}_{k-1}\to1$ as $k\to \infty$. 
\end{lemma}
\begin{proof} Fix any $k\in \mathbb{N}$. From \eqref{eq:*2}, we get
\begin{equation}\label{uu2}
\frac{\hat{c}_k}{\hat{c}_{k-1}}=\frac{2-\hat{a}_k}{2+\hat{a}_{k-1}}.
\end{equation}
Substituting this into \eqref{eq:*1}, we have  
\begin{equation}\label{uu1}
\frac{2}{2+\hat{a}_{k-1}}\log{\frac{2+\hat{a}_{k-1}}{2-\hat{a}_k}}-1+\frac{2-\hat{a}_k}{2+\hat{a}_{k-1}}=0.
\end{equation}
Then the proof is done similarly to Lemma \ref{lem:bf2}. For readers' convenience, we show the proof. This time, we put
\[
g(x)=\frac{2}{2+\hat{a}_{k-1}}\log{\frac{2+\hat{a}_{k-1}}{2-x}}-1+\frac{2-x}{2+\hat{a}_{k-1}}
\]  
for all $x\in [0,2)$. Then we easily get $g(0)<0$, $g(a_k)=0$, and $g'(x)>0$ for all $x\in (0,2]$. Moreover, if we put
\[
\tilde{g}(x)=\frac{2}{2+x}\log{\frac{2+x}{2-x}}-1+\frac{2-x}{2+x}
\]
for all $x\in [0,2)$, an elementary argument shows that  $\tilde{g}(x)>0$ for all $x\in(0,2)$. In particular, $g(\hat{a}_{k-1})=\tilde{g}(\hat{a}_{k-1})>0$. Hence the monotonicity of $g$ implies  $\hat{a}_k<\hat{a}_{k-1}$. Then \eqref{uu2} ensures that $\hat{c}_k/\hat{c}_{k-1}<\hat{c}_{k+1}/\hat{c}_{k}$. This proves the former assertions.  It follows that there exist  constants $\hat{a}_*\in[0,2)$ and $\hat{c}_{*}\in(0,1]$ such that $\hat{a}_k\to \hat{a}_*$ and $\hat{c}_k/\hat{c}_{k-1}\to \hat{c}_*$ as $k\to \infty$. Then \eqref{uu2} and \eqref{uu1} imply
\[
\hat{c}_*=\frac{2-\hat{a}_*}{2+\hat{a}_*}
\] 
and 
\[
\frac{2}{2+\hat{a}_*}\log{\frac{2+\hat{a}_*}{2-\hat{a}_*}}-1+\frac{2-\hat{a}_*}{2+\hat{a}_*}=0.
\]
Recalling that $\tilde{g}(x)>0$ for all $x\in(0,2)$, we get  $\hat{a}_*=0$ and thus, $\hat{c}_{*}=1$. We finish the proof. 
\end{proof}
We also prove that the infinite series of $(\hat{a}_k)$ diverges. 
\begin{lemma}\label{lem:p3}
We get $\sum_{k=0}^\infty \hat{a}_k=\infty$.
\end{lemma}
\begin{proof} From the previous lemma, we have a decreasing sequence $(\e_k)\subset (0,1)$ such that $\hat{c}_k/\hat{c}_{k-1}=1-\e_k$ for all $k\in \mathbb{N}$ and $\e_k\to0$ as $k\to \infty$. Then from \eqref{eq:*1}, we get
\begin{equation}\label{v1}
\hat{a}_{k-1}= -2+\frac{2}{\e_k}\log{\frac1{1-\e_k}}=\e_{k}+\frac23 \e_{k}^2+O(\e_{k}^3).
\end{equation}
Substituting this into \eqref{eq:*2}, we see
\begin{equation}\label{v2}
\hat{a}_k=\e_k+\frac13 \e_k^2+O(\e_k^3).
\end{equation}
Replacing $k$ with $k+1$ in \eqref{v1}, we get
\[
\hat{a}_k=
\e_{k+1}+\frac23 \e_{k+1}^2+O(\e_{k+1}^3).
\]
From these two formulas, we obtain
\[
\e_{k+1}=\e_k+\frac13 \e_k^2-\frac23 \e_{k+1}^2+O(\e_k^3)
\]
since $\e_{k+1}<\e_k$. Then we get $\e_{k+1}=\e_k+O(\e_k^2)$ and thus,  
\[
\begin{split}
\e_{k+1}&=\e_k+\frac13 \e_k^2-\frac23 (\e_k+O(\e_k^2))^2+O(\e_k^3)
=\e_k-\frac13 \e_k^2+O(\e_k^3).
\end{split}
\]
Then from Lemma \ref{pre0}, there exist numbers $k_0\in \mathbb{N}$ and $\beta>0$ such that $\e_k\ge \beta/k$ for all $k\ge k_0$. Hence, it follows from  \eqref{v2} that $\hat{a}_k\ge \beta/k$ for all $k\ge k_0$ by taking $k_0$ larger if necessary. Consequently, we get for all $k\ge k_0$ that
\[
\sum_{i=0}^{k} \hat{a}_i\ge \sum_{i=k_0}^{k} \hat{a}_i\ge \beta\sum_{i=k_0}^{k} \frac1i\to \infty
\]
as $k\to \infty$. This completes the proof. 
\end{proof}
Now we show Proposition \ref{prop:p*}.
\begin{proof}[Proof of Proposition \ref{prop:p*}]
The proof readily follows from  Lemmas \ref{lem:ad2}, \ref{lem:p2}, \ref{lem:p3}, and the definitions of $z_k^{(p)}$ and $\hat{z}_k$. We complete the proof.  
\end{proof}

Let us prove Corollary \ref{thm4}.
\begin{proof}[Proof of Corollary \ref{thm4}] Let $(p_m)_{m\in \mathbb{N}}\subset (2,3)$ be any sequence of values such that $p_m\to2$ as $m\to \infty$. From the assumption,  for each large $m\in \mathbb{N}$, there exists a sequence $\{(\la_{m,n},\mu_{m,n}u_{m,n})\}_{n\in \mathbb{N}}$ such that $\{(\la_n,\mu_n,u_n)\}=\{(\la_{m,n},\mu_{m,n},u_{m,n})\}$ is a sequence of solutions of  \eqref{q} with $p=p_m$ and $\mu_{m,n}\to \infty$ as $n\to \infty$. Fix such a large $m\in \mathbb{N}$. The idea of the rest of the proof is simple. By Theorems \ref{thm1} and \ref{thm30}, we choose a sufficiently large number $n_m\in \mathbb{N}$ so that the first and the next $m$ bubbles on $u_{n_m}$ are well-approximated by $z_0$ and $\tilde{z}_0$ respectively. Then $\{(\la_m,\mu_m,v_m)\}:=\{(\la_{n_m},\mu_{n_m},u_{n_m})\}$ is the desired sequence of solutions of \eqref{eq:q2} which finishes the proof. For the sake of the completeness, we show a more detailed  proof. To this end, for all $k\in \mathbb{N}\cup\{0\}$, we write $a_{m,k}=a_k(p_m)$, $b_{m,k}=b_k(p_m)$, $\delta_{m,k}=\delta_k(p_m)$, and $\beta_{m,k}^*=\beta_k^*(p_m)$. Then from Theorems \ref{thm1} and \ref{thm30}, up to a subsequence,  for each $k\in \mathbb{N}\cup\{0\}$, there exists a sequence  $(r_{m,k,n})\subset (0,1)$ of values  such that $r_{m,k,n}\to0$ and $u_{m,n}(r_{m,k,n})/\mu_{m,n}\to \delta_{m,k}$ as $n\to \infty$ and if we define sequences $(\ga_{k,n})$ of values  and  $(z_{k,n})$ of functions   so that
\[
p_m\mu_{m,n}^{p-1}\la_{m,n}f_n(\mu_{m,n})\ga_{m,0,n}^2=1
\]
and
\[
z_{m,0,n}(r)=p_m\mu_{m,n}^{p-1}(u_{m,n}(\ga_{m,0,n}r)-\mu_{m,n})
\]
for all $r\in [0,1/\ga_{m,0,n}],$ and if $k\ge1$,
\[
p_mu_{m,n}(r_{m,k,n})^{p-1}\la_{m,n}f_n(u_{m,n}(r_{m,k,n}))\ga_{m,k,n}^2=1
\]
and
\[
z_{m,k,n}(r)=p_mu_{m,n}(r_{m,k,n})^{p-1}(u_{m,n}(\ga_{m,k,n}r)-u_{m,n}(r_{m,k,n}))
\]
for all $r\in [0,1/\ga_{m,k,n}]$, then there exist sequences $(\bar{\rho}_{m,k,n}),(\rho_{m,k,n})\subset (0,1)$ such that $\bar{\rho}_{m,0,n}=0$ for all $n\in \mathbb{N}$, $u_{m,n}(\rho_{m,0,n})/\mu_{m,n}\to 1$, $\rho_{m,0,n}/\ga_{m,0,n}\to \infty$, 
\[
\|z_{m,0,n}- z_0\|_{C^1([0,\rho_{m,0,n}/\ga_{m,0,n}])}\to0,
\]
\[
p_m\mu_{m,n}^{p_m-1}\int_{0}^{\rho_{m,0,n}}\la_{m,n}f_m(u_{m,n})rdr\to 4,
\]
and for all $k\in \mathbb{N}$, $\rho_{m,k-1,n}/\bar{\rho}_{m,k,n}\to0$, $\bar{\rho}_{m,k,n}/r_{m,k,n}\to0$, $r_{m,k,n}/\rho_{m,k,n}\to 0$, $u_{m,n}(\bar{\rho}_{m,k,n})/\mu_{m,n}\to \delta_{m,k}$, $u_{m,n}(\rho_{m,k,n})/\mu_{m,n}\to \delta_{m,k}$, 
 $\bar{\rho}_{m,k,n}/\ga_{m,k,n}\to0$, $\rho_{m,k,n}/\ga_{m,k,n}\to \infty$, 
\[
\|z_{m,k,n}-z_{m,k}\|_{C^2([\bar{\rho}_{m,k,n}/\ga_{m,k,n},\rho_{m,k,n}/\ga_{m,k,n}])}\to0
\]
where $z_{m,k}(r)=z_k^{(p_m)}(r)$,
\[
p_m\mu_{m,n}^{p_m-1}\int_{\rho_{m,k-1,n}}^{\bar{\rho}_{m,k,n}}\la_{m,n}f_m(u_{m,n})rdr\to 0,
\]
and
\[
p_mu_{m,n}(r_{m,k,n})^{p_m-1}\int_{\bar{\rho}_{m,k,n}}^{\rho_{m,k,n}}\la_{m,n}f_m(u_{m,n})rdr\to 2a_{m,k}
\]
as $n\to \infty$. Moreover, by Theorems \ref{thm3} and \ref{thm:osc}, we have 
\[
p_m\int_0^1\la_{m,n}u_{m,n}^{p_m-1} f_m(u_{m,n})rdr\to \infty,
\]
\[
\frac{\log{\frac1{\la_{m,n}}}}{\mu_{m,n}^{p_m}}\to0,
\]
and  for all $k\in \mathbb{N}\cup\{0\}$, there exists a sequence $(r_{m,k,n}^*)\subset (r_{m,k,n},r_{m,k+1,n})$ such that 
\[\frac{u_{m,n}(r_{m,k,n})^{p_m}}{\log{\frac1{r_{m,k,n}}}}\to 2\ \text{ and }\ \frac{u_{m,n}(r_{m,k,n}^*)^{p_m}}{\log{\frac1{r_{m,k,n}^*}}}\to \beta_{m,k}^*
\]
as $n\to \infty$. Then we find a number $n_m\in \mathbb{N}$ such that for all $k=0,1,\cdots,m$, 
  $r_{m,k,n_m}<1/m$, $u_{m,k,n_m}(r_{m,k,n_m})>m$, $|u_{m,n_m}(r_{m,k,n_m})/(\delta_{m,k}\mu_{m,n_m})- 1|<1/m$, $|u_{m,n}(\rho_{m,0,n_m})/\mu_{m,n_m}-1|<1/m$, $\rho_{m,0,n_m}/\ga_{m,0,n_m}>m$, 
\[
\|z_{m,0,n_m}- z_0\|_{C^1([0,\rho_{m,0,n_m}/\ga_{m,0,n_m}])}<1/m,
\]
\[
\left|p_m\mu_{m,n_m}^{p_m-1}\int_{0}^{\rho_{m,0,n_m}}\la_{m,n_m}f_m(u_{m,n_m})rdr- 4\right|<1/m,
\]
and for all $k=1,\cdots,m$, $\rho_{m,k-1,n_m}/\bar{\rho}_{m,k,n_m}<1/m$, $\bar{\rho}_{m,k,n_m}/r_{m,k,n_m}<1/m$, $r_{m,k,n_m}/\rho_{m,k,n_m}<1/m$,  $|u_{m,n}(\rho_{m,k,n_m})/(\delta_{m,k}\mu_{m,n_m})-1|<1/m$, $|u_{m,n}(\bar{\rho}_{m,k,n_m})/(\delta_{m,k}\mu_{m,n_m})-1|<1/m$,  $\bar{\rho}_{m,k,n_m}/\ga_{m,k,n_m}<1/m$, $\rho_{m,k,n_m}/\ga_{m,k,n_m}>m$, 
\[
\|z_{m,k,n_m}-z_{m,k}\|_{C^2([\bar{\rho}_{m,k,n_m}/\ga_{m,k,n_m},\rho_{m,k,n_m}/\ga_{m,k,n_m}])}<1/m, 
\]
\[
p_m\mu_{m,n_m}^{p_m-1}\int_{\rho_{m,k-1,n_m}}^{\bar{\rho}_{m,k,n_m}}\la_{m,n_m}f_m(u_{m,n_m})rdr<1/m,
\]
\[
\left|p_mu_{m,n_m}(r_{m,k,n_m})^{p_m-1}\int_{\bar{\rho}_{m,k,n_m}}^{\rho_{m,k,n_m}}\la_{m,n_m}f_m(u_{m,n_m})rdr- 2a_{m,k}\right|<1/m,
\]
\[
p_m\int_0^1\la_{m,n_m}u_{m,n_m}^{p_m-1} f_m(u_{m,n_m})rdr>m,\ 
\frac{\log{\frac1{\la_{m,n_m}}}}{\mu_{m,n_m}^{p_m}}<\frac1m,
\]
and  for all $k=0,1,\cdots,m$, 
\[
\left|\frac{u_{m,n}(r_{m,k,n})^{p_m}}{\log{\frac1{r_{m,k,n_m}}}}- 2\right|<\frac1m\ \text{ and }\ \left|\frac{u_{m,n_m}(r_{m,k,n_m}^*)^{p_m}}{\log{\frac1{r_{m,k,n_m}^*}}}- \beta_{m,k}^*\right|<\frac1m.
\]
Then put $v_m=u_{m,n_m}$ and $\la_m=\la_{m,n_m}$, and for all $k\in \mathbb{N}\cup\{0\}$, $\tilde{r}_{k,m}=r_{m,k,n_m}$, $\tilde{\bar{\rho}}_{k,m}=\bar{\rho}_{m,k,n_m}$, $\tilde{\rho}_{k,m}=\rho_{m,k,n_m}$, $\tilde{r}_{k,m}^*=r_{m,k,n_m}^*$ if $k=0,\cdots,m$ and $\tilde{r}_{k,m}=\tilde{\bar{\rho}}_{k,m}=\tilde{\rho}_{k,m}=\tilde{r}_{k,m}^*=1/2$ if $k\ge m+1$ and then, define $\tilde{\ga}_{k,n}$ and $\tilde{z}_{k,m}$ as in the corollary. 
 These are the desired sequence. To see this, it suffices to note Proposition \ref{prop:p2} and find that for each $k\in \mathbb{N}$, $a_{m,k}\to 2$, $\beta_{m,k}^*\to 0$, and for every compact set  $K\subset (0,\infty)$, 
\[
\begin{split}
&\|z_{m,k,n_m}-\tilde{z}_0\|_{C^2(K)}\le \|z_{m,k,n_m}-z_{m,k}\|_{C^2(K)}+ \|z_{m,k}-\tilde{z}_0\|_{C^2(K)}\to0
\end{split}
\]
as $m\to \infty$.  We finish the proof. 
\end{proof}
We finally show  Corollary \ref{cor5}.
\begin{proof}[Proof of Corollary \ref{cor5}] Noting the facts in Proposition \ref{prop:p*} and $\beta_k^*(p)\to \hat{\beta}_k^*$ as $p\to \infty$, the proof is done with the  same argument to that  in the previous corollary. We finish the proof. 
\end{proof}
\subsubsection*{Acknowledgement} The author is grateful to Norisuke Ioku for favorable discussions and a useful comment on Lemma \ref{lem:sg}.  He thanks the authors of \cite{FIRT}  for providing him with detailed information referred in Remarks \ref{rmk:B} and \ref{rmk:D} when he was preparing this paper.  He also thanks  Kenta Kumagai for information about his recent work and  stimulating discussions. This work is supported by JSPS KAKENHI Grant Numbers 17K14214 and 21K13813.  

\end{document}